\RequirePackage{fix-cm}
\documentclass{svjour3}                     % onecolumn (standard format)
\smartqed  % flush right qed marks, e.g. at end of proof
\usepackage{graphicx}
\usepackage{hyperref}
\usepackage{tikz-cd}
\usepackage{amsmath}
\usepackage{amssymb}
\numberwithin{equation}{section}
\numberwithin{theorem}{section}
\numberwithin{remark}{section}
\numberwithin{lemma}{section}
\numberwithin{corollary}{section}
\newcommand\restr[2]{{% we make the whole thing an ordinary symbol
		\left.\kern-\nulldelimiterspace % automatically resize the bar with \right
		#1 % the function
		\vphantom{\big|} % pretend it's a little taller at normal size
		\right|_{#2} % this is the delimiter
}}
\begin{document}

\title{Frequency Theorem and Inertial Manifolds for Neutral Delay Equations}
%\title{Frequency Theorem and Inertial Manifolds for Neutral Delay Equations}
%\subtitle{Do you have a subtitle?\\ If so, write it here}

\titlerunning{Frequency Theorem and Inertial Manifolds for Neutral Delay Equations}        % if too long for running head

\author{Mikhail Anikushin %etc.
}

%\authorrunning{Short form of author list} % if too long for running head

\institute{Mikhail Anikushin \at
              Department of
              Applied Cybernetics, Faculty of Mathematics and Mechanics,
              St. Petersburg University, 28 Universitetskiy prospekt, Peterhof, 198504, Russia \\
              \email{demolishka@gmail.com}           %  \\
%             \emph{Present address:} of F. Author  %  if needed
           %\and
           %S. Author \at
              %second address
}

\date{Received: date / Accepted: date}
% The correct dates will be entered by the editor

\maketitle

\begin{abstract}
We study the quadratic regulator problem for linear control systems in Hilbert spaces, where the cost functional is in some sense unbounded. Our motivation comes from delay equations with the feedback part containing discrete delays or, in other words, measurements given by $\delta$-functionals, which are unbounded in $L_{2}$. Working in an abstract context in which such (and many others, including parabolic boundary control problems) equations can be treated, we obtain a version of the Frequency Theorem. It guarantees the existence of a unique optimal process and shows that the optimal cost is given by a quadratic Lyapunov-like functional. In our adjacent works it is shown that such functionals can be used to construct inertial manifolds and allow to treat and extend many works in the field in a unified manner. Here we concentrate on applications to delay equations and especially mention the works of R.A.~Smith on developments of convergence theorems and the Poincar\'{e}-Bendixson theory; and also the works of Yu.A.~Ryabov, R.D.~Driver and C.~Chicone on inertial manifolds for equations with small delays and their recent generalization for equations of neutral type given by S.~Chen and J.~Shen.

% \PACS{PACS code1 \and PACS code2 \and more}
\end{abstract}

%\newpage
\tableofcontents
% !TeX spellcheck = en_US
\section{Introduction}
We start by introducing precise definitions and then give discussions on the topic. In what follows, we assume that all the vector spaces are complex, unless otherwise is specified. Throughout this paper, $\mathcal{L}(\mathbb{E};\mathbb{F})$ denotes the space of bounded linear operators between Banach spaces $\mathbb{E}$ and $\mathbb{F}$ with the operator norm denoted by $\|\cdot\|_{\mathcal{L}(\mathbb{E};\mathbb{F})}$. If $\mathbb{E}=\mathbb{F}$, we usually write $\mathcal{L}(\mathbb{E})$. For the norm in a Banach space $\mathbb{E}$ we use the notation $\|\cdot\|_{\mathbb{E}}$. If $\mathbb{E}$ is a Hilbert space we will sometimes (mainly for $\mathbb{E}$ being the space of values of certain functions) use $|\cdot|_{\mathbb{E}}$ to denote the norm.

Let $G(t)$, where $t \geq 0$, be a $C_{0}$-semigroup acting in a Hilbert space $\mathbb{H}$ and let $A\colon \mathcal{D}(A) \subset \mathbb{H} \to \mathbb{H}$ be its generator with the domain $\mathcal{D}(A)$. For the theory of $C_{0}$-semigroups, we refer to the monographs of K.-J.~Engel and R.~Nagel \cite{EngelNagel2000} or S.G.~Krein \cite{Krein1971}. Let $\Xi$ be another Hilbert space and $B \in \mathcal{L}(\Xi;\mathbb{H})$. We will consider the linear inhomogeneous equation associated with the pair $(A,B)$ as
\begin{equation}
\label{EQ: ControlSystem}
\dot{v}(t) = Av(t) + B\xi(t)
\end{equation}
and call it a \textit{control system}. Usually $\Xi$ is referred as the \textit{control space} and $\Xi$-valued functions of time like $\xi(\cdot)$ are called \textit{control functions}. The operator $B$ is called a \textit{control operator}.

It is well-known that for every $T>0$, $v_{0} \in \mathbb{H}$ and $\xi(\cdot) \in L_{2}(0,T;\Xi)$ there exists a unique mild solution $v(t)=v(t; v_{0},\xi)$, where $t \in [0,T]$, to \eqref{EQ: ControlSystem} such that $v(0)=v_{0}$. This solution is a continuous $\mathbb{H}$-valued function and for $t \in [0,T]$ it is given by the Cauchy formula as
\begin{equation}
\label{EQ: MildSolution}
v(t) = G(t)v_{0} + \int_{0}^{t}G(t-s)B\xi(s)ds.
\end{equation}
Below we will require some additional regularity properties.

Throughout the paper, we always endow $\mathcal{D}(A)$ with the graph norm. Let us assume that there are a Banach space $\mathbb{E}_{0}$ and a Hilbert space $\mathbb{W}$ which are continuously embedded into $\mathbb{H}$ and identified with their images such that all the embeddings
\begin{equation}
	\label{EQ: AuxiliarySpacesEmbeddings}
	\mathcal{D}(A) \subset \mathbb{E}_{0} \subset \mathbb{H} \text{ and } \mathcal{D}(A) \subset \mathbb{W} \subset \mathbb{H}
\end{equation}
become continuous. Thus, after the named identifications, in \eqref{EQ: AuxiliarySpacesEmbeddings} we just have inclusions of dense (since $\mathcal{D}(A)$ is dense) subspaces in $\mathbb{H}$, where each inclusion becomes a bounded operator when the subspaces are endowed with their norms. In this sense, we do not distinguish between the elements of $\mathbb{E}_{0}$ or $\mathbb{W}$ and their images under the embeddings.

Note that in \eqref{EQ: AuxiliarySpacesEmbeddings} we assumed neither $\mathbb{E}_{0} \subset \mathbb{W}$ nor $\mathbb{W} \subset \mathbb{E}_{0}$, but both situations may appear (see Examples \ref{EX: ExampleDelayand} and \ref{EX: ExampleBoundary} below). Let us state the following regularity assumption which requires $\mathbb{E}_{0} \subset \mathbb{W}$.
\begin{description}
	\item[\textbf{(REG1)}\label{DESC:REG1}] For any $v_{0} \in \mathbb{E}_{0}$, $T>0$ and $\xi(\cdot) \in L_{2}(0,T;\Xi)$ we have $v(\cdot;v_{0},\xi) \in C([0,T];\mathbb{W})$ and the map $\mathbb{E}_{0} \times L_{2}(0,T;\Xi) \ni (v_{0},\xi(\cdot)) \mapsto v(\cdot;v_{0},\xi) \in C([0,T];\mathbb{W})$ is continuous.
\end{description}

Aimed to study the infinite-horizon quadratic optimization problem, we consider a bounded symmetric sesquilinear\footnote{We use the convention that a sesquilinear form is linear in the first and conjugate-linear in the second argument.} form $\mathcal{W}$ on $\mathcal{D}(A) \times \Xi$ and its Hermitian quadratic form $\mathcal{F}$ given by $\mathcal{F}(v,\xi) = \mathcal{W}\left((v,\xi),(v,\xi)\right)$ for $(v,\xi) \in \mathcal{D}(A) \times \Xi$. Let us consider the following assumption.
\begin{description}
	\item[\textbf{(REG2)}\label{DESC:REG2}] There exists a constant $C_{\mathcal{W}}>0$ such that for any $v_{0} \in \mathcal{D}(A)$, $T>0$ and $\xi(\cdot) \in C^{1}([0,T];\Xi)$ for the solution $v(\cdot)=v(\cdot;v_{0},\xi)$ we have\footnote{Note that in this case the solution $v(\cdot)$ is classical and, in particular, we have $v(\cdot) \in C([0,T];\mathcal{D}(A))$ (see Theorem 6.5, Chapter I in \cite{Krein1971}).}
	\begin{equation}
		\begin{split}
			\int_{0}^{T}\left| \mathcal{F} (v(t), \xi(t)) \right|dt \leq C_{\mathcal{W}} \left( |v(0) |^{2}_{\mathbb{W}} + |v(T) |^{2}_{\mathbb{W}} + \|v(\cdot)\|^{2}_{L_{2}(0,T;\mathbb{W})} + \| \xi(\cdot) \|^{2}_{L_{2}(0,T;\Xi)}\right).
		\end{split}
	\end{equation}
\end{description}
This condition is a kind of regularity of the control system w.r.t. the quadratic form $\mathcal{F}$, which will be used below to extend the integral functional to a larger class of solutions.

We say that the pair $(A,B)$ is $(L_{2},\mathbb{W})$-\textit{controllable} in $\mathbb{E}_{0}$ if for every $v_{0} \in \mathbb{E}_{0}$ there exists $\xi(\cdot) \in L_{2}(0,+\infty;\Xi)$ such that $v(\cdot)=v(\cdot;v_{0},\xi) \in L_{2}(0,+\infty;\mathbb{W})$.

Let us consider the Hilbert space $\mathcal{Z} := L_{2}(0,+\infty;\mathbb{W}) \times L_{2}(0,+\infty;\Xi)$ with the usual norm. Following A.L.~Likhtarnikov and V.A.~Yakubovich  \cite{Likhtarnikov1977}, let $\mathfrak{M}_{v_{0}}$ be the set of all $( v(\cdot),\xi(\cdot) ) \in \mathcal{Z}$ such that $v(\cdot)=v(\cdot;v_{0},\xi)$ for some $v_{0} \in \mathbb{H}$. Every such a pair is called a \textit{process through} $v_{0}$. Note that the property of $(L_{2},\mathbb{W})$-controllability in $\mathbb{E}_{0}$ of the pair $(A,B)$ is equivalent to that $\mathfrak{M}_{v_{0}}$ is nonempty for all $v_{0} \in \mathbb{E}_{0}$. Since the embedding $\mathbb{W} \subset \mathbb{H}$ is continuous, from \eqref{EQ: MildSolution} one can see that $\mathfrak{M}_{v_{0}}$ is a closed affine subspace of $\mathcal{Z}$ given by a proper translate of $\mathfrak{M}_{0}$ (see also Lemma \ref{LEM: OperatorDDelay} below).

Now we consider the quadratic functional $\mathcal{J}_{\mathcal{F}}$ associated with $\mathcal{F}$, which is formally given by
\begin{equation}
	\label{EQ: QuadraticFunctional}
	\mathcal{J}_{\mathcal{F}}(v(\cdot),\xi(\cdot)) := \int_{0}^{+\infty}\mathcal{F}(v(t),\xi(t))dt.
\end{equation}
It is important that $\mathcal{J}_{\mathcal{F}}$ may be not defined everywhere in $\mathcal{Z}$ (and $\mathcal{J}_{\mathcal{F}}$ may be unbounded in this sense). It will be shown below that it is sufficient (and, of course, necessary for the optimization problem to be correct) for $\mathcal{J}_{\mathcal{F}}$ to be defined on the \textit{space of processes} over $\mathbb{E}_{0}$, which is the linear space given by
\begin{equation}
	\label{EQ: SpaceZ0}
	\mathcal{Z}_{0} := \bigcup_{v_{0} \in \mathbb{E}_{0}} \mathfrak{M}_{v_{0}}
\end{equation}
and endowed with the norm
\begin{equation}
	\label{EQ: SpaceZ0Norm}
	\| (v(\cdot),\xi(\cdot)) \|_{\mathcal{Z}_{0}} := \| (v(\cdot), \xi(\cdot)) \|_{\mathcal{Z}} + \| v(0) \|_{\mathbb{E}_{0}}.
\end{equation}
It is clear that $\mathcal{Z}_{0}$ is a Banach space (see Lemma \ref{LEM: Z0IsBanach}).

\begin{remark}
\label{REM: RegConditionsImplyExtension}
It is important that under \nameref{DESC:REG1} and \nameref{DESC:REG2} the functional $\mathcal{J}_{\mathcal{F}}$ can be treated as a bounded continuous functional on $\mathcal{Z}_{0}$. Let us give a sketch of the proof. For $T>0$ and $v_{0} \in \mathbb{E}_{0}$ we consider the space $\mathfrak{M}^{T}_{v_{0}}$ of all finite-time processes through $v_{0}$ given by the pairs $(v(\cdot),\xi(\cdot))$ where $\xi(\cdot) \in L_{2}(0,T;\Xi)$ and $v(\cdot)=v(\cdot;v_{0},\xi)$. Consider the space $\mathcal{Z}^{T}_{0}$ given by the union of $\mathfrak{M}^{T}_{v_{0}}$ over $v_{0} \in \mathbb{E}_{0}$ and endow it with the norm (which is well-defined due to \nameref{DESC:REG1})
\begin{equation}
	\label{EQ: ZTcontinuousNorm}
	\| v(\cdot),\xi(\cdot) \|^{2}_{\mathcal{Z}^{T}_{0}} := |v(0) |^{2}_{\mathbb{W}} + |v(T) |^{2}_{\mathbb{W}} + \|v(\cdot)\|^{2}_{L_{2}(0,T;\mathbb{W})} + \| \xi(\cdot) \|^{2}_{L_{2}(0,T;\Xi)}.
\end{equation}
From \nameref{DESC:REG2} we can define the bounded quadratic functional $\mathcal{J}^{T}_{\mathcal{F}}$ on $\mathcal{Z}^{T}_{0}$ as
\begin{equation}
	\label{EQ: FiniteTimeQuadraticFuncGeneral}
	\mathcal{J}^{T}_{\mathcal{F}}(v(\cdot),\xi(\cdot)) := \int_{0}^{T}\mathcal{F}(v(s),\xi(s)) ds.
\end{equation} 
Note that the integral in \eqref{EQ: FiniteTimeQuadraticFuncGeneral} makes sense for $v_{0} \in \mathcal{D}(A)$ and $\xi(\cdot) \in C^{1}([0,T];\Xi)$. For more general elements of $\mathcal{Z}^{T}_{0}$ it should be treated as the limit of such integrals obtained by continuity. This is why we prefer to use the notations like $\mathcal{J}_{\mathcal{F}}$ and $\mathcal{J}^{T}_{\mathcal{F}}$ instead of integrals.

Now let $R_{T} \colon \mathcal{Z}_{0} \to \mathcal{Z}^{T}_{0}$ be the natural map given by the restriction of a process to $[0,T]$. Then it is not hard to check\footnote{A proper argument for this should use the (non-quadratic!) functionals \begin{equation}
		\mathcal{I}^{T}_{\mathcal{F}}(v(\cdot),\xi(\cdot)) := \int_{0}^{T}| \mathcal{F}(v(s),\xi(s)) |ds \text{ for } (v(\cdot),\xi(\cdot)) \in \mathcal{Z}^{T}_{0},
\end{equation}
which can be also defined by continuity on $\mathcal{Z}^{T}_{0}$.
} that for $(v(\cdot), \xi(\cdot)) \in \mathcal{Z}_{0}$ the limit
\begin{equation}
	\label{EQ: QuadraticLimitFunctionals}
	\mathcal{J}_{\mathcal{F}}(v(\cdot),\xi(\cdot)) := \lim\limits_{T \to +\infty} \mathcal{J}^{T}_{\mathcal{F}}(R_{T}( v(\cdot),\xi(\cdot) ) )
\end{equation}
exists and $\mathcal{J}_{\mathcal{F}}$ is bounded on $\mathcal{Z}_{0}$, namely, for all $(v(\cdot),\xi(\cdot)) \in \mathcal{Z}_{0}$ we have (here we use that $\mathbb{E}_{0} \subset \mathbb{W}$)
\begin{equation}
	\label{EQ: QuadaraticLimitExtensionBound}
	|\mathcal{J}_{\mathcal{F}}(v(\cdot),\xi(\cdot)) | \leq C_{\mathcal{W}}\| (v(\cdot),\xi(\cdot)) \|^{2}_{\mathcal{Z}_{0}}. 
\end{equation}
\qed
\end{remark}

For purposes of wider applications (for example, to parabolic systems with boundary controls), where \nameref{DESC:REG1} may not be satisfied, we consider the following assumption.
\begin{description}
	\item[\textbf{(QF)}\label{DESC:QF}] The formally defined quadratic functional $\mathcal{J}_{\mathcal{F}}$ from \eqref{EQ: QuadraticFunctional} induces a bounded quadratic functional on $\mathcal{Z}_{0}$ (endowed with the norm in \eqref{EQ: SpaceZ0Norm}) in the sense that the functionals $\mathcal{J}^{T}_{\mathcal{F}}$, which are defined for $v_{0} \in \mathcal{D}(A)$ and $\xi(\cdot) \in C^{1}([0,T];\Xi)$ by the integral in \eqref{EQ: FiniteTimeQuadraticFuncGeneral}, can be extended to the whole $\mathcal{Z}^{T}_{0}$ endowed with the norm
	\begin{equation}
		\label{EQ: ZTsummablenorm}
			\| v(\cdot),\xi(\cdot) \|^{2}_{\mathcal{Z}^{T}_{0}} := |v(0) |^{2}_{\mathbb{W}} + \|v(\cdot)\|^{2}_{L_{2}(0,T;\mathbb{W})} + \| \xi(\cdot) \|^{2}_{L_{2}(0,T;\Xi)}.
	\end{equation}
	by continuity and the limit \eqref{EQ: QuadraticLimitFunctionals} holds along with the bound in \eqref{EQ: QuadaraticLimitExtensionBound}.
\end{description}

Note that under \nameref{DESC:REG1} the norms in \eqref{EQ: ZTcontinuousNorm} and \eqref{EQ: ZTsummablenorm} are equivalent due to the bounded inverse theorem. Thus, as it was shown in Remark \ref{REM: RegConditionsImplyExtension}, conditions \nameref{DESC:REG1} and \nameref{DESC:REG2} imply \nameref{DESC:QF}.

Now let us take any process $(v(\cdot),\xi(\cdot)) \in \mathfrak{M}_{0}$. Let $\hat{v}(\cdot)$ and $\hat{\xi}(\cdot)$ be the Fourier transforms of $v(\cdot)$ and $\xi(\cdot)$, where the latter are extended by zero to the negative semi-axis. Then it can be shown that for almost all $\omega \in \mathbb{R}$ we have $\hat{v}(\omega) \in \mathcal{D}(A)$ and $i \omega \hat{v}(\omega) = A\hat{v}(\omega) + B\hat{\xi}(\omega)$ satisfied (see \cite{Likhtarnikov1977,LouisWexler1991} or Lemma \ref{LEM: QuadraticFormDelayQf} below). Consider the following assumption.
\begin{description}
	\item[\textbf{(FT)}\label{DESC:FT}] For any process $(v(\cdot),\xi(\cdot)) \in \mathfrak{M}_{0}$ we have
	\begin{equation}
		\label{EQ: QF2FourierIdentity}
		\mathcal{J}_{\mathcal{F}}( v(\cdot),\xi(\cdot) ) = \int_{-\infty}^{+\infty}\mathcal{F}(\hat{v}(\omega),\hat{\xi}(\omega))d\omega,
	\end{equation}
	where $\hat{v}(\cdot)$ and $\hat{\xi}(\cdot)$ are the Fourier transforms of $v(\cdot)$ and $\xi(\cdot)$ (extended by zero to the negative semi-axis) in the spaces $L_{2}(\mathbb{R};\mathbb{H})$ and $L_{2}(\mathbb{R};\Xi)$ respectively.
\end{description}

\begin{remark}
	In the case of delay equations \nameref{DESC:FT} is related to special structure of solutions (see Lemma \ref{LEM: QuadraticFormDelayQf}), which also allows to interpret the quadratic functional $\mathcal{J}_{\mathcal{F}}$ on $\mathcal{Z}_{0}$ (see \eqref{EQ: NeutralInterpretationQuadraticFunctionalInfinity}).
\end{remark}

To help the reader understand the context better, we give two examples of our main interest (of course, these are not the only ones possible, see also \cite{Anikushin2020FreqParab}). Note also that there appears a Banach space $\mathbb{E} \supset \mathcal{D}(A)$ such that the quadratic form $\mathcal{F}$ is bounded on $\mathbb{E} \times \Xi$. This space is essential for assumption \nameref{DESC:RES} below and further investigations concerned with delay equations.
\begin{example}[Neutral delay equations]
	\label{EX: ExampleDelayand}
    For some $\tau>0$ let $\mathbb{E} := C([-\tau;0];\mathbb{C}^{n})$ be embedded into 
		\begin{equation}
			\mathbb{E}_{0} = \mathbb{W} = \mathbb{H} := \mathbb{C}^{n} \times L_{2}(-\tau,0;\mathbb{C}^{n})
		\end{equation}
		as $\phi \mapsto (\phi(0) + D_{0} \phi, \phi)$, where $D_{0} \colon C([-\tau;0];\mathbb{C}^{n}) \to \mathbb{C}^{n}$ is a bounded linear operator which is non-atomic at zero (as in Subsection \ref{SUBSEC: NDELinearSystems}). Put also $\Xi := \mathbb{C}^{m}$. In Subsection \ref{SUBSEC: NDELinearSystems} we will define operators $A$ and $B$ corresponding to neutral delay equations in $\mathbb{R}^{n}$ ($\mathbb{C}^{n}$). Let $\mathbb{M} := \mathbb{C}^{r}$ and $C \colon \mathbb{E} \to \mathbb{M}$ be a linear bounded operator. One can consider the quadratic form given by
		\begin{equation}
			\label{EQ: DelayEquationFormExample}
			\mathcal{F}(v,\xi) = |\xi|^{2}_{\Xi} - \Lambda^{2}|Cv|^{2}_{\mathbb{M}}
		\end{equation}
		for some constant $\Lambda>0$. Although $C$ is defined only on $\mathbb{E}$ (for example, it may be a $\delta$-functional), we have the estimate in \nameref{DESC:REG2} satisfied (see \eqref{EQ: NeutralInterpretationEstimateFiniteT}) for the solutions of linear inhomogeneous problems corresponding to the operator $A + \nu I$, where $\nu \in \mathbb{R}$. Since \nameref{DESC:REG1} is obviously satisfied in this case and $\mathbb{E}_{0} = \mathbb{W}$, we have \nameref{DESC:QF} satisfied. Moreover, \nameref{DESC:FT} is proved in Lemma \ref{LEM: QuadraticFormDelayQf}. \qed
\end{example}
\begin{example}[Parabolic boundary control problems]
	\label{EX: ExampleBoundary}
	Let $-A$ be a positive self-adjoint operator acting in a Hilbert space $\mathbb{E}_{0}=\mathbb{W}_{0}$ and having compact resolvent. Let $\mathbb{W}_{\beta} := \mathcal{D}((-A)^{\beta})$ for $\beta \in \mathbb{R}$ be the scale of Hilbert spaces given by the powers of $-A$. Then $A$ can be extended to an operator (which we still denote by $A$) in $\mathbb{H} := \mathbb{W}_{-1/2}$ with the domain $\mathcal{D}(A) = \mathbb{E} = \mathbb{W} := \mathbb{W}_{1/2}$, which generates a $C_{0}$-semigroup in $\mathbb{H}$ and satisfies
	\begin{equation}
		(-Av,v)_{\mathbb{W}_{0}} = |v|^{2}_{\mathbb{W}_{1/2}}
	\end{equation}
	Let $B \colon \Xi \to \mathbb{H}$ be a bounded linear operator (for example, $\Xi$ may be the Sobolev space of traces on the boundary of a domain). Such operators correspond to parabolic problems with boundary controls (nonhomogeneous boundary conditions; see \cite{Likhtarnikov1976}). Now we consider the control system given by the pair $(A+\nu I, B)$ for some $\nu \in \mathbb{R}$. For $v_{0} \in \mathbb{E}_{0} = \mathbb{W}_{0}$ and $\xi(\cdot) \in L_{2}(0,T;\Xi)$ the solution $v(\cdot)=v(\cdot;v_{0},\xi)$ belongs to the Lions-Magenes space $W(0,T)$ given by the functions $v(\cdot)$ such that $v(\cdot) \in L_{2}(0,T;\mathbb{W}_{1/2})$ and $\dot{v}(\cdot) \in L_{2}(0,T;\mathbb{W}_{-1/2})$. Since $\mathbb{W} = \mathcal{D}(A) = \mathbb{W}_{1/2}$ and $v(\cdot) \in L_{2}(0,T;\mathbb{W}_{1/2})$, we have \nameref{DESC:QF} satisfied. Moreover, \nameref{DESC:FT} is obvious due to the Plancherel theorem and $\mathbb{E}=\mathbb{W}$. Note also that $W(0,T)$ can be naturally continuously embedded in $C([0,T];\mathbb{W}_{0})$ (see Chapter 3 in the monograph of J.L.~Lions \cite{Lions1971OptCont}). The regulator problem for such systems was studied by V.A.~Yakubovich and A.L.~Likhtarnikov in \cite{Likhtarnikov1976}. \qed
\end{example}

Under \nameref{DESC:QF} we call a process through $v_{0} \in \mathbb{E}_{0}$ \textit{optimal} if it is a minimum point of $\mathcal{J}_{\mathcal{F}}$ on $\mathfrak{M}_{v_{0}}$. To describe conditions for the existence of a unique optimal process we will use the values $\alpha_{1}$,$\alpha_{2}$ and $\alpha_{3}$ introduced below.

Namely, we will use the value $\alpha_{1}$ given by
\begin{equation}
	\label{EQ: Alpha1}
	\alpha_{1}:= \inf_{(v(\cdot),\xi(\cdot)) \in \mathfrak{M}_{0}} \frac{\mathcal{J}_{\mathcal{F}}(v(\cdot),\xi(\cdot))}{\|v(\cdot),\xi(\cdot)\|^{2}_{\mathcal{Z}}},
\end{equation}
the value $\alpha_{2}$ given by
\begin{equation}
	\label{EQ: Alpha2}
	\alpha_{2} := \inf \frac{ \mathcal{F}(v,\xi) }{ |v|^{2}_{\mathbb{W}} + |\xi|^{2}_{\Xi} },
\end{equation}
where the infimum is taken over all $\omega \in \mathbb{R}$, $v \in \mathcal{D}(A)$ and $\xi \in \Xi$ such that $i\omega v = Av + B\xi$, and the value $\alpha_{3}$ given by
\begin{equation}
	\label{EQ: Alpha3}
	\alpha_{3} := \inf_{\omega \in \mathbb{R}} \inf_{ \xi \in \Xi} \frac{\mathcal{F}(-(A-i\omega I)^{-1}B\xi,\xi) }{|\xi|^{2}_{\Xi}}
\end{equation}
Moreover, we say that $\alpha_{3}$ is \textit{well-defined} if the spectrum of $A$ does not intersect with the imaginary axis and the operators $(A-i\omega I)^{-1}B$ are bounded in the norm of $\mathcal{L}(\Xi;\mathbb{W})$ uniformly for all $\omega \in \mathbb{R}$.

Let $\langle v, f \rangle := f(v)$ be the natural dual pairing between $v \in \mathbb{E}_{0}$ and $f \in \mathbb{E}^{*}_{0}$. We follow the convention that $\mathbb{E}^{*}_{0}$ consists of linear functionals $f$ on $\mathbb{E}_{0}$ and the multiplication of $f$ on $\lambda \in \mathbb{C}$ is given by $(\lambda f)(v) := \overline{\lambda} f(v)$ for each $v \in \mathbb{E}_{0}$, where $\overline{\lambda}$ is the complex conjugate of $\lambda$. One of our main results is contained in the following theorem.
\begin{theorem}
	\label{TH: OptimalProcessDelay}
	Let the pair $(A,B)$ be $(L_{2},\mathbb{W})$-controllable in $\mathbb{E}_{0}$ and suppose that \nameref{DESC:QF} and \nameref{DESC:FT} are satisfied for the Hermitian form $\mathcal{F}$. We have the following:
	
	1). If $\alpha_{1}>0$, then for every $v_{0} \in \mathbb{E}_{0}$ the quadratic functional $\mathcal{J}_{\mathcal{F}}$ has a unique minimum $(v^{0}(\cdot;v_{0}),\xi^{0}(\cdot;v_{0}))$ on $\mathfrak{M}_{v_{0}}$ and the map
	\begin{equation}
		\mathbb{E}_{0} \ni v_{0} \mapsto (v^{0}(\cdot;v_{0}),\xi^{0}(\cdot;v_{0})) \in \mathcal{Z}_{0}
    \end{equation}
	is a bounded linear operator. Moreover, there exists $P \in \mathcal{L}(\mathbb{E}_{0};\mathbb{E}^{*}_{0})$ such that for all $v_{0} \in \mathbb{E}_{0}$ we have
    \begin{equation}
    \langle v_{0}, Pv_{0} \rangle = \mathcal{J}_{\mathcal{F}}(v^{0}(\cdot;v_{0}),\xi^{0}(\cdot;v_{0})).
    \end{equation}
    The operator $P$ is symmetric in the sense that $\langle v, Pw \rangle = \overline{\langle w, Pv \rangle}$ for all $v, w \in \mathbb{E}_{0}$. For its quadratic form $V(v):=\langle v, Pv \rangle$ and any $T \geq 0$ we have
    \begin{equation}
    \label{EQ: OperatorSolutionsInequality}
    V(v(T)) - V(v_{0}) + \int_{0}^{T} \mathcal{F}(v(t),\xi(t))dt \geq 0,
    \end{equation}
	where $v(t)=v(t;v_{0},\xi)$ is the mild solution to \eqref{EQ: ControlSystem} with arbitrary $v(0)=v_{0} \in \mathbb{E}_{0}$ and $\xi(\cdot) \in L_{2}(0,T;\Xi)$ which additionally satisfies $v(\cdot) \in C([0,T];\mathbb{E}_{0})$ and the integral with $\mathcal{F}$ is understood in the sense of \nameref{DESC:QF}.
	
	2). $\alpha_{1}>0$ is equivalent to $\alpha_{2}>0$. Moreover, if $\alpha_{3}$ is well-defined in the above given sense, then $\alpha_{2}>0$ is equivalent to $\alpha_{3}>0$.
\end{theorem}

For many applications (for example, concerned with inertial manifolds) we often need only the operator $P$ such that the inequality in \eqref{EQ: OperatorSolutionsInequality} is satisfied. It turns out that in this case we can get rid of the controllability assumption (necessary for the optimization problem to be correct), which is not easy to check in practice (or it even may fail to hold). On this way, the main idea is to consider a modified control system with an extended control space, for which the controllability assumption is automatically satisfied (see below). This idea was firstly suggested in the work of A.V.~Proskurnikov \cite{Proskurnikov2015} for the Frequency Theorem from \cite{Likhtarnikov1977}. In our case, we will use the following assumption.
\begin{description}
	\item[\textbf{(RES)}\label{DESC:RES}] There exists a Banach space $\mathbb{E}$ continuously embedded into $\mathbb{H}$ and identified with its image under the embedding such that $\mathcal{D}(A) \subset \mathbb{E} \subset \mathbb{W}$ and both inclusions are continuous. Moreover, the operator $A$ does not have spectrum on the imaginary axis and
	\begin{description}
		\item[\textbf{(RES1)}\label{DESC:RES1}] The quadratic form $\mathcal{F}$ is defined and bounded on $\mathbb{E} \times \Xi$;
		\item[\textbf{(RES2)}\label{DESC:RES2}] The norms of $(A-i\omega I)^{-1}$ in $\mathcal{L}(\mathbb{W};\mathbb{E})$ are bounded uniformly in $\omega \in \mathbb{R}$;
		\item[\textbf{(RES3)}\label{DESC:RES3}] The norms of $(A-i\omega I)^{-1}B$ in $\mathcal{L}(\Xi;\mathbb{E})$ are bounded uniformly in $\omega \in \mathbb{R}$.
	\end{description}
\end{description}

It is convenient to introduce the following definition. We will say that the semigroup $G(t)$, where $t \geq 0$, is \textit{potentially} $(L_{2},\mathbb{W})$-\textit{controllable} in $\mathbb{E}_{0}$ if for any $v_{0} \in \mathbb{E}_{0}$ there exists $\varkappa \in \mathbb{R}$ such that for $v_{\varkappa}(t) = e^{\varkappa t}G(t)v_{0}$ we have $v_{\varkappa}(\cdot) \in L_{2}(0,\infty;\mathbb{W})$. For example, such a property is satisfied if $G(t)$ is also a $C_{0}$-semigroup in $\mathbb{W}$ and $\mathbb{E}_{0} \subset \mathbb{W}$ (as in Example \ref{EX: ExampleDelayand} or semilinear parabolic equations \cite{Anikushin2020FreqParab}). Moreover, in Example \ref{EX: ExampleBoundary} it is not the case that $\mathbb{E}_{0} \subset \mathbb{W}$, but we have a smoothing estimate that gives the desired.

For a given pair $(A,B)$ we define the \textit{extended control pair} $(A,\widehat{B})$, where $\widehat{B} \in \mathcal{L}(\widehat{\Xi},\mathbb{H})$ is the \textit{extended control operator} acting from the \textit{extended control space} $\widehat{\Xi} := \Xi \times \mathbb{W}$ such that $\widehat{B}(\xi,\eta) = B\xi + \eta$ for any $(\xi,\eta) \in \widehat{\Xi}$. It can be seen that the potential $(L_{2},\mathbb{W})$-controllability in $\mathbb{E}_{0}$ implies that $(A,\widehat{B})$ is $(L_{2},\mathbb{W})$-controllable in $\mathbb{E}_{0}$ (see Section \ref{SUBSEC: RelaxingControllability}).

The following theorem extends the corresponding result from \cite{Proskurnikov2015} and our work\footnote{Note that Theorem 3 from \cite{Anikushin2020FreqParab} requires a clarification of its statement which is concerned with replacing formula (3.6) therein to an analog of \eqref{EQ: RelaxedFreqThLyapunivInequality} from the present work.} \cite{Anikushin2020FreqParab}.
\begin{theorem}
	\label{TH: FreqThRelaxedDelay}
	Let the pair $(A,B)$ satisfy \nameref{DESC:RES} with some Hermitian form $\mathcal{F}$. Moreover, let the extended control pair $(A,\widehat{B})$ be $(L_{2},\mathbb{W})$-controllable in $\mathbb{E}_{0}$ and satisfy \nameref{DESC:QF} and \nameref{DESC:FT} w.r.t. the extended Hermitian form $\widehat{\mathcal{F}}(v,\xi,\eta):=\mathcal{F}(v,\xi)$ of $(v,\xi,\eta) \in \mathbb{E} \times \Xi \times \mathbb{W}$. Then the following conditions posed for $(A,B)$ and $\mathcal{F}$ are equivalent:
	
	1). $\alpha_{2}>0$
	
	2). There exists $P \in \mathcal{L}(\mathbb{E}_{0};\mathbb{E}^{*}_{0})$ such that $P$ is symmetric and for $V(v):=\langle v, Pv \rangle$ there exists $\delta>0$ such that for any $T \geq 0$ we have
	\begin{equation}
	\label{EQ: RelaxedFreqThLyapunivInequality}
	V(v(T)) - V(v_{0}) + \int_{0}^{T} \mathcal{F}(v(t),\xi(t))dt \geq \delta \int_{0}^{T}( |v(t)|^{2}_{\mathbb{W}} + |\xi(t)|^{2}_{\Xi})dt,
	\end{equation}
	where $v(t)=v(t;v_{0},\xi)$ is the mild solution to \eqref{EQ: ControlSystem} with arbitrary $v(0)=v_{0} \in \mathbb{E}_{0}$ and $\xi(\cdot) \in L_{2}(0,T;\Xi)$ which additionally satisfies $v(\cdot) \in C([0,T];\mathbb{E}_{0})$ and the integral with $\mathcal{F}$ is understood in the sense of \nameref{DESC:QF}.
	
	Moreover, $\alpha_{2}>0$ is equivalent to $\alpha_{3}>0$.
\end{theorem}

Theorem \ref{TH: OptimalProcessDelay} covers the non-singular versions of the Frequency Theorem from \cite{Likhtarnikov1976} and \cite{Likhtarnikov1977} as well as the recently established in \cite{Anikushin2020FreqParab} version which is appropriate for semilinear parabolic equations. Note also that the paper of J.-Cl.~Louis and D.~Wexler \cite{LouisWexler1991} is a rediscovery of the main result from \cite{Likhtarnikov1977}. We note that in \cite{Likhtarnikov1976,Likhtarnikov1977,LouisWexler1991} it is also studied the regularity of optimal controls and their expression in a closed form. Such results usually require additional regularity assumptions and are out of our interest, at least in the case of stationary problems (i.e. with constant $A$, $B$ and $\mathcal{F}$). 

Although the proofs of Theorems \ref{TH: OptimalProcessDelay} and \ref{TH: FreqThRelaxedDelay} in most follow similar arguments as in our paper \cite{Anikushin2020FreqParab}, but applied in the wider context, we present them in Section \ref{SEC: ProofsFreqTheorem} for the sake of completeness and, that is more important, to emphasize the role of the quadratic form $\mathcal{F}$ defined on $\mathbb{E} \times \Xi$ with some intermediate Banach space $\mathbb{E}$ as in \nameref{DESC:RES}. This is essential for the study of delay equations presented in further sections.

In our work \cite{Anikushin2020Geom}, we use quadratic Lyapunov functionals to construct invariant manifolds (such as inertial manifolds or local stable/unstable/slow manifolds along invariant sets) and the corresponding foliations for nonautonomous dynamical systems given by abstract asymptotically compact cocycles in Banach spaces. Our approach allows to treat and improve many scattered results in the field (some of which will be discussed below) within a unified geometric context. As to improvements, besides the ones discussed below, note that we do not know any papers concerned with inertial manifolds for parabolic equations with nonlinear boundary conditions.

In applications, the inequality $\alpha_{3} > 0$, which is called a \textit{frequency inequality} or \textit{frequency-domain condition}, is useful. In concrete situations, it takes well-known forms and provides optimal conditions (see our Remark \ref{REM: NDEoptimalityFreqIneq} below and Subsection 5.2 in \cite{Anikushin2020FreqParab}). For example, for semilinear parabolic equations with a self-adjoint linear part, it takes the optimal form of the Spectral Gap Condition (see \cite{Anikushin2020FreqParab}), the optimal version of which was firstly established by A.V.~Romanov \cite{Romanov1994} and M.~Miclav\v{c}i\v{c} \cite{Miclavcic1991} by using different methods. For non-self-adjoint linear parts, it is related to the inequality from \cite{Miclavcic1991}. Namely, the conditions of M.~Miclav\v{c}i\v{c} are not quadratic in general, but our theory allows to consider their quadratic analogs. We refer to our papers \cite{Anikushin2020FreqParab,Anikushin2020Geom} for more discussions in the context of semilinear parabolic problems.

Moreover, in applications to nonlinear systems, the inequality in \eqref{EQ: RelaxedFreqThLyapunivInequality} implies (see Theorem \ref{TH: QuadraticLyapunovFuncDelay}) a condition similar to the so-called strong cone property in a differential form (see the survey of S.~Zelik \cite{Zelik2014} or A.~Kostianko et al. \cite{KostiankoZelikSA2020}), although the class of cones in our case is more general and appropriate for the study of non-self-adjoint problems, among which delay equations occupy a special place. We note that from the applications perspective, the Frequency Theorem justifies that quadratic cones are more natural for the construction of inertial manifolds, at least in the cases covered by quadratic constraints. This may possibly explain why the attempt of N.~Koksch and S.~Siegmund \cite{KokschSiegmund2002}, who used a special class of non-quadratic cones to generalize inertial manifolds theory, did not succeed in applications. Although, it is interesting to generalize the Frequency Theorem and our theory from \cite{Anikushin2020Geom} for the case of general non-quadratic cones, for example, to construct such cones under the ``non-quadratic'' conditions of M.Miclav\v{c}i\v{c} \cite{Miclavcic1991}. This problem follows the philosophy we call ``the dialectics of the Lyapunov-Perron and Hadamard methodologies''. It is particularly concerned with a kind of equivalence between these two methods from the theory of invariant manifolds, but in a much broader context than it is usually considered (see the introduction in \cite{Anikushin2020Geom} for a discussion).

In \cite{Anikushin2020Red} we posed a problem concerned with applications of the Frequency Theorem from \cite{Likhtarnikov1977} to delay equations. Namely, in \cite{Likhtarnikov1977} it is required that the form $\mathcal{F}$ is bounded in $\mathbb{H} \times \Xi$ and, consequently, the results therein cannot be applied for delay equations (for which $\mathbb{H}$ is a kind of $L_{2}$-space) and quadratic forms containing $\delta$-functionals (which are unbounded in $L_{2}$). It turns out that the presented relaxation given by assumptions \nameref{DESC:REG1} and \nameref{DESC:REG2} is sufficient to study a large class of such equations. Namely, we will show that main parts of solutions to linear inhomogeneous problems associated with neutral delay operators belong to what we call spaces of agalmanated functions (see Theorem \ref{TH: DelayRegCond}). Such functions are given by the sum of what we call adorned and twisted functions and both functional classes are turn out to be linearly independent. We call the corresponding decomposition of the main part of a solution into the sum of such functions a \textit{structural Cauchy formula}. Such a decomposition reveals some structural (not to be confused with regularity) properties of solutions. Namely, in Section \ref{SEC: LemmaOnEstimatesForLF} we show that, on such spaces, operators given by pointwise measurements of certain unbounded functionals can be defined and understood in the integral sense. In Section \ref{SEC: NeutralDelayEquationsApplications}, these operators arise in the study of solutions and quadratic functionals, where \nameref{DESC:REG2} just follows from their boundedness. Here we also use and develop some results of J.A.~Burns, T.L.~Herdman and H.W.~Stech \cite{BurnsHerdmanStech1983} on the well-posedness of linear neutral delay equations in Hilbert spaces.

In the case of delay equations, we should especially mention the works of R.A.~Smith on the Poincar\'{e}-Bendixson theory and convergence theorems \cite{Smith1992,Smith1980} and the results of Yu.A.~Ryabov \cite{Ryabov1967}, R.D.~Driver \cite{Driver1968SmallDelays} and C.~Chicone \cite{Chicone2003,Chicone2004} on inertial manifolds for delay equations in $\mathbb{R}^{n}$ with small delays and their extension to neutral equations recently obtained by S.~Chen and J.~Shen \cite{ChenShen2021}.  We will return to these works in Section \ref{SEC: NeutralDelayEquationsApplications}. In particular, we will show that Smith's frequency condition from \cite{Smith1992} leads to the existence of quadratic functionals\footnote{Note that R.A.~Smith used quadratic functionals and a ``pocket'' version of the Frequency Theorem in his works for ODEs \cite{Smith1987OrbStab}, but he probably abandoned this approach for infinite-dimensional systems due to inability of obtaining optimal conditions for their existence \cite{Smith1980}. Returning to this approach is our main motivation for this and adjacent works.} and in the case of small delays is similar (and sometimes more sharper) to the conditions used in \cite{Chicone2003}. This allows to cover and improve a part of the results from \cite{Chicone2003,Ryabov1967,Driver1968SmallDelays,ChenShen2021} (as well as from Smith's papers) with the aid of our work \cite{Anikushin2020Geom} (see Subsection \ref{SUBSEC: EquationsWithSmallDelays}). In general, Smith's frequency condition is applicable for general delay equations in $\mathbb{R}^{n}$ and allows to construct inertial manifolds whose dimension is not limited to $n$ as in the case of small delays. 

In addition, a simple, but nontrivial, example in the field of delay equations that goes beyond the scope of the mentioned results is given by the delayed oscillator proposed by M.J.~Suarez and P.S.~Schopf \cite{Suarez1988} as a simplified model for El~Ni\~{n}o Southern Oscillation (ENSO). It is described by a scalar delay equation as
\begin{equation}
	\label{EQ: SuarezSchopfExample}
	\dot{x}(t) = x(t) - \alpha x(t-\tau) - x^{3}(t),
\end{equation}
where $\alpha \in (0,1)$ and $\tau>0$ are parameters. In our paper \cite{Anikushin2021SS} (joint with A.O.~Romanov) we provided an analytical-numerical investigation based on the Frequency Theorem and our work \cite{Anikushin2020Geom} that indicates the existence of two-dimensional inertial manifolds in the model. Namely, analytically they exist in the ball of radius 1 for a range of parameters including $\alpha \in (0.5,1)$ and $\tau \in (1,1.96)$ and numerical experiments (in the absence of rigorous and sufficiently sharp estimates for the dissipativity region) suggest that the global attractor lies in the ball of radius 1 for many interested parameters. For such parameters the model may exhibit a multistable behavior, which went overlooked by previous investigations. Here hidden periodic orbits and homoclinics may appear and the influence of a small periodic forcing may lead to a low-dimensional chaos. This relates the model to periodically forced oscillators on the plane and may be essential to understand the modeled phenomenon. See \cite{Anikushin2021SS} and references therein for further discussions.

Another nontrivial application of the presented results is the studying of spectral gaps for compound cocycles generated by nonautonomous delay equations developed in our work \cite{Anikushin2023Comp}. Such cocycles can describe the evolution of infinitesimal volumes over global attractors and the exponential stability of such cocycles is related to their dimensional-like characteristics (see N.V.~Kuznetsov and V.~Reitmann \cite{KuzReit2020}; S.~Zelik \cite{ZelikAttractors2022}). In \cite{Anikushin2020Semigroups} we discuss why the problem of obtaining effective dimension estimates for delay equations cannot be studied by standard methods (namely, via the Liouville trace formula). Our results allow to obtain conditions for the uniform exponential stability of such cocycles by comparison with stationary problems. Here, frequency conditions cannot be verified analytically, since it requires solving first-order PDEs (in high-dimensional domains) with boundary conditions involving both partial derivatives and delays, but the problem seems to have the prospect of being studied by numerical methods. To the best of our knowledge, this is the only analytical method that can be used to study dimensional-like characteristics for a fairly general class of equations with delay.

Although here we are concentrated on applications of Theorem \ref{TH: FreqThRelaxedDelay} to delay equations in $\mathbb{R}^{n}$, we must note that Theorem \ref{TH: FreqThRelaxedDelay} can be applied to construct quadratic Lyapunov functionals for partial differential equations with delays. Here we should mention the work of L.~Boutet de Monvel, I.D.~Chueshov and A.V.~Rezounenko \cite{BoutetChueshovRezounenko1998} concerned with spectral gap conditions for the existence of inertial manifolds for semilinear parabolic equations with delay. It seems that the obtained in \cite{BoutetChueshovRezounenko1998} conditions are not sharp and our approach shall lead to sharper estimates for the existence of inertial manifolds. However, to apply our approach some preparatory work should be done. On this way, the monograph of I.D.~Chueshov \cite{Chueshov2015}, where the well-posedness of such problems in the space of continuous functions is studied, may be useful along with an extension of our approach from \cite{Anikushin2020Semigroups} to construct semigroups for such equations in appropriate Hilbert spaces.

In finite dimensions, the Frequency Theorem is known as the Kalman-Yakubovich-Popov (KYP) lemma. It has great success in the study of stability, periodicity, almost periodicity and dimension-like properties of ODEs. See, for example, the monographs of A.Kh.~Gelig, G.A.~Leonov and V.A.~Yakubovich \cite{Gelig1978} or N.V.~Kuznetsov and V.~Reitmann \cite{KuzReit2020} for a range of applications. Let us also mention our works \cite{Anikushin2019Vestnik,AnikushinRR2019,Anikushin2019SmithRed}, where quadratic cones constructed from the KYP lemma are used to study almost periodic ODEs. Note that the extension of these results (mainly the part of a quantitative nature) for infinite-dimensional systems requires further studying (in addition to the results of this paper).

There are few more works to mention. The optimal regulator problem (for some special cases of nonnegative forms $\mathcal{F}$) for delay equations was also considered by L.~Pandolfi \cite{Pandolfi1995}, who used an approach motivated by boundary control problems. Moreover, the regulator problem studied in the work of A.~Pritchard and D.~Salamon \cite{Pritchard1985} is also motivated by delay equations. However, as in \cite{Pandolfi1995}, they consider the case of a nonnegative form $\mathcal{F}$. Such a context does not allow to study problems of our interest, where indefinite quadratic functionals arise. Unbounded positive self-adjoint operator solutions to the Kalman-Yakubovich-Popov inequality in infinite dimensions were considered by D.Z.~Arov and O.J.~Staffans \cite{ArovStaffans2006} in order to study linear systems with a scattering supply rate. In \cite{Anikushin2021AADyn} we studied almost automorphic dynamics in almost periodic cocycles with one-dimensional inertial manifolds and conditions for the existence of such manifolds are derived via the Frequency Theorem obtained in the present work. In \cite{Anikushin2019+OnCom} we investigated compactness properties of the operator $P$ under certain regularity assumptions.

This paper is organized as follows. In Section \ref{SEC: ProofsFreqTheorem} we give a proof of Theorem \ref{TH: OptimalProcessDelay} (see Section \ref{SUBSEC: ProffDelayFTTheorem1}), Theorem \ref{TH: FreqThRelaxedDelay} (see Section \ref{SUBSEC: RelaxingControllability}) and establish several corollaries useful for applications. They are, in particular, concerned with the realification of $P$ and its sign properties (see Section \ref{SUBSEC: Realification}). In Section \ref{SEC: LemmaOnEstimatesForLF} we develop the theory of pointwise measurement operators on embracing spaces (see Section \ref{SUBSEC: EmbracingSpacesPointwiseMesOp}) and introduce the spaces of adorned (see Section \ref{SUBSEC: SpacedAdornedFunctions}), twisted (see Section \ref{SUBSEC: SpacesTwistedFunctions}) and agalmanated (see Section \ref{SUBSEC: AgalmanatedFunctionsNeutral}) functions which can be continuously embedded into the embracing spaces. Section \ref{SEC: NeutralDelayEquationsApplications} is devoted to applications for neutral delay equations in $\mathbb{R}^{n}$. In Section \ref{SUBSEC: NDELinearSystems} we study linear problems arising from such equations and, in particular, establish the structural Cauchy formula (see Theorem \ref{TH: DelayRegCond}). In Section \ref{SUBSEC: NDENonlinear} we transit to nonlinear systems and apply previous results to construct quadratic Lyapunov functionals for such problems (see Theorem \ref{TH: QuadraticLyapunovFuncDelay}). In Section \ref{SUBSEC: EquationsWithSmallDelays} we discuss connections with the works concerned with the existence of inertial manifolds for equations with small delays. In Section \ref{SEC: Discussion} a brief discussion on the presented and further research is given.
\section{Infinite-horizon quadratic regulator problem and quadratic Lyapunov-like functionals}
\label{SEC: ProofsFreqTheorem}
Most of the arguments used in this section are similar to our paper \cite{Anikushin2020FreqParab} and main ideas go back to the papers of A.L.~Likhtarnikov and V.A.~Yakubovich \cite{Likhtarnikov1976,Likhtarnikov1977}. The only difference is that they are applied in the wider context where some technical nuances arise.
\subsection{Proof of Theorem \ref{TH: OptimalProcessDelay}}
\label{SUBSEC: ProffDelayFTTheorem1}

Recall that we are dealing with the generator $A \colon \mathcal{D}(A) \subset \mathbb{H} \to \mathbb{H}$ of a $C_{0}$-semigroup $G(t)$, where $t \geq 0$, acting in a Hilbert space $\mathbb{H}$ and a control operator $B \in \mathcal{L}(\Xi;\mathbb{H})$ from a Hilbert space $\Xi$. With the pair $(A,B)$ we associate the control system given by \eqref{EQ: ControlSystem}.

Below, to the end of this subsection, we always suppose that the pair $(A,B)$ is $(L_{2},\mathbb{W})$-controllable in $\mathbb{E}_{0}$, where $\mathbb{E}_{0}$ is a Banach space and $\mathbb{W}$ is a Hilbert space which are continuously embedded into the main Hilbert space $\mathbb{H}$ and identified with their images under the embeddings such that the inclusions $\mathcal{D}(A) \subset \mathbb{E}_{0}$ and $\mathcal{D}(A) \subset \mathbb{W}$ from \eqref{EQ: AuxiliarySpacesEmbeddings} become bounded operators. Recall that the controllability assumption means that the affine subspace $\mathfrak{M}_{v_{0}}$ in $\mathcal{Z} = L_{2}(0,+\infty;\mathbb{W}) \times L_{2}(0,+\infty;\Xi)$ given by all processes $(v(\cdot),\xi(\cdot))$ through $v_{0}$, i.e. solution pairs from $\mathcal{Z}$ corresponding to the control system associated with the pair $(A,B)$ via \eqref{EQ: ControlSystem} such that $v(0)=v_{0}$, is not empty for each $v_{0} \in \mathbb{E}_{0}$.

Recall here the the space $\mathcal{Z}_{0}$ of processes over $\mathbb{E}_{0}$ given by the union of $\mathfrak{M}_{v_{0}}$ over all $v_{0} \in \mathbb{E}_{0}$. We endow $\mathcal{Z}_{0}$ with the norm $\|\cdot\|_{\mathcal{Z}_{0}}$ as in \eqref{EQ: SpaceZ0Norm}, i.e. $\|(v(\cdot),\xi(\cdot))\|_{\mathcal{Z}_{0}} = \| (v(\cdot),\xi(\cdot)) \|_{\mathcal{Z}} + \| v(0) \|_{\mathbb{E}_{0}}$ for all $(v(\cdot),\xi(\cdot)) \in \mathcal{Z}_{0}$.
\begin{lemma}
	\label{LEM: Z0IsBanach}
	We have that $\mathcal{Z}_{0}$ endowed with the norm $\|\cdot\|_{\mathcal{Z}_{0}}$ is a Banach space.
\end{lemma}
\begin{proof}
	Let $v_{k,0} \in \mathbb{E}_{0}$, where $k=1,2,\ldots$, be a sequence and suppose $z_{k} = (v_{k}(\cdot),\xi_{k}(\cdot)) \in \mathfrak{M}_{v_{k,0}}$ is a fundamental sequence in $\mathcal{Z}_{0}$. This is equivalent to that $v_{k,0}$ is fundamental in $\mathbb{E}_{0}$ and $z_{k}$ is fundamental in $\mathcal{Z}$. Since $\mathbb{E}_{0}$ and $\mathcal{Z}$ are Banach spaces, we have that $v_{k,0}$ converges in $\mathbb{E}_{0}$ to some $v_{0} \in \mathbb{E}_{0}$ and $z_{k}$ converges in $\mathcal{Z}$ to some $z = (v(\cdot),\xi(\cdot)) \in \mathcal{Z}$ as $k \to \infty$. We have to show that $z \in \mathfrak{M}_{v_{0}}$. Indeed, from the Cauchy formula \eqref{EQ: MildSolution} we have for all $t \geq 0$ that
	\begin{equation}
		\label{EQ: Z0BanachLemmaCauchyFormula}
		v_{k}(t) = G(t)v_{k,0} + \int_{0}^{t}G(t-s)B\xi_{k}(s)ds.
	\end{equation}
    Since $\xi_{k}(\cdot)$ converges to $\xi(\cdot)$ in $L_{2}(0,\infty;\Xi)$ and $v_{k,0}$ converges to $v_{0}$ in the space $\mathbb{E}_{0}$ that is continuously embedded into $\mathbb{H}$, we may pass to the limit in \eqref{EQ: Z0BanachLemmaCauchyFormula} as $k \to \infty$ and obtain that
    \begin{equation}
    	v(t) = G(t)v_{0} + \int_{0}^{t}G(t-s)B\xi(s)ds \text{ for all } t \geq 0.
    \end{equation}
    Since we have $z \in \mathcal{Z}$ by definition, this shows that $z \in \mathfrak{M}_{v_{0}}$ and finishes the proof.
\end{proof}

From Lemma \ref{LEM: Z0IsBanach} we particularly have that the space $\mathfrak{M}_{0}$ of processes through zero is a closed subspace in $\mathcal{Z}$ and $\mathfrak{M}_{v_{0}}$ is a closed affine subspace in $\mathcal{Z}$. Clearly, $\mathfrak{M}_{v_{0}}$ is given by a proper translate of $\mathfrak{M}_{0}$. A bit stronger statement is contained in the following lemma.
\begin{lemma}
	\label{LEM: OperatorDDelay}
	There exists $D \in \mathcal{L}(\mathbb{E}_{0};\mathcal{Z}_{0})$ such that $\mathfrak{M}_{v_{0}} = \mathfrak{M}_{0} + Dv_{0}$ for all $v_{0} \in \mathbb{E}_{0}$.
\end{lemma}
\begin{proof}
	Indeed, suppose $v_{0} \in \mathbb{E}_{0}$ is given and let $z \in \mathfrak{M}_{v_{0}}$ be any process through $v_{0}$. We define $Dv_{0} := z - \Pi z$, where $\Pi \colon \mathcal{Z} \to \mathfrak{M}_{0}$ is the orthogonal (in $\mathcal{Z}$) projector onto $\mathfrak{M}_{0}$. Note that $z_{1}-z_{2} \in \mathfrak{M}_{0}$ for any $z_{1},z_{2} \in \mathfrak{M}_{v_{0}}$ and, therefore, the definition of $Dv_{0} \in \mathfrak{M}_{v_{0}}$ is independent on the choice of $z \in \mathfrak{M}_{v_{0}}$.
	
	Let us show that $D \colon \mathbb{E}_{0} \to \mathcal{Z}_{0}$ is closed. Indeed, let a sequence $v_{k} \in \mathbb{E}_{0}$, where $k=1,2,\ldots$, converge in $\mathbb{E}_{0}$ to some $v_{0} \in \mathbb{E}_{0}$ and $Dv_{k}$ converge in $\mathcal{Z}_{0}$ to some $z \in \mathcal{Z}_{0}$ as $k \to \infty$. In particular, we have $z \in \mathfrak{M}_{v_{0}}$. Moreoever, since $0=\Pi Dv_{k} \to \Pi z$ as $k \to \infty$, we must also have $\Pi z = 0$ and, consequently, $Dv_{0} = z - \Pi z = z$ that shows $D$ is closed. 
	
	Since $\mathbb{E}_{0}$ and $\mathcal{Z}_{0}$ are Banach spaces (see Lemma \ref{LEM: Z0IsBanach}), from the Closed Graph Theorem we obtain that $D \in \mathcal{L}(\mathbb{E}_{0};\mathcal{Z}_{0})$. The proof is finished.
\end{proof}

For the next lemma, we recall the formally defined quadratic functional $\mathcal{J}_{\mathcal{F}}$ associated with the bounded quadratic form $\mathcal{F}$ on $\mathcal{D}(A) \times \Xi$ via \eqref{EQ: QuadraticFunctional}. In this context, assumption \nameref{DESC:QF} gives an interpretation of $\mathcal{J}_{\mathcal{F}}$ as a bounded quadratic functional on the space $\mathcal{Z}_{0}$ of processes over $\mathbb{E}_{0}$. Then $\alpha_{1} > 0$, where $\alpha_{1}$ is given by \eqref{EQ: Alpha1}, means exactly that $\mathcal{J}_{\mathcal{F}}$ is coercive on $\mathfrak{M}_{0}$. The following lemma is a generalization of the corresponding lemmas from \cite{Anikushin2020FreqParab,Likhtarnikov1977}.
\begin{lemma}
	\label{LEM: OptimizationOnAffineSubspaces}
	Suppose $\alpha_{1} > 0$ and let \nameref{DESC:QF} be satisfied. Then for every $v_{0} \in \mathbb{E}_{0}$ there exists a unique minimum $(v^{0}(\cdot;v_{0}), \xi^{0}(\cdot;v_{0}))$ of $\mathcal{J}_{\mathcal{F}}$ on $\mathfrak{M}_{v_{0}}$. Moreover, there exists $T \in \mathcal{L}(\mathbb{E}_{0};\mathcal{Z}_{0})$ such that $(v^{0}(\cdot;v_{0}), \xi^{0}(\cdot;v_{0})) = T v_{0}$ for all $v_{0} \in \mathbb{E}_{0}$.
\end{lemma}
\begin{proof}
	Let us take $v_{0} \in \mathbb{E}_{0}$ and consider two processes $( v_{0}(\cdot),\xi_{0}(\cdot) ), ( v(\cdot),\xi(\cdot) ) \in \mathfrak{M}_{v_{0}}$. Put $h_{v}(\cdot):=v(\cdot)-v_{0}(\cdot)$ and $h_{\xi}:= \xi(\cdot) - \xi_{0}(\cdot)$. Note that $(h_{v},h_{\xi}) \in \mathfrak{M}_{0}$. We have
	\begin{equation}
	\begin{split}
	\mathcal{J}_{\mathcal{F}}(v_{0}(\cdot)+h_{v}(\cdot),\xi_{0}(\cdot)+h_{\xi}(\cdot)) - \mathcal{J}_{\mathcal{F}}( v_{0}(\cdot), \xi_{0}(\cdot) ) =\\= 2\operatorname{Re}\mathcal{L}_{\mathcal{W}}( (h_{v}(\cdot),h_{\xi}(\cdot)) , (v_{0}(\cdot),\xi_{0}(\cdot)) ) + \mathcal{J}_{\mathcal{F}}(h_{v}(\cdot),h_{\xi}(\cdot)),
	\end{split}	
	\end{equation}
	where $\mathcal{W}$ is the symmetric sesquilinear form corresponding to $\mathcal{F}$ and
	\begin{equation}
		\label{EQ: LinearPartIntegralFunctional}
		\mathcal{L}_{\mathcal{W}}\left( (h_{v}(\cdot),h_{\xi}(\cdot)) , (v_{0}(\cdot),\xi_{0}(\cdot)) \right) = \int_{0}^{+\infty}\mathcal{W}((h_{v}(t),h_{\xi}(t) ),(v_{0}(t),\xi_{0}(t)))dt,
	\end{equation} 
	where the integral should be understood according to \nameref{DESC:QF}.
	
	Standard arguments show that the necessary and sufficient conditions for $\mathcal{J}_{\mathcal{F}}$ to attain a minimum at $(v_{0}(\cdot), \xi_{0}(\cdot))$ are given by
	\begin{equation}
	\label{EQ: OrthogonalityCond}
	\mathcal{L}_{\mathcal{W}}\left( (h_{v}(\cdot),h_{\xi}(\cdot)) , (v_{0}(\cdot),\xi_{0}(\cdot)) \right) = 0 \text{ for all } (h_{v},h_{\xi}) \in \mathfrak{M}_{0}
	\end{equation}
	and
	\begin{equation}
	\label{EQ: FormCond}
	\mathcal{J}_{\mathcal{F}}( h_{v}(\cdot),h_{\xi}(\cdot) ) \geq 0 \text{ for all } (h_{v},h_{\xi}) \in \mathfrak{M}_{0}.
	\end{equation}

	Now we put $h(\cdot) := (h_{v}(\cdot),h_{\xi}(\cdot)) \in \mathfrak{M}_{0}$. Then, by \nameref{DESC:QF}, for any $v_{0} \in \mathbb{E}_{0}$ and $z(\cdot) = ( v_{0}(\cdot),\xi_{0}(\cdot) ) \in \mathfrak{M}_{v_{0}}$ the left-hand side of \eqref{EQ: OrthogonalityCond} defines a continuous linear functional on $\mathfrak{M}_{0}$. Therefore, by the Riesz representation theorem, there exists a unique element $Qz \in \mathfrak{M}_{0}$ such that
	\begin{equation}
	\label{EQ: RieszRepresentEquality}
	\mathcal{L}_{\mathcal{W}}(h(\cdot),z(\cdot)) = (h(\cdot),Qz(\cdot))_{\mathcal{Z}} \text{ for all } h(\cdot) \in \mathfrak{M}_{0}.
	\end{equation}
	Clearly, $Q$ is a linear operator from $\mathcal{Z}_{0}$ to $\mathfrak{M}_{0}$ and it is bounded since for some $C>0$ we have
	\begin{equation}
		\| Qz \|^{2}_{\mathfrak{M}_{0}} = ( Qz(\cdot), Qz(\cdot) )_{\mathcal{Z}} = \mathcal{L}_{\mathcal{W}}(Qz(\cdot),z(\cdot)) \leq C \| Qz\|_{\mathcal{Z}_{0}} \cdot \| z \|_{\mathcal{Z}_{0}} = C  \| Qz\|_{\mathfrak{M}_{0}} \cdot \| z \|_{\mathcal{Z}_{0}}.
	\end{equation}
	In virtue of Lemma \ref{LEM: OperatorDDelay}, we have $\mathfrak{M}_{v_{0}} = \mathfrak{M}_{0} + Dv_{0}$ and, consequently, any $z \in \mathfrak{M}_{v_{0}}$ can be written as $z = z_{0} + Dv_{0}$ for some $z_{0} \in \mathfrak{M}_{0}$. Therefore, for such $z$ the ``orthogonality condition'' from \eqref{EQ: OrthogonalityCond} can be written as
	\begin{equation}
	(h,Qz_{0} + QDv_{0})_{\mathcal{Z}} = 0 \text{ for all } h \in \mathfrak{M}_{0}, 
	\end{equation}
	or, equivalently, as
	\begin{equation}
	\label{EQ: ProjectedEquationMinima}
	Qz_{0} = -QDv_{0}.
	\end{equation}
	Since $Q$ is bounded, its restriction $Q_{0} := \restr{Q}{\mathfrak{M}_{0}} \colon \mathfrak{M}_{0} \to \mathfrak{M}_{0}$ is bounded. From \eqref{EQ: RieszRepresentEquality} it is clear that
	\begin{equation}
	\inf_{z \in \mathfrak{M}_{0}} \frac{(z,Q_{0}z)_{\mathcal{Z}}}{(z,z)_{\mathcal{Z}}} = \alpha_{1} > 0.
	\end{equation}
	In particular, $Q_{0}$ is a bounded self-adjoint operator in $\mathfrak{M}_{0}$ and its quadratic form is coercive. Therefore, the Lax-Milgram theorem guarantees that $Q_{0}^{-1}$ is everywhere defined and bounded. So, \eqref{EQ: ProjectedEquationMinima} has a unique solution $z_{0} = -Q_{0}^{-1} Q D v_{0} $. Since \eqref{EQ: FormCond} is also satisfied, the optimal process is now given by 
	\begin{equation}
		\label{EQ: OptimalControlOperatorDefinition}
		z = -Q_{0}^{-1}QDv_{0} + Dv_{0} =: Tv_{0}
	\end{equation}
	Clearly, $T$ is a bounded linear operator from $\mathbb{E}_{0}$ to $\mathcal{Z}_{0}$. The proof is finished.
\end{proof}

The next lemma shows that the optimal cost is given by a quadratic form in $\mathbb{E}_{0}$.
\begin{lemma}
	\label{LEM: OperatorP}
	Let the assumptions of Lemma \ref{LEM: OptimizationOnAffineSubspaces} hold. Then there exists $P \in \mathcal{L}(\mathbb{E}_{0};\mathbb{E}^{*}_{0})$ such that for all $v_{0} \in \mathbb{E}_{0}$ we have
	\begin{equation}
	\langle v_{0}, Pv_{0} \rangle = \mathcal{J}_{\mathcal{F}}( v^{0}(\cdot;v_{0}), \xi^{0}(\cdot;v_{0}) ),
	\end{equation}
	where $(v^{0}(\cdot;v_{0}), \xi^{0}(\cdot;v_{0}) \in \mathfrak{M}_{v_{0}}$ is the unique optimal process through $v_{0}$. Moreover, $P$ is symmetric in the sense that $\langle v, Pw \rangle = \overline{ \langle w, Pv \rangle }$ for all $v,w \in \mathbb{E}_{0}$.
\end{lemma}
\begin{proof}
	Let $z=z(\cdot) \in \mathcal{Z}_{0}$ be fixed. Consider the map
	\begin{equation}
	\label{EQ: LinearBanachFunctional}
	\mathbb{E}_{0} \ni v_{0} \mapsto \mathcal{L}_{\mathcal{W}}(Tv_{0},z),
	\end{equation}
	where $T$ and $\mathcal{L}_{\mathcal{W}}$ are defined in \eqref{EQ: OptimalControlOperatorDefinition} and \eqref{EQ: LinearPartIntegralFunctional} respectively. Clearly, \eqref{EQ: LinearBanachFunctional} defines a continuous linear functional on $\mathbb{E}_{0}$ which we denote by $\widetilde{P}z$. Then $\widetilde{P}$ is a linear operator\footnote{Recall that the multiplication by a scalar in $\mathbb{E}^{*}_{0}$ is given by the multiplication of functional's values by scalar's complex-conjugate.} from $\mathcal{Z}_{0}$ to $\mathbb{E}^{*}_{0}$. Since $\mathcal{L}_{\mathcal{W}}$ is a bounded symmetric sesquilinear form on $\mathcal{Z}_{0}$ due to \nameref{DESC:QF}, $\widetilde{P}$ is bounded. Thus, for $z=Tv_{0}$ we have
	\begin{equation}
	\langle v_{0}, \widetilde{P} T v_{0} \rangle = \mathcal{L}_{\mathcal{W}}(Tv_{0} , Tv_{0}) = \mathcal{J}_{\mathcal{F}}( v^{0}(\cdot;v_{0}), \xi^{0}(\cdot;v_{0}) ).
	\end{equation}
	Therefore, $P:= \widetilde{P} T$ satisfies the required property. The symmetricity $\langle v, Pw \rangle = \overline{ \langle w, Pv \rangle }$ follows from the symmetricity of $\mathcal{L}_{\mathcal{W}}$. The proof is finished.
\end{proof}

Now we recall the values $\alpha_{2}$ from \eqref{EQ: Alpha2} and $\alpha_{3}$ from \eqref{EQ: Alpha3}. For convenience, remind that $\alpha_{2} \in \mathbb{R}$ is the best possible (i.e. the exact upper bound) constant such that the inequality $\mathcal{F}(v,\xi) \geq \alpha_{2}(|v|^{2}_{\mathbb{W}} + |\xi|^{2}_{\Xi})$ is valid for all $v \in \mathcal{D}(A)$, $\xi \in \Xi$ and $\omega \in \mathbb{R}$ such that $i\omega v = Av + B\xi$. Remind also that $\alpha_{3} \in \mathbb{R}$ is determined similarly from the inequality $\mathcal{F}(-(A-i\omega I)^{-1}B\xi,\xi) \geq \alpha_{3} |\xi|^{2}_{\Xi}$ required to be satisfied for all $\xi \in \Xi$ and $\omega \in \mathbb{R}$. Note that $\alpha_{3}$ is considered under the condition that the operator $A$ does not have spectrum on the imaginary axis and the operators $(A-i\omega I)^{-1}B$ are bounded in the norm of $\mathcal{L}(\Xi;\mathbb{H})$ uniformly in $\omega \in \mathbb{R}$ and in this case we say that $\alpha_{3}$ is well-defined.

The following lemma is an analog of Lemma 2 from \cite{Likhtarnikov1976}. Note that the proof of its second part is based on the use of the Fourier transform in $L_{2}$. Along with the first part it makes Theorem \ref{TH: OptimalProcessDelay} and its consequences very useful in applications and justifies the name ``Frequency Theorem'' coined by V.A.~Yakubovich.
\begin{lemma}
	\label{EQ: AlphaEquivalenceLemma}
	1). Suppose that $\alpha_{3}$ from \eqref{EQ: Alpha3} is well-defined. Then $\alpha_{3}>0$ if and only if $\alpha_{2} > 0$.
	
	2). Under \nameref{DESC:FT} $\alpha_{2}>0$ implies $\alpha_{1}>0$.
\end{lemma}
\begin{proof}
	1). It is obvious that $\alpha_{2} \leq \alpha_{3}$. So, $\alpha_{2}>0$ implies $\alpha_{3} > 0$. Put $\beta := 1 + \sup_{\omega \in \mathbb{R}} \|(A-i\omega)^{-1}B\|$, where the norm is taken in $\mathcal{L}(\Xi;\mathbb{W})$. It is clear that $\alpha_{2} \geq \beta^{-1} \alpha_{3}$ that proves item 1) of the lemma.
	
	2). Let us extend any process $(v(\cdot),\xi(\cdot)) \in \mathfrak{M}_{0}$ to the negative semi-axis $(-\infty,0]$ by zero and consider its Fourier transform denoted by $(\hat{v}(\cdot), \hat{\xi}(\cdot) ) \in L_{2}(\mathbb{R};\mathbb{H} \times \Xi)$. Note that since $v(\cdot) \in L_{2}(0,+\infty;\mathbb{W})$ and the embedding $\mathbb{W} \subset \mathbb{H}$ is continuous, we also have $\hat{v}(\cdot) \in L_{2}(\mathbb{R};\mathbb{W})$ and the Parseval identity in that space. By Lemma 11 from \cite{LouisWexler1991}, we have $\hat{v}(\omega) \in \mathcal{D}(A)$ for almost all $\omega \in \mathbb{R}$ and
	\begin{equation}
	\label{EQ: FourierTransformEquation}
	i\omega \hat{v}(\omega) = A \hat{v}(\omega) + B\hat{\xi}(\omega).
	\end{equation}
	From \nameref{DESC:FT}, the definition of $\alpha_{2}$ and the Parseval identity we obtain
	\begin{equation}
	\label{EQ: FourierTransformFirstEquation}
	\begin{split}
	\mathcal{J}_{\mathcal{F}}(v(\cdot),\xi(\cdot)) = \int_{-\infty}^{+\infty}\mathcal{F}(\hat{v}(\omega),\hat{\xi}(\omega))d\omega \geq \alpha_{2} \int_{-\infty}^{+\infty} \left(|\hat{v}(\omega)|^{2}_{\mathbb{W}} + |\hat{\xi}(\omega)|^{2}_{\Xi}\right) = \alpha_{2} \int_{0}^{+\infty}\left(|v(t)|^{2}_{\mathbb{W}} + |\xi(t)|^{2}_{\Xi}\right)dt
	\end{split}
	\end{equation}
	that finishes the proof.
\end{proof}

Below we shall see that $\alpha_{1}>0$ implies $\alpha_{2}>0$. For this we have to establish the following lemma.

\begin{lemma}
	\label{LEM: ConstantDifferentiationProperty}
	Let $f(\cdot)$ be a twice continuously differentiable $\mathbb{H}$-valued function on $[0,\varepsilon]$ for some $\varepsilon>0$. Suppose that $v_{0} \in \mathcal{D}(A)$ is such that $Av_{0} + f(0) \in \mathcal{D}(A)$. Then for the classical solution $v(\cdot)$ on $[0,\varepsilon]$ of
	\begin{equation}
	\label{EQ: InhomogeneousConstantEquation}
	\dot{v}(t) = A v(t) + f(t)
	\end{equation}
	with $v(0)=v_{0}$ we have
	\begin{equation}
	\label{EQ: AuxiliaryDifferentiationLemmaKrein}
	\lim\limits_{h \to 0+}\frac{v(h)-v(0)}{h} = v'(0)=Av(0)+f(0),
	\end{equation}
	where the limit exists in $\mathcal{D}(A)$ endowed with the graph norm.
\end{lemma}
\begin{proof}
	Since $[0,\varepsilon] \ni t \mapsto f(t) \in \mathbb{H}$ is continuously differentiable, we have that $v(\cdot)$ is a continuously differentiable $\mathbb{H}$-valued function on $[0,\varepsilon]$ with $v(t) \in \mathcal{D}(A)$ and $v(t)$ satisfies \eqref{EQ: InhomogeneousConstantEquation} for all $t \in [0,\varepsilon]$ (see, for example, Theorem 6.5, Chapter I in \cite{Krein1971}). In particular, \eqref{EQ: AuxiliaryDifferentiationLemmaKrein} holds in the norm of $\mathbb{H}$. It remains to show that $A((v(h) - v(0))/h)$ tends to $A(Av(0)+f(0))$ in $\mathbb{H}$ as $h \to 0+$. Let us suppose, for simplicity of notation, that $0$ is a regular point for $A$, i.e. $0$ does not belong to the spectrum of $A$. Otherwise, we can consider $A + \rho I$ with a proper $\rho \in \mathbb{R}$ and the solution $v_{\rho}(t) = e^{\rho t}v(t)$. From
	\begin{equation}
	v(h) = G(h)v_{0} + \int_{0}^{h} G(h-s)f(s)ds
	\end{equation}
	and the identity (see equation (6.8) on p. 133 in \cite{Krein1971})
	\begin{equation}
	\begin{split}
    \int_{0}^{h}G(h-s)f(s)ds = G(h)A^{-1}f(0) + \int_{0}^{h} G(h-s)A^{-1}f'(s)ds - A^{-1}f(h)
	\end{split}
	\end{equation}
	we have
	\begin{equation}
	A\left( \frac{v(h)-v(0)}{h} \right) = \frac{G(h)-I}{h} ( Av(0) + f(0) ) + \int_{0}^{h}\frac{G(h-s) - I}{h} f'(s)ds,
	\end{equation}
	where the first term in the right-hand side tends to $A(Av(0)+f(0))$ as $h \to 0+$ since $Av(0)+f(0) \in \mathcal{D}(A)$ and the second term tends to zero since $f$ is twice continuously differentiable. The proof is finished.
\end{proof}

Now we can give a proof of Theorem \ref{TH: OptimalProcessDelay}.
\begin{proof}[Proof of Theorem \ref{TH: OptimalProcessDelay}]
	We have already proved the statement from item 1) of the theorem in Lemma \ref{LEM: OptimizationOnAffineSubspaces} and Lemma \ref{LEM: OperatorP} by modulo the Lyapunov inequality \eqref{EQ: OperatorSolutionsInequality}.
	
	Let us show \eqref{EQ: OperatorSolutionsInequality}. For this let $v(t)=v(t;v_{0},\xi)$, where $t \in [0,T]$, be the solution to \eqref{EQ: ControlSystem} with $v(0)=v_{0} \in \mathbb{E}_{0}$ and $\xi(\cdot) \in L_{2}(0,T;\Xi)$ which additionally satisfies $v(\cdot) \in C([0,T];\mathbb{E}_{0})$. Consider the process $(\widetilde{v}(\cdot),\widetilde{\xi}(\cdot)) \in \mathfrak{M}_{v_{0}}$, where
	\begin{equation}
	\widetilde{v}(t)= \begin{cases}
	v(t), &\text{ if } t \in [0,T),\\
	v^{0}(t-T,v(T)), &\text{ if } t \geq T.
	\end{cases}
	\end{equation}
	and
	\begin{equation}
	\widetilde{\xi}(t)= \begin{cases}
	\xi(t), &\text{ if } t \in [0,T),\\
	\xi^{0}(t-T,v(T)), &\text{ if } t \geq T.
	\end{cases}
	\end{equation}
	From the inequality $\mathcal{J}_{\mathcal{F}}(v^{0}(\cdot;v_{0}),\xi^{0}(\cdot;v_{0})) \leq \mathcal{J}_{\mathcal{F}}(\widetilde{v}(\cdot),\widetilde{\xi}(\cdot))$ and Lemma \ref{LEM: OperatorP} we have
	\begin{equation}
	V(v(T)) - V(v_{0}) + \int_{0}^{T}\mathcal{F}(v(t),\xi(t)) dt \geq 0.
	\end{equation}
	Thus, item 1) of the theorem is proved. 
	
	For item 2), in virtue of Lemma \ref{EQ: AlphaEquivalenceLemma}, it only requires to show that $\alpha_{1}>0$ implies $\alpha_{2}>0$. Indeed, suppose $\alpha_{1}>0$. Consider the quadratic form $\mathcal{F}_{\delta}(v,\xi):=\mathcal{F}(v,\xi)-\delta (|v|^{2}_{\mathbb{W}}+|\xi|_{\Xi}^{2})$. It is clear that for the form $\mathcal{F}_{\delta}$ a condition analogous to $\alpha_{1}>0$ will be satisfied if $\delta>0$ is chosen sufficiently small. Therefore, there exists an operator $P_{\delta} \in \mathcal{L}(\mathbb{E}_{0},\mathbb{E}^{*}_{0})$, which is symmetric and for $h>0$ we have
	\begin{equation}
	\label{EQ: MonotoneIneqAlpha1ImpliesA2}
	\langle v(h),P_{\delta}v(h) \rangle - \langle v_{0}, P_{\delta}v_{0} \rangle + \int_{0}^{h} \mathcal{F}(v(s),\xi(s)) \geq \delta \int_{0}^{h}(|v(s)|^{2}_{\mathbb{W}}+|\xi(s)|_{\Xi}^{2})ds,
	\end{equation}
	where $v(\cdot)=v(\cdot;v_{0},\xi)$ and $\xi(\cdot) \in L_{2}(0,h;\Xi)$ are such that $v(\cdot) \in C([0,h];\mathbb{E}_{0})$. Let us choose $\xi(\cdot) \equiv \xi_{0}$ for a fixed $\xi_{0} \in \Xi$ and $v_{0} \in \mathcal{D}(A)$ such that $Av_{0} + B\xi_{0} \in \mathcal{D}(A)$. Then from Lemma \ref{LEM: ConstantDifferentiationProperty} it follows that $v(\cdot)$ is $\mathcal{D}(A)$-differentiable at $0$ and, in virtue of the continuous embedding $\mathcal{D}(A) \subset \mathbb{E}_{0}$, it is also $\mathbb{E}_{0}$-differentiable. Dividing \eqref{EQ: MonotoneIneqAlpha1ImpliesA2} by $h>0$ and letting $h$ tend to zero, we get
	\begin{equation}
	2\operatorname{Re}\langle Av_{0}+B\xi_{0}, P_{\delta}v_{0}\rangle + \mathcal{F}(v_{0},\xi_{0}) \geq \delta ( |v_{0}|^{2}_{\mathbb{W}}+|\xi_{0}|_{\Xi}^{2} ).
	\end{equation}
	If we choose $v_{0}$ and $\xi_{0}$ such that $i \omega v_{0} = A v_{0} + B \xi_{0}$ for some $\omega \in \mathbb{R}$, then we immediately have
	\begin{equation}
	\label{EQ: TH11Alpha2PositiveEnd}
	\mathcal{F}(v_{0},\xi_{0}) \geq \delta ( |v_{0}|^{2}_{\mathbb{W}}+|\xi_{0}|_{\Xi}^{2} ),
	\end{equation}
	which in turn implies that $\alpha_{2}>0$. The proof is finished.
\end{proof}
\subsection{Proof of Theorem \ref{TH: FreqThRelaxedDelay}}
\label{SUBSEC: RelaxingControllability}

Before we start giving a proof of Theorem \ref{TH: FreqThRelaxedDelay}, let us consider the modification of \eqref{EQ: ControlSystem} associated with the extended control pair $(A,\widehat{B})$, i.e.
\begin{equation}
\label{EQ: ModifiedControlSystem}
\dot{v}(t) = Av(t) + B\xi(t) + \eta(t) = Av(t) + \widehat{B}\zeta(t),
\end{equation}
where $\zeta(\cdot)=(\xi(\cdot),\eta(\cdot)) \in L_{2}(0,T;\widehat{\Xi})$ is the extended control function and $\widehat{\Xi} = \Xi \times \mathbb{W}$ is the extended control space with $\widehat{B}\zeta := B\xi + \eta$ for each $\zeta=(\xi,\eta) \in \widehat{\Xi}$. 

Let us show that if the semigroup $G(t)$ is potentially $(L_{2},\mathbb{W})$-controllable in $\mathbb{E}_{0}$, then the pair $(A,\widehat{B})$ is $(L_{2},\mathbb{W})$-controllable in $\mathbb{E}_{0}$ no matter what $B$. Indeed, by definition, for any $v_{0} \in \mathbb{E}_{0}$ there exists $\varkappa \in \mathbb{R}$ such that $v_{\varkappa}(t) = e^{\varkappa t}G(t)v_{0}$ satisfies $v_{\varkappa}(\cdot) \in L_{2}(0,\infty;\mathbb{W})$. But then $v(\cdot) = v_{\varkappa}(\cdot)$ satisfies \eqref{EQ: ModifiedControlSystem} with the control functions $\xi(\cdot) \equiv 0$ and $\eta(\cdot) = \varkappa v(\cdot)$ and, consequently, the space of processes through $v_{0}$ is not empty.

For $\gamma > 0$ let us consider the quadratic form
\begin{equation}
\mathcal{F}_{\gamma}(v,\xi,\eta) := \mathcal{F}(v,\xi) + \gamma |\eta|^{2}_{\mathbb{W}}
\end{equation}
and the value
\begin{equation}
\alpha_{\gamma} := \inf \frac{ \mathcal{F}_{\gamma}(v,\xi,\eta) }{|v|^{2}_{\mathbb{W}}+|\xi|^{2}_{\Xi} + |\eta|^{2}_{\mathbb{W}}},
\end{equation}
where the infimum is taken over all $\omega \in \mathbb{R}$, $v \in \mathcal{D}(A)$, $\xi \in \Xi$ and $\eta \in \mathbb{W}$ such that $i \omega v = A v + B \xi + \eta$. 

Recall that for Theorem \ref{TH: FreqThRelaxedDelay} we assume $(A,\widehat{B})$ to satisfy \nameref{DESC:QF} and \nameref{DESC:FT} w.r.t. $\widehat{\mathcal{F}}(v,\xi,\eta) = \mathcal{F}(v,\xi)$. From this it is clear that these properties are also satisfied for $\mathcal{F}_{\gamma}$ since the added term is bounded in $\mathbb{W}$.  So, under the inequality $\alpha_{\gamma}>0$ we can apply Theorem \ref{TH: OptimalProcessDelay} to the pair $(A,\widehat{B})$ and the form $\mathcal{F}_{\gamma}$. 

Recall also that we require in \nameref{DESC:RES1} that $\mathcal{F}$ is defined and bounded on $\mathbb{E} \times \Xi$, where $\mathbb{E}$ is a Banach space which is continuously embedded into $\mathbb{H}$ and identified with its image such that the inclusions in $\mathcal{D}(A) \subset \mathbb{E} \subset \mathbb{W}$ become continuous embeddings. Then assumptions \nameref{DESC:RES2} and \nameref{DESC:RES3}, under the condition that $A$ does not have spectrum on the imaginary axis, require the norms of $(A-i\omega I)^{-1}$ in $\mathcal{L}(\mathbb{W};\mathbb{E})$ and $(A-i\omega I)^{-1}B$ in $\mathcal{L}(\Xi;\mathbb{E})$ respectively to be bounded uniformly in $\omega \in \mathbb{R}$.

Now we can proceed to the proof. Its arguments are similar to the corresponding part of our adjacent work \cite{Anikushin2020FreqParab}, which in its turn develops the ideas of A.V.~Proskurnikov \cite{Proskurnikov2015}.
\begin{proof}[Proof of Theorem \ref{TH: FreqThRelaxedDelay}]
	We have to prove that $\alpha_{2}>0$ (given by \eqref{EQ: Alpha2} for $\mathcal{F}$) implies that $\alpha_{\gamma} > 0$, where $\gamma>0$ is sufficiently large. In this case, we may apply Theorem \ref{TH: OptimalProcessDelay} to the pair $(A,\widehat{B})$ and the form
	\begin{equation}
		\mathcal{F}^{\delta}_{\gamma}(v,\xi,\eta) := \mathcal{F}_{\gamma}(v,\xi,\eta) - \delta (|v|^{2}_{\mathbb{W}} + |\xi|^{2}_{\Xi} + |\eta|^{2}_{\mathbb{W}})
	\end{equation} 
    for a sufficiently small $\delta>0$ to get an operator $P_{\gamma} \in \mathcal{L}(\mathbb{E}_{0};\mathbb{E}^{*}_{0})$ such that for $V(v):=\langle v, P_{\gamma}v \rangle$ the inequality 
	\begin{equation}
	V(v(T)) - V(v_{0}) + \int_{0}^{T} \mathcal{F}_{\gamma}(v(t),\xi(t),\eta(t))dt \geq \delta \int_{0}^{T}( |v(t)|^{2}_{\mathbb{W}} + |\xi(t)|^{2}_{\Xi} + |\eta(t)|^{2}_{\mathbb{W}})dt
	\end{equation}
	is satisfied for any solution $v(t)=v(t;v_{0},(\xi(\cdot),\eta(\cdot)))$ of \eqref{EQ: ModifiedControlSystem} which additionally satisfies $v(\cdot) \in C([0,T];\mathbb{E}_{0})$. Putting $\eta(t) \equiv 0$, we get the inequality \eqref{EQ: RelaxedFreqThLyapunivInequality} for the form $\mathcal{F}$ and the pair $(A,B)$. So, $P:=P_{\gamma}$ is the desired operator.
	
	In the same way as in the proof of Theorem \ref{TH: OptimalProcessDelay}, where \eqref{EQ: TH11Alpha2PositiveEnd} is derived from \eqref{EQ: MonotoneIneqAlpha1ImpliesA2}, one can show that the inequality $\alpha_{2}>0$ is necessary for the existence of $P$ as in \eqref{EQ: RelaxedFreqThLyapunivInequality}. 
	
	Now let us assume that $\alpha_{2}>0$. From this and \nameref{DESC:RES3} there exists $\varepsilon>0$ such that
	\begin{equation}
	\label{EQ: RelaxationStrongerIneq}
	\mathcal{F}(v,\xi) \geq \varepsilon ( \|v\|^{2}_{\mathbb{E}} + |\xi|^{2}_{\Xi} )
	\end{equation}
	is satisfied for all $v = -(A - i\omega I)^{-1}B\xi$ corresponding to arbitrary $\omega \in \mathbb{R}$ and $\xi \in \Xi$. For any $\eta \in \mathbb{W}$ and $\omega \in \mathbb{R}$ we define vectors $v_{\omega}(\eta) \in \mathcal{D}(A)$ and $\xi_{\omega}(\eta) \in \Xi$ such that we have
	\begin{equation}
	i \omega v_{\omega}(\eta) = A v_{\omega}(\eta) + B \xi_{\omega}(\eta) + \eta.
	\end{equation}
	and there are constants $M_{1}>0, M_{2}>0$ such that
	\begin{equation}
	\label{EQ: RelaxationFunctionsConstants}
	\| v_{\omega}(\eta) \|_{\mathbb{E}} \leq M_{1} |\eta|_{\mathbb{W}} \text{ and } |\xi_{\omega}(\eta)|_{\Xi} \leq M_{2} |\eta|_{\mathbb{W}}.
	\end{equation}
	For example, in virtue of \nameref{DESC:RES2}, one can take $v_{\omega}(\eta):= -(A - i\omega I)^{-1}\eta$ and $\xi_{\omega}(\eta) = 0$. 
	
	Now let $\omega \in \mathbb{R}$, $v \in \mathcal{D}(A)$, $\xi \in \Xi$ and $\eta \in \mathbb{W}$ be given such that
	\begin{equation}
	i \omega v = A v + B \xi + \eta.
	\end{equation}
	Putting $\delta v := v - v_{\omega}(\eta)$, $\delta \xi := \xi - \xi_{\omega}(\omega)$, we get the identity
	\begin{equation}
	\label{EQ: RelaxationDeltaFreq}
	i\omega \delta v = A \delta v + B \delta \xi
	\end{equation}
	and the representation
	\begin{equation}
	\mathcal{F}_{\gamma}(v,\xi,\eta) = \mathcal{F}(v_{\omega}(\eta),\xi_{\omega}(\eta)) + 2\operatorname{Re}\mathcal{W}((\delta v,\delta \xi), (v_{\omega}(\eta),\xi_{\omega}(\eta)) ) + \mathcal{F}(\delta v,\delta \xi) + \gamma |\eta|^{2}_{\mathbb{W}}.
	\end{equation}
	In virtue of \eqref{EQ: RelaxationDeltaFreq} and \eqref{EQ: RelaxationStrongerIneq} we get that $\mathcal{F}(\delta v,\delta \xi) \geq \varepsilon ( \|\delta v\|^{2}_{\mathbb{E}} + |\delta\xi|^{2}_{\Xi} )$. From \eqref{EQ: RelaxationFunctionsConstants} we have constants $\widetilde{M}_{1}>0$, $\widetilde{M}_{2}>0$ such that
	\begin{equation}
	\label{EQ: RelaxingFormConstants}
	|\mathcal{F}(v_{\omega}(\eta),\xi_{\omega}(\eta))| \leq \widetilde{M}_{1}|\eta|^{2}_{\mathbb{W}} \text{ and } |\mathcal{W}((\delta v, \delta \xi), (v_{\omega}(\eta),\xi_{\omega}(\eta)))| \leq \widetilde{M}_{2} |\delta\xi|_{\Xi} \cdot |\eta|_{\mathbb{W}}.
	\end{equation}
	Let us choose numbers $\gamma_{1} := \widetilde{M}_{1}$ and $\gamma_{2} > 0$ such that $\gamma_{2} x^{2} - \widetilde{M}_{2} \beta x + \beta^2 \varepsilon/2 \geq 0$ for all $x \in \mathbb{R}$ and all $\beta \geq 0$. For this, it is sufficient to satisfy $2 \gamma_{2} \varepsilon \geq \widetilde{M}^{2}_{2}$. Now define $\gamma := \gamma_{1} + \gamma_{2} + \gamma_{3}$ for any $\gamma_{3}>0$. From this and \eqref{EQ: RelaxingFormConstants} we have
	\begin{equation}
	\mathcal{F}_{\gamma}(v,\xi,\eta) \geq \frac{\varepsilon}{2} \cdot \left( \|\delta v\|^{2}_{\mathbb{E}} + |\delta\xi|^{2}_{\Xi} \right) + \gamma_{3} |\eta|^{2}_{\mathbb{W}}.
	\end{equation}
	From the obvious inequalities $\|v\|_{\mathbb{E}} \leq M_{1}|\eta|_{\mathbb{W}} + \|\delta v\|_{\mathbb{E}}$ and $|\xi|_{\Xi} \leq M_{2} |\eta|_{\mathbb{W}} + |\delta \xi|_{\Xi}$ and the continuous embedding $\mathbb{E} \subset \mathbb{W}$ we get that	 $\alpha_{\gamma} > 0$. The proof is finished.
\end{proof}
\subsection{Realification of the operator $P$ and its sign properties}
\label{SUBSEC: Realification}
In practice, we often work in the context of real spaces and operators. Suppose a real quadratic form $\mathcal{F}$ is given and consider its Hermitian extension to the complexifications of the spaces given by $\mathcal{F}^{\mathbb{C}}(v_{1}+iv_{2}, \xi_{1}+i\xi_{2}):=\mathcal{F}(v_{1},\xi_{2}) + \mathcal{F}(v_{2},\xi_{2})$ for $v_{1},v_{2} \in \mathcal{D}(A)$, $\xi_{1},\xi_{2} \in \Xi$. From Theorem \ref{TH: OptimalProcessDelay} or Theorem \ref{TH: FreqThRelaxedDelay} after taking complexifications\footnote{Here and below, the complexification of a real vector space $\mathbb{E}_{0}$ is denoted by $\mathbb{E}_{0}^\mathbb{C}$. For convenience, we use the same notations for the corresponding complexifcations of the operators.}, we get an operator $P \in \mathcal{L}((\mathbb{E}_{0})^{\mathbb{C}};(\mathbb{E}^{*}_{0})^{\mathbb{C}})$. From this one can use a realification procedure to obtain an operator between the real spaces $\mathbb{E}_{0}$ and $\mathbb{E}^{*}_{0}$.

Recall that the semigroup $G(t)$, where $t \geq 0$, is called \textit{potentially} $(L_{2},\mathbb{W})$-\textit{controllable} in $\mathbb{E}_{0}$ if for any $v_{0} \in \mathbb{E}_{0}$ there exists $\varkappa \in \mathbb{R}$ such that for $v_{\varkappa}(t) = e^{\varkappa t}G(t)v_{0}$ we have $v_{\varkappa}(\cdot) \in L_{2}(0,\infty;\mathbb{W})$

We give a real formulation of Theorem \ref{TH: FreqThRelaxedDelay}, which is convenient for further applications.
\begin{theorem}
	\label{TH: FrequencyTheoremRealOperator}
	In the context of real spaces and operators, let the $C_{0}$-semigroup $G(t)$ generated by $A$ in $\mathbb{H}$ be potentially $(L_{2},\mathbb{W})$-controllable in $\mathbb{E}_{0}$ suppose \nameref{DESC:RES} is satisfied for the complexificated pair $(A,B)$ and the Hermitian extension $\mathcal{F}^{\mathbb{C}}$ of $\mathcal{F}$. Moreover, let \nameref{DESC:QF} and \nameref{DESC:FT} be satisfied for the complexificated extended control pair $(A,\widehat{B})$ as in \eqref{EQ: ModifiedControlSystem} and the extended Hermitian form $\widehat{\mathcal{F}}(v,\xi,\eta) := \mathcal{F}^{\mathbb{C}}(v,\xi)$. Let the frequency inequality
	\begin{equation}
		\label{EQ: FrequencyConditionNegative}
		\sup_{\omega \in \mathbb{R}}\sup_{\xi \in \Xi^{\mathbb{C}}} \frac{\mathcal{F}^{\mathbb{C}}(-(A-i\omega I)^{-1}B\xi,\xi)}{|\xi|^{2}_{\Xi^{\mathbb{C}}}} < 0
	\end{equation}
	be satisfied. Then there exists a symmetric operator $P \in \mathcal{L}(\mathbb{E}_{0};\mathbb{E}^{*}_{0})$ and a number $\delta>0$ such that for $V(v):=\langle v, Pv\rangle$ and all $v_{0} \in \mathbb{E}_{0}$ we have
	\begin{equation}
		\label{EQ: RealificationInequalityModified}
		V(v(T)) - V(v_{0}) + \int_{0}^{T} \mathcal{F}(v(t),\xi(t))dt \leq -\delta \int_{0}^{T}\left(|v(t)|^{2}_{\mathbb{W}} + |\xi(t)|^{2}_{\Xi} \right)dt,
	\end{equation}
	where $v(t)=v(t;v_{0},\xi)$ is the mild solution to \eqref{EQ: ControlSystem} with arbitrary $v(0)=v_{0} \in \mathbb{E}_{0}$ and $\xi(\cdot) \in L_{2}(0,T;\Xi)$ which additionally satisfies $v(\cdot) \in C([0,T];\mathbb{E}_{0})$ and the integral with $\mathcal{F}$ is understood in the sense of \nameref{DESC:QF}.
\end{theorem}
\begin{proof}
	From the potential $(L_{2},\mathbb{W})$-controllability in $\mathbb{E}_{0}$ we get that the pair $(A,\widehat{B})$ is $(L_{2},\mathbb{W})$-controllable in $\mathbb{E}_{0}$ (see below \eqref{EQ: ModifiedControlSystem}). Now the result follows from Theorem \ref{TH: FreqThRelaxedDelay} applied to the form $-\mathcal{F}^{\mathbb{C}}$ and the realification procedure as in \cite{Anikushin2020FreqParab,Anikushin2019+OnCom}. The proof is finished.
\end{proof}

For purposes of applications, it is important to decompose the phase space, which is assumed to be the space $\mathbb{E}_{0}$, into the direct sum of positive and negative subspaces, where the sign is understood in terms of the sign of the operator $P$ or, equivalently, its quadratic form $V(\cdot)$ restricted to each subspace. Such properties can be determined from linear dichotomies and the Lyapunov inequality, which can be obtained from \eqref{EQ: RealificationInequalityModified} in applications. Namely, for a Banach space $\mathbb{E}_{0}$, which is continuously embedded into $\mathbb{H}$, we suppose that $G(t) \colon \mathbb{E}_{0} \to \mathbb{E}_{0}$, where $t \geq 0$, is a $C_{0}$-semigroup in $\mathbb{E}_{0}$. Suppose also that there exists a decomposition of $\mathbb{E}_{0}$ into the direct sum $\mathbb{E}_{0} = \mathbb{E}_{0}^{s} \oplus \mathbb{E}_{0}^{u}$ of two closed subspaces $\mathbb{E}_{0}^{s}$ and $\mathbb{E}_{0}^{u}$ such that for $v_{0} \in \mathbb{E}_{0}^{s}$ we have $G(t)v_{0} \to 0$ in $\mathbb{E}_{0}$ as $t \to +\infty$ and any $v_{0} \in \mathbb{E}_{0}^{u}$ admits a backward extension $v(\cdot)$ (that is a continuous function $v(\cdot) \colon \mathbb{R} \to \mathbb{E}_{0}$ with $v(0)=v_{0}$ and $v(t+s)=G(t)v(s)$ for all $t \geq0$ and $s \in \mathbb{R}$) such that $v(t) \to 0$ in $\mathbb{E}_{0}$ as $t \to -\infty$. We have the following theorem, the proof of which is analogous to Theorem 5 from \cite{Anikushin2020FreqParab}.
\begin{theorem}
	\label{TH: HighRankConesTheorem}
	In the above context, consider any symmetric operator $P \in \mathcal{L}(\mathbb{E}_{0};\mathbb{E}_{0}^{*})$ and its quadratic form $V(v):= \langle v, Pv \rangle$ for $v \in \mathbb{E}_{0}$. Let the Lyapunov inequality
	\begin{equation}
	\label{EQ: LyapunovInequalityIntegralForm}
	V(G(t)v_{0}) - V(v_{0}) \leq -\delta \int_{0}^{T}|v(t)|^{2}_{\mathbb{H}}dt
	\end{equation}
	hold for all $T>0$ and $v_{0} \in \mathbb{E}_{0}$. Then $P$ is positive on $\mathbb{E}_{0}^{s}$ (that is $V(v) > 0$ for all non-zero $v \in \mathbb{E}_{0}^{s}$) and $P$ is negative on $\mathbb{E}_{0}^{u}$ (that is $V(v) < 0$ for all non-zero $v \in \mathbb{E}_{0}^{u}$).
\end{theorem}

\begin{remark}
Under the hypotheses of Theorem \ref{TH: HighRankConesTheorem}, suppose that $\operatorname{dim}\mathbb{E}_{0}^{u} =: j < \infty$. In this case the set $\mathcal{C}_{V}:=\{ v \in \mathbb{E}_{0} \ | \ V(v) \leq 0 \}$ is a quadratic cone of rank $j$. In applications, inequality \eqref{EQ: RealificationInequalityModified} usually contains the Lyapunov inequality \eqref{EQ: LyapunovInequalityIntegralForm} and gives a strong monotonicity of the system with respect to this cone (see Theorem \ref{TH: QuadraticLyapunovFuncDelay} below). In the theory of inertial manifolds such monotonicity is known as the Cone Condition. However, in most of known works the class of considered cones is either a very special class of quadratic cones (as for parabolic equations, see \cite{Zelik2014}) or naive non-quadratic generalizations of this class in the Banach space setting (see \cite{KokschSiegmund2002} and the introduction in \cite{Anikushin2020FreqParab} for a discussion).
\end{remark}
\section{Pointwise measurement operators on embracing spaces and spaces of agalmanated functions}
\label{SEC: LemmaOnEstimatesForLF}

Throughout this section, we fix a real or complex separable Hilbert space\footnote{In fact, the Hilbert space structure is required only when we are studying the Fourier transform. Moreover, the separability is only used in the proof of Lemma \ref{LEM: TwistedFuncUniquenessY}.} $\mathbb{F}$ and a number $\tau > 0$. Suppose $\mathbb{M}_{\gamma}$ is another Hilbert space over the same field and let $\gamma(\theta) \in \mathcal{L}(\mathbb{F};\mathbb{M}_{\gamma})$, where $\theta \in [-\tau,0]$, be an operator-valued function of bounded variation on $[-\tau,0]$. Then with such $\gamma$ we associate a bounded operator $C^{\gamma}$ from $C([-\tau,0];\mathbb{F})$ into $\mathbb{M}_{\gamma}$ given by 
\begin{equation}
	\label{EQ: OperatorCgammaDefinitionNeutral}
	C^{\gamma}\phi := \int_{-\tau}^{0}d\gamma(\theta)\phi(\theta) \text{ for any } \phi \in C([-\tau,0];\mathbb{F}),
\end{equation}
where the integral is understood as the Riemann–Stieltjes integral.

Below we will describe certain spaces of $L_{2}(-\tau,0;\mathbb{F})$-valued functions $\phi(\cdot)$ on $[0,T]$ on which the pointwise action of $C^{\gamma}$ can be understood in the integral sense that agrees with $(C^{\gamma}\phi)(t)=C^{\gamma}\phi(t)$ for continuous $C([-\tau,0];\mathbb{F})$-valued functions of $t \in [0,T]$. We refer to Appendix B from our adjacent work \cite{Anikushin2023Comp} where some parts of the theory are developed in a wider context. In our case, the proofs do not require as many preparations as in \cite{Anikushin2023Comp}, so we give them for the sake of completeness. Moreover, we also complement the results by some theorems on whether $(C^{\gamma}\phi)(\cdot)$ belongs to the Sobolev space $W^{1,2}(0,T;\mathbb{M}_{\gamma})$. Such results are essential to derive the structural Cauchy formula in Theorem \ref{TH: DelayRegCond} describing the structure of solutions to linear inhomogeneous systems associated with neutral delay equations.

\subsection{Embracing spaces and pointwise measurement operators}
\label{SUBSEC: EmbracingSpacesPointwiseMesOp}
We start by introducing, in a sense, the largest space where the pointwise measurement operators can be defined. Roughly speaking, the corresponding functions of time must have the values in $L_{2}(-\tau,0;\mathbb{F})$ for which the delta measurement $\delta_{\theta}$ at each $\theta \in [-\tau,0]$ is a well-defined element of $L_{p}$ over the time interval and some uniform in $\theta$ bound is satisfied. More precisely, for each $-\infty \leq a < b \leq +\infty$ and $p \geq 1$ we define the \textit{embracing space} $\mathcal{E}_{p}(a,b;L_{p}(-\tau,0;\mathbb{F}))$ as the completion of $L_{p}(a,b;C([-\tau,0];\mathbb{F}))$ by the norm
\begin{equation}
	\label{EQ: NormEmbracingSpaceNeutralDelay}
	\| \phi(\cdot) \|_{\mathcal{E}_{p}(a,b;L_{p}(-\tau,0;\mathbb{F}))}:= \sup\limits_{\theta \in [-\tau,0]}\| (\mathcal{I}_{\delta_{\theta}}\phi)(\cdot) \|_{L_{p}(a,b;\mathbb{F})},
\end{equation}
where $(\mathcal{I}_{\delta_{\theta}}\phi)(t) = \phi(t)(\theta)$ for almost all $t \in [a,b]$. For brevity, we will write $\mathcal{E}_{p}(a,b;L_{p})$ to denote the space just introduced.

Note that for each $\phi(\cdot) \in L_{p}(a,b;C([-\tau,0];\mathbb{F}))$ the mapping $[-\tau,0] \ni \theta \mapsto (\mathcal{I}_{\delta_{\theta}}\phi)(\cdot) \in L_{p}(a,b;\mathbb{F})$ is continuous by the Dominated Convergence Theorem or by an approximation argument. In fact, this property characterizes $\mathcal{E}_{p}(a,b;L_{p})$ as follows.
\begin{lemma}
	\label{LEM: EmbracingSpaceDescriptionCLp}
	We have the following
	\begin{enumerate}
		\item[1)] There is an isometric isomorphism
		\begin{equation}
			\label{EQ: EmbracingSpaceCorrespondenceContFunc}
			\mathcal{E}_{p}(a,b;L_{p}(-\tau,0;\mathbb{F})) \ni \phi(\cdot) \mapsto \mathcal{R}_{\phi}(\cdot) \in C([-\tau,0];L_{p}(a,b;\mathbb{F})), 
		\end{equation}
	    where $\mathcal{R}_{\phi}(\theta) := (\mathcal{I}_{\delta_{\theta}}\phi)(\cdot) \text{ for } \theta \in [-\tau,0]$.
	    \item[2)] There is an inclusion of $\mathcal{E}_{p}(a,b;L_{p}(-\tau,0;\mathbb{F}))$ into $L_{p}(a,b;L_{p}(-\tau,0;\mathbb{F}))$ such that
	    \begin{equation}
	    	\| \phi(\cdot) \|_{L_{p}(a,b;L_{p}(-\tau,0;\mathbb{F}))} \leq \tau^{1/p} \cdot \| \phi(\cdot) \|_{\mathcal{E}_{p}(a,b;L_{p}(-\tau,0;\mathbb{F}))}.
	    \end{equation}	    
	\end{enumerate} 
\end{lemma}
\begin{proof}
	Indeed, consider any fundamental in $\mathcal{E}_{p}(a,b;L_{p}(-\tau,0;\mathbb{F}))$ sequence $\phi_{k}(\cdot) \in L_{p}(a,b;C([-\tau,0];\mathbb{F}))$, where $k=1,2,\ldots$. Clearly, this is equivalent to that $\mathcal{R}_{\phi_{k}}(\cdot)$ is fundamental in $C([-\tau,0];L_{p}(a,b;\mathbb{F}))$. Since such elements as $\mathcal{R}_{\phi}(\cdot)$ are dense in the latter space, we obtain item 1) of the theorem.
	
	For item 2) we take $\phi(\cdot) \in L_{p}(a,b;C([-\tau,0];\mathbb{F}))$ and note that $\phi(t)(\theta) = \mathcal{R}_{\phi}(\theta)(t)$ must hold for almost all $(t,\theta) \in (a,b) \times [-\tau,0]$. Then the Fubini theorem gives
	\begin{equation}
		\int_{a}^{b}\|\phi(t)\|^{p}_{L_{p}(-\tau,0;\mathbb{F})}dt = \int_{a}^{b}dt\int_{-\tau}^{0}|\mathcal{R}_{\phi}(\theta)(t)|^{p}_{\mathbb{F}}d\theta = \int_{-\tau}^{0}d\theta\int_{a}^{b}|(\mathcal{I}_{\delta_{\theta}}\phi)(t)|^{p}_{\mathbb{F}}dt \leq \tau \| \phi(\cdot) \|^{p}_{\mathcal{E}_{p}(a,b;L_{p}(-\tau,0;\mathbb{F}))}.
	\end{equation}
    The proof is finished.
\end{proof}

Now we define the pointwise measurement operator associated with $C^{\gamma}$ from \eqref{EQ: OperatorCgammaDefinitionNeutral} on the embracing spaces. This is contained in the following theorem.
\begin{theorem}
	\label{TH: EmbracingSpaceFucntionalIGammaNeutral}
	Let $\gamma$ and $C^{\gamma}$ be as in \eqref{EQ: OperatorCgammaDefinitionNeutral} and $p \geq 1$. Then there exists a bounded linear operator
	\begin{equation}
		\mathcal{I}_{C^{\gamma}} \colon \mathcal{E}_{p}(a,b;L_{p}(-\tau,0;\mathbb{F})) \to L_{p}(a,b;\mathbb{M}_{\gamma})
	\end{equation}
    with the norm not exceeding the total variation $\operatorname{Var}_{[-\tau,0]}(\gamma)$ of $\gamma$ on $[-\tau,0]$ and such that for any $\phi(\cdot) \in L_{p}(a,b;C([-\tau,0];\mathbb{F}))$ we have
    \begin{equation}
    	\label{EQ: OperatorIGammaIdentityContEmbr}
    	(\mathcal{I}_{C^{\gamma}}\phi)(t) = C^{\gamma}\phi(t) \text{ for almost all } t \in (a,b).
    \end{equation}
\end{theorem}
\begin{proof}
	Firstly, note that for $C^{\gamma} = \delta_{\theta}$ being the $\mathbb{F}$-valued $\delta$-functional at some $\theta \in [-\tau,0]$ the operator $\mathcal{I}_{C^{\gamma}} = \mathcal{I}_{\delta_{\theta}}$ is well-defined by definition and its norm does not exceed $1$ (see \eqref{EQ: NormEmbracingSpaceNeutralDelay}).
	
	Now we use a particular pointwise approximation of general $C^{\gamma}$ by operator-linear combinations $\delta$-functionals. Namely, we take for any $k=1, 2, \ldots$, a partition $-\tau = \theta^{(k)}_{0} < \ldots < \theta^{(k)}_{N_{k}} = 0$ of $[-\tau,0]$ by $N_{k}+1$ points such that $\max_{1 \leq l \leq N_{k}}|\theta^{(k)}_{l}-\theta^{(k)}_{l-1}| \to 0$ as $k \to \infty$. Now for $l=1,\ldots,N_{k}$ consider the operators $\alpha^{(k)}_{l}:=\gamma(\theta^{(k)}_{l})-\gamma(\theta^{(k)}_{l-1})$ from $\mathcal{L}(\mathbb{F};\mathbb{M}_{\gamma})$ and the $\delta$-functionals $\delta^{(k)}_{l}:=\delta_{\theta^{(k)}_{l}}$ at $\theta^{(k)}_{l}$. We have
	\begin{equation}
		\label{EQ: ConvergenceDeltaFunc}
		C^{\gamma_{k}} \phi \to C^{\gamma}\phi \text{ for any } \phi \in C([-\tau,0];\mathbb{F}), \text{ where } C^{\gamma_{k}} = \sum_{l=1}^{N_{k}} \alpha^{(k)}_{l} \delta^{(k)}_{l}.
	\end{equation}
	Let $\operatorname{Var}_{[-\tau,0]}(\gamma)$ denote the total variation of $\gamma$ on $[-\tau,0]$. By the definition of $\alpha^{(k)}_{l}$, we have
	\begin{equation}
		\label{EQ: CgammaApproxAlphaVariationEstimate}
		\sum_{l=1}^{N_{k}} \|\alpha^{(k)}_{l}\|_{\mathcal{L}(\mathbb{F};\mathbb{M}_{\gamma})} \leq \operatorname{Var}_{[-\tau,0]}(\gamma).
	\end{equation}
	Combining this with the Fatou lemma, the Minkowski inequality and \eqref{EQ: CgammaApproxAlphaVariationEstimate}, for $\phi(\cdot) \in L_{p}(a,b;C([-\tau,0];\mathbb{F}))$ we obtain
	\begin{equation}
		\begin{split}
			\left(\int_{a}^{b} |C^{\gamma}\phi(t)|^{p}_{\mathbb{M}_{\gamma}} dt \right)^{1/p} \leq \lim_{k \to \infty}\left(\int_{a}^{b}\left|C^{\gamma_{k}}\phi(t)\right|^{p}_{\mathbb{M}_{\gamma}} dt \right)^{1/p} \leq \operatorname{Var}_{[-\tau,0]}(\gamma) \cdot \| \phi(\cdot) \|_{\mathcal{E}_{p}(a,b;L_{p}(-\tau,0;\mathbb{F}))}.
		\end{split}
	\end{equation}
	The proof is finished.
\end{proof}

Suppose that $[a,b] \subset [c,d]$ for $-\infty \leq c \leq a \leq b \leq d \leq +\infty$. Consider the operators  $R^{1}_{T} \colon \mathcal{E}_{p}(c,d;L_{p}) \to \mathcal{E}_{p}(a,b;L_{p})$ and $R^{2}_{T} \colon L_{p}(c,d;\mathbb{M}_{\gamma}) \to L_{p}(a,b;\mathbb{M}_{\gamma})$ that restrict functions from $[c,d]$ to $[a,b]$. From \eqref{EQ: OperatorIGammaIdentityContEmbr} we immediately have the following.
\begin{lemma}
	\label{LEM: CommutativeOperatorIGammaEmbr}
	Under the above notations, the following diagram
	\begin{equation}
		\label{EQ: CommDiagramICgammaAdorned}
		\begin{tikzcd}
			\mathcal{E}_{p}(c,d;L_{p})\ar[d, "R^{1}_{T}"] \ar[r,"\mathcal{I}_{C^{\gamma}}"] 
			& L_{p}(c,d;\mathbb{M}_{\gamma}) \arrow[d,"R^{2}_{T}"] 
			\\
			\mathcal{E}_{p}(a,b;L_{p})  \arrow[r,"\mathcal{I}_{C^{\gamma}}"] 
			& L_{p}(a,b;\mathbb{M}_{\gamma})
		\end{tikzcd}
	\end{equation}
	is commutative. Here the operators $\mathcal{I}_{C^{\gamma}}$ are given by Theorem \ref{TH: EmbracingSpaceFucntionalIGammaNeutral}.
\end{lemma}

From Lemma \ref{LEM: CommutativeOperatorIGammaEmbr} and since $\mathcal{E}_{p}(a,b;L_{p}) \subset \mathcal{E}_{1}(a,b;L_{1})$ for finite $a$ and $b$, we obtain the following.
\begin{corollary}
	\label{EQ: ComputationICgammaEmbracingL1Loc}
	Suppose that $\mathcal{I}_{C^{\gamma}}$ is given by Theorem \ref{TH: EmbracingSpaceFucntionalIGammaNeutral}. Then
	\begin{equation}
		\label{EQ: OperatorCGammaComputationEmbracingL1loc}
		(\mathcal{I}_{C^{\gamma}}\phi)(t) = C^{\gamma}\phi(t) \text{ for almost all } t \in (a,b)
	\end{equation}
    is satisfied for all $\phi(\cdot) \in \mathcal{E}_{p}(a,b;L_{p}) \cap L_{1,loc}(a,b;C([-\tau,0;\mathbb{F}]))$.
\end{corollary}

Now let us consider the space $\mathcal{E}_{p}(a,b;W^{1,p}(-\tau,0;\mathbb{F}))$ or, for brevity, $\mathcal{E}_{p}(a,b;W^{1,p})$ corresponding to such $\phi(\cdot) \in \mathcal{E}_{p}(a,b;L_{p})$ for which in terms of \eqref{EQ: EmbracingSpaceCorrespondenceContFunc} we have $\mathcal{R}_{\phi}(\cdot) \in C([-\tau,0];W^{1,p}(a,b;\mathbb{F}))$. For each such $\phi(\cdot)$ we put $\phi'(\cdot)$ to be the element of $\mathcal{E}_{p}(a,b;L_{p})$ satisfying $\mathcal{R}_{\phi'}(\theta) = \frac{d}{dt}\mathcal{R}_{\phi}(\theta)$ for each $\theta \in [-\tau,0]$, where $\frac{d}{dt}$ denotes the derivative in $W^{1,p}(a,b;\mathbb{F})$. Clearly, $\mathcal{E}_{p}(a,b;W^{1,p})$ can be endowed with the norm
\begin{equation}
	\label{EQ: EmbracingSpaceSobolevNorm}
	\| \phi(\cdot) \|^{p}_{\mathcal{E}_{p}(a,b;W^{1,p})} := \|\phi(\cdot)\|^{p}_{\mathcal{E}_{p}(a,b;L_{p})} + \| \phi'(\cdot) \|^{p}_{\mathcal{E}_{p}(a,b;L_{p})}
\end{equation}
that makes it a Banach space.

We have the following theorem on differentiability of $\mathcal{I}_{C^{\gamma}}$.
\begin{theorem}
	\label{TH: EmbracingSobolevDifferentiabilityTheorem}
	Let $\gamma$ and $C^{\gamma}$ be as in \eqref{EQ: OperatorCgammaDefinitionNeutral} and $p \geq 1$. Then for each $\phi(\cdot) \in \mathcal{E}_{p}(a,b;W^{1,p})$ we have $(\mathcal{I}_{C^{\gamma}}\phi)(\cdot) \in W^{1,p}(a,b;\mathbb{M}_{\gamma})$ and
	\begin{equation}
		\label{EQ: EmbracingSobolevDerFormula}
		\frac{d}{dt}(\mathcal{I}_{C^{\gamma}}\phi)(t) = (\mathcal{I}_{C^{\gamma}}\phi')(t) \text{ for almost all } t \in (a,b),
	\end{equation}
    where $\phi'$ as in \eqref{EQ: EmbracingSpaceSobolevNorm}. In particular, the linear operator
    \begin{equation}
    	\mathcal{I}_{C^{\gamma}} \colon \mathcal{E}_{p}(a,b;W^{1,p}) \to W^{1,p}(a,b;\mathbb{M}_{\gamma})
    \end{equation}
    is bounded and its norm does not exceed the total variation $\operatorname{Var}_{[-\tau,0]}(\gamma)$ of $\gamma$ on $[-\tau,0]$.
\end{theorem}
\begin{proof}
	Note that \eqref{EQ: EmbracingSobolevDerFormula} is satisfied by definition for $C^{\gamma} = \delta_{\theta}$ and any $\theta \in [-\tau,0]$. For general $C^{\gamma}$ we can use approximations by $C^{\gamma_{k}}$ as in \eqref{EQ: ConvergenceDeltaFunc} and the pointwise convergence argument. The proof is finished.
\end{proof}

Now we are going to explore some connections of $\mathcal{I}_{C^{\gamma}}$ with the Fourier transform. Below, $\mathbb{F}$ is a complex Hilbert space. We firstly illustrate the problem by means of the following remark.
\begin{remark}
	Let $\mathfrak{F}_{1}$ be the Fourier transform in $L_{2}(\mathbb{R};L_{2}(-\tau,0;\mathbb{F}))$ and let $\mathfrak{F}^{\gamma}_{2}$ be the Fourier transform in $L_{2}(\mathbb{R};\mathbb{M}_{\gamma})$. Suppose $C \colon L_{2}(-\tau,0;\mathbb{F}) \to \mathbb{M}_{\gamma}$ is a bounded linear operator. Then from the Parseval identity for any $\phi(\cdot) \in L_{2}(\mathbb{R};L_{2}(-\tau,0;\mathbb{F}))$ we have
	\begin{equation}
		(\mathfrak{F}^{\gamma}_{2}C\phi)(\omega) = C\mathfrak{F}_{1}(\phi)(\omega) \text{ for almost all } \omega \in \mathbb{R}.
	\end{equation}
	In this sense, the pointwise measurement operator $\mathcal{I}_{C}$ associated with $C$ commutes with the Fourier transform. In Theorem \ref{TH: EmbracingFourierCommutesIGamma} below, we state an analog of this property satisfied for more general class of operators $C=C^{\gamma}$ as in \eqref{EQ: OperatorCgammaDefinitionNeutral} if we are restricted to the embracing space $\mathcal{E}_{2}(\mathbb{R};L_{2}(-\tau,0;\mathbb{F}))$.
\end{remark}

\begin{theorem}
	\label{TH: EmbracingFourierCommutesIGamma}
	The Fourier transform $\mathfrak{F}_{1}$ in $L_{2}(\mathbb{R};L_{2}(-\tau,0;\mathbb{F}))$ provides an isometric automorphism of the embracing space $\mathcal{E}_{2}(\mathbb{R};L_{2}(-\tau,0;\mathbb{F}))$. Moreover, for any $\gamma$ and $C^{\gamma}$ as in \eqref{EQ: OperatorCgammaDefinitionNeutral} the following diagram
	\begin{equation}
		\label{EQ: EmbracingSpaceIGammaFourierIdentity}
		\begin{tikzcd}
			\mathcal{E}_{2}(\mathbb{R};L_{2})\ar[d, "\mathfrak{F}_{1}"] \ar[r,"\mathcal{I}_{C^{\gamma}}"] 
			& L_{2}(\mathbb{R};\mathbb{M}_{\gamma}) \arrow[d,"\mathfrak{F}^{\gamma}_{2}"] 
			\\
			\mathcal{E}_{2}(\mathbb{R};L_{2})  \arrow[r,"\mathcal{I}_{C^{\gamma}}"] 
			& L_{2}(\mathbb{R};\mathbb{M}_{\gamma})
		\end{tikzcd}
	\end{equation}
	is commutative. Here $\mathcal{I}_{C^{\gamma}}$ is given by Theorem \ref{TH: EmbracingSpaceFucntionalIGammaNeutral}, $\mathfrak{F}^{\gamma}_{2}$ is the Fourier transform in $L_{2}(\mathbb{R};\mathbb{M}_{\gamma})$ and $L_{2}$ in the range stands for $L_{2}(-\tau,0;\mathbb{F})$.
\end{theorem}
\begin{proof}
	Take any $\phi(\cdot) \in \mathcal{E}_{2}(\mathbb{R};L_{2})$. Firstly, we have to show that $(\mathfrak{F}_{1}\phi)(\cdot) \in \mathcal{E}_{2}(\mathbb{R};L_{2})$. By definition and since $L_{2}(\mathbb{R};W^{1,2}(-\tau,0;\mathbb{F}))$ is dense in $L_{2}(\mathbb{R};C([-\tau,0];\mathbb{F}))$, there exists a sequence $\phi_{k}(\cdot) \in L_{2}(\mathbb{R};W^{1,2}(-\tau,0;\mathbb{F}))$, where $k=1,2,\ldots$, approximating $\phi(\cdot)$ in $\mathcal{E}_{2}(\mathbb{R};L_{2})$, i.e. in terms of \eqref{EQ: EmbracingSpaceCorrespondenceContFunc} we have
	\begin{equation}
		\label{EQ: EmbracingFourierInvarianceApprox}
		\mathcal{R}_{\phi_{k}}(\cdot) \to \mathcal{R}_{\phi}(\cdot) \text{ in } C([-\tau,0];L_{2}(\mathbb{R};\mathbb{F})).
	\end{equation}
	Note that we have $(\mathfrak{F}_{1}\phi_{k})(\cdot) \in L_{2}(\mathbb{R};W^{1,2}(-\tau,0;\mathbb{F}))$. Let $\mathfrak{F}_{2}$ be the Fourier transform in $L_{2}(\mathbb{R};\mathbb{F})$. Then for each $\theta \in [-\tau,0]$ we have the identities in $L_{2}(\mathbb{R};\mathbb{F})$ w.r.t. $\omega \in \mathbb{R}$ as
	\begin{equation}
		(\mathfrak{F}_{2}\mathcal{R}_{\phi_{k}}(\theta))(\omega) = \lim_{T \to +\infty}\frac{1}{\sqrt{2\pi}}\int_{-T}^{T}e^{-\omega t}\delta_{\theta}\phi_{k}(t)dt = \lim_{T \to +\infty}\delta_{\theta}\frac{1}{\sqrt{2\pi}}\int_{-T}^{T}e^{-\omega t}\phi_{k}(t)dt = \delta_{\theta}(\mathfrak{F}_{1}\phi_{k})(\omega).
	\end{equation}
	From this and \eqref{EQ: EmbracingFourierInvarianceApprox} we have the limits in $L_{2}(\mathbb{R};\mathbb{F})$, which are uniform in $\theta \in [-\tau,0]$, as
	\begin{equation}
		\label{EQ: FourierEmbracingSpaceIdentityLast}
		\mathfrak{F}_{2}\mathcal{R}_{\phi}(\theta) = \lim_{k \to \infty} \mathfrak{F}_{2}\mathcal{R}_{\phi_{k}}(\theta) = \lim_{k \to \infty} \mathcal{I}_{\delta_{\theta}}(\mathfrak{F}_{1}\phi_{k}).
	\end{equation}
	In other words, $\mathfrak{F}_{1}\phi_{k}(\cdot)$ is a convergent sequence in $\mathcal{E}_{2}(\mathbb{R};L_{2})$. Since, by Lemma \ref{LEM: EmbracingSpaceDescriptionCLp}, the embracing space is naturally continuously embedded into $L_{2}(\mathbb{R};L_{2}(-\tau,0;\mathbb{F}))$ and $\mathfrak{F}_{1}\phi_{k}$ converges to  $\mathfrak{F}_{1}\phi$ in the latter space, we must have $(\mathfrak{F}_{1}\phi)(\cdot) \in \mathcal{E}_{2}(\mathbb{R};L_{2})$. 
	
	From \eqref{EQ: FourierEmbracingSpaceIdentityLast} we get the identity
	\begin{equation}
		\label{EQ: EmbracingFourierIdentityDelta}
		\mathcal{R}_{\mathfrak{F}_{1}\phi}(\theta) = \mathcal{I}_{\delta_{\theta}}\mathfrak{F}_{1}\phi = \mathfrak{F}_{2} \mathcal{I}_{\delta_{\theta}} \phi \text{ for any } \phi(\cdot) \in \mathcal{E}_{2}(\mathbb{R};L_{2}).
	\end{equation} 
    From this and since $\mathfrak{F}_{1}$ takes the dense subspace $L_{2}(\mathbb{R};W^{1,2}(-\tau,0;\mathbb{F}))$ into itself, we get that $\mathfrak{F}_{1}$ provides an isometric automorphism of $\mathcal{E}_{2}(\mathbb{R};L_{2})$. 
	
	Moreover, \eqref{EQ: EmbracingFourierIdentityDelta} is equivalent to the commutativity of the diagram from \eqref{EQ: EmbracingSpaceIGammaFourierIdentity} for $C^{\gamma} = \delta_{\theta}$ and any $\theta \in [-\tau,0]$. Clearly, this is sufficient to show the commutativity for general $C^{\gamma}$ by using its approximations by $C^{\gamma_{k}}$ as in \eqref{EQ: ConvergenceDeltaFunc} since for any $\phi(\cdot) \in \mathcal{E}_{2}(\mathbb{R};L_{2})$ we have the pointwise convergence of $\mathcal{I}_{C^{\gamma_{k}}}\phi$ to $\mathcal{I}_{C^{\gamma}}\phi$ as $k \to \infty$. The proof is finished.
\end{proof}

\subsection{Spaces of adorned functions}
\label{SUBSEC: SpacedAdornedFunctions}
Now we are going to explore some particular classes of functions that belong to the embracing space $\mathcal{E}_{p}(0,T;L_{p})$ for $T>0$ or $T=\infty$. We will sometimes refer to the interval $[0,T]$ with $T=\infty$ meaning $[0,\infty)$.

Let $T>0$ and $x \colon [-\tau,T] \to \mathbb{F}$ be given. Then for any $t \in [0,T]$ we use the notation $x_{t}(\theta) := x(t+\theta)$, where $\theta \in [-\tau,0]$, to denote the $\tau$-history segment of $x(\cdot)$ at $t$.

Let us fix a nonnegative continuous function $\rho \colon [0,+\infty) \to \mathbb{R}_{\geq 0}$. We assume that for some constant $\rho_{0}=\rho_{0}(\rho,\tau)$ we have 
\begin{equation}
	\label{EQ: WeightFunctionProperty}
	\rho(t+s) \leq \rho_{0} \cdot \rho(t) \text{ for all } t \geq 0 \text{ and } s \in [0,\tau].
\end{equation}
Under such conditions, $\rho(\cdot)$ is called a \textit{weight function}. If $\rho(t) > 0$ for all $t \geq 0$, we say that $\rho(\cdot)$ is \textit{positive}. In further sections we will take $\rho(t) = e^{\nu t}$ for some fixed $\nu \in \mathbb{R}$ and $t \geq 0$.

For a given $T>0$ and $x \colon [-\tau, T] \to \mathbb{F}$ we define a function $\phi_{x,\rho}(\cdot)$ on $[0,T]$ as
\begin{equation}
	\label{EQ: WeightedWindowDefinition}
	\phi_{x,\rho}(t) := \rho(t)x_{t} \text{ for all } t \in [0,T].
\end{equation}
We are interested in $\phi_{x,\rho}(\cdot)$ corresponding to $x(\cdot)$ from some particular spaces. Namely, it is not hard to see the following implications (here $p \geq 1$)
\begin{equation}
	\label{EQ: WindowFunctionsDifferentSpaces}
	\begin{split}
		&x(\cdot) \in C([-\tau,T];\mathbb{F}) \Rightarrow \phi_{x,\rho}(\cdot) \in C([0,T];C([-\tau,0];\mathbb{F})),\\
		&x(\cdot) \in L_{p}(-\tau,T;\mathbb{F}) \Rightarrow \phi_{x,\rho}(\cdot) \in C([0,T];L_{p}(-\tau,0;\mathbb{F})),\\
		&x(\cdot) \in W^{1,p}(-\tau,T;\mathbb{F}) \Rightarrow \phi_{x,\rho}(\cdot) \in  C([0,T];W^{1,p}(-\tau,0;\mathbb{F}))
	\end{split}
\end{equation}
are satisfied since the weight function $\rho(\cdot)$ is assumed to be continuous and the group of translations in $L_{p}(\mathbb{R}; \mathbb{F})$ is strongly continuous.

We use $|\cdot|_{\mathbb{F}}$ and $|\cdot|_{\mathbb{M}_{\gamma}}$ to denote the norms in $\mathbb{F}$ and $\mathbb{M}_{\gamma}$ respectively. Let us start with the following lemma.
\begin{lemma}
	\label{LEM: DelayMeasuEstimate}
	For all $T>0$, $p \geq 1$, any $x(\cdot) \in C([-\tau,T];\mathbb{F})$ and $\phi_{x,\rho}(\cdot)$ associated with $x(\cdot)$ via \eqref{EQ: WeightedWindowDefinition} we have
	\begin{equation}
		\label{EQ: FuncLemmaLPIneq}
		\left(\int_{0}^{T} |\delta_{\theta}\phi_{x,\rho}(t)|^{p}_{\mathbb{M}_{\gamma}} dt\right)^{1/p} \leq \kappa(\rho,\tau) \cdot \left( \int_{-\tau}^{0}|x(t)|^{p}_{\mathbb{F}}dt + \int_{0}^{T} |\rho(t) x(t)|^{p}_{\mathbb{F}} \right)^{1/p} \text{ for all } \theta \in [-\tau,0],
	\end{equation}
    where $\kappa(\rho,\tau)$ is given by \eqref{EQ: WeightKappaDefinition}.
\end{lemma}
\begin{proof}
	Indeed, for $x(\cdot)$ and $\phi_{x,\rho}(\cdot)$ as in the statement and any $\theta \in [-\tau,0]$ we have
	\begin{equation}
		\label{EQ: LemmaAdornedEstimateDeltaFunc}
		\begin{split}
			\int_{0}^{T}|\delta_{\theta}\phi_{x,\rho}(t)|^{p}_{\mathbb{M}_{\gamma}} dt &= \int_{0}^{T}|\rho(t) x(t+\theta)|^{p}_{\mathbb{F}}dt = \int_{-\tau_{0}}^{T-\tau_{0}}|\rho(t-\theta)x(t)|^{p}_{\mathbb{F}}dt \leq \\ &\leq  \kappa(\rho,\tau)^{p} \cdot \left( \int_{-\tau}^{0}|x(t)|^{p}_{\mathbb{F}}dt + \int_{0}^{T} |\rho(t) x(t)|^{p}_{\mathbb{F}} \right)dt,
		\end{split}
	\end{equation}
	where 
	\begin{equation}
		\label{EQ: WeightKappaDefinition}
		\kappa(\rho,\tau) = \max\left\{ \rho_{0}(\rho,\tau), \max\limits_{t \in [0,\tau]} \rho(t) \right\}.
	\end{equation}
	The proof is finished.
\end{proof}

For $p \geq 1$ and $0 < T < \infty$ we define the space $\mathcal{Y}^{p}_{\rho}(0,T;L_{p}(-\tau,0;\mathbb{F}))$ consisting of all functions $\phi_{x,\rho}(\cdot)$ on $[0,T]$ given by \eqref{EQ: WeightedWindowDefinition} with some $x(\cdot) \in L_{p}(-\tau,T;\mathbb{F})$. We define $\mathcal{Y}^{p}_{\rho}(0,\infty;L_{p}(-\tau,0;\mathbb{F}))$ as the space of all $\phi_{x,\rho}(\cdot)$ as in \eqref{EQ: WeightedWindowDefinition} with $T=\infty$, $x(\cdot) \in L_{p,loc}([-\tau,\infty);\mathbb{F})$ and $\rho(\cdot) x(\cdot) \in L_{p}(0,\infty;\mathbb{F})$.

Recall that $\phi(t)$ must depend continuously on $t \in [0,T]$ in $L_{p}(-\tau,0;\mathbb{F})$ (see \eqref{EQ: WindowFunctionsDifferentSpaces}). We call any such $\phi_{x,\rho}(\cdot)$ a $\rho$-\textit{adorned} $L_{p}(-\tau,0;\mathbb{F})$-valued function on $[0,T]$ or simply that $\phi_{x,\rho}(\cdot)$ is $\rho$-adorned. Sometimes we will say that $\phi_{x,\rho}(\cdot)$ is a $\rho$-\textit{adornment} of $x(\cdot)$.

We suppose that $\rho(\cdot)$ is positive and endow the space $\mathcal{Y}^{p}_{\rho}(0,T;L_{p}(-\tau,0;\mathbb{F}))$ of $\rho$-adorned functions with the weighted $L_{p}$-norm of $x(\cdot)$ motivated by \eqref{EQ: FuncLemmaLPIneq} as
\begin{equation}
	\label{EQ: NorminHTP}
	\| \phi_{x,\rho}(\cdot) \|_{\mathcal{Y}^{p}_{\rho}(0,T;L_{p}(-\tau,0;\mathbb{F}))} := \left( \int_{-\tau}^{0}|x(\theta)|^{p}_{\mathbb{F}}d\theta +  \int_{0}^{T}|\rho(t)x(t)|^{p}_{\mathbb{F}}dt\right)^{1/p}.
\end{equation}
Since we assumed that $\rho(\cdot)$ is positive, we have the uniqueness of $x(\cdot)$ representing $\phi_{x,\rho}(\cdot)$ via \eqref{EQ: WeightedWindowDefinition} and, consequently, \eqref{EQ: NorminHTP} is well-defined. This makes $\mathcal{Y}^{p}_{\rho}(0,T;L_{p}(-\tau,0;\mathbb{F}))$ a Banach space for $p \geq 1$ and a Hilbert space for $p=2$.

For brevity, we will often write, $\mathcal{Y}^{p}_{\rho}(0,T;L_{p})$ instead of $\mathcal{Y}^{p}_{\rho}(0,T;L_{p}(-\tau,0;\mathbb{F}))$.

Clearly, the subspace of $\phi_{x,\rho}(\cdot)$ with $x(\cdot) \in C([-\tau,T];\mathbb{F})$ is dense in $\mathcal{Y}^{p}_{\rho}(0,T;L_{p})$. Thus, Lemma \ref{LEM: DelayMeasuEstimate} immediately implies the following.
\begin{lemma}
	\label{LEM: EmbeddingAdornedIntoEmbracingNeutral}
	Suppose $\rho(\cdot)$ is positive. Then for $T>0$ or $T=\infty$ and $p \geq 1$ there is a natural embedding of the space $\mathcal{Y}^{p}_{\rho}(0,T;L_{p}(-\tau,0;\mathbb{F}))$ into $\mathcal{E}_{p}(0,T;L_{p}(-\tau,0;\mathbb{F}))$ such that
	\begin{equation}
		\| \phi_{x,\rho}(\cdot) \|_{\mathcal{E}_{p}(0,T;L_{p}(-\tau,0;\mathbb{F}))} \leq \kappa(\rho,\tau) \cdot \| \phi_{x,\rho}(\cdot) \|_{\mathcal{Y}^{p}_{\rho}(0,T;L_{p}(-\tau,0;\mathbb{F}))},
	\end{equation}
     where $\kappa(\rho,\tau)$ is given by \eqref{EQ: WeightKappaDefinition}.
\end{lemma}

From Lemma \ref{LEM: EmbeddingAdornedIntoEmbracingNeutral} and Theorem \ref{TH: EmbracingSpaceFucntionalIGammaNeutral} we obtain the following.
\begin{theorem}
	\label{TH: OperatorIcAdornedLp}
	Suppose $\rho(\cdot)$ is positive and let $\gamma$ and $C^{\gamma}$ be as in \eqref{EQ: OperatorCgammaDefinitionNeutral}. Then for any $T>0$ or $T=\infty$ and $p \geq 1$ there exists a bounded linear operator
	\begin{equation}
		\mathcal{I}_{C^{\gamma}} \colon \mathcal{Y}^{p}_{\rho}(0,T;L_{p}(-\tau,0;\mathbb{F})) \to L_{p}(0,T;\mathbb{M}_{\gamma})
	\end{equation}
	with the norm not exceeding $\kappa(\rho,\tau) \cdot \operatorname{Var}_{[-\tau,0]}(\gamma)$ and such that for any $\phi_{x,\rho}(\cdot) \in \mathcal{Y}^{p}_{\rho}(0,T;L_{p}(-\tau,0;\mathbb{F}))$ with $x(\cdot) \in C([-\tau,T];\mathbb{F})$ we have
	\begin{equation}
		\label{EQ: OperatorICgammaAdornedContinuousIdentity}
		(\mathcal{I}_{C^{\gamma}}\phi_{x,\rho})(t) = C^{\gamma}\phi_{x,\rho}(t) \text{ for all } t \in [0,T].
	\end{equation}
\end{theorem}

Now let us for $T>0$ consider the space $\mathcal{Y}^{p}_{\rho}(0,T;W^{1,p}(-\tau,0;\mathbb{F}))$ consisting of all $\phi_{x,\rho}$ as in \eqref{EQ: WeightedWindowDefinition} with $x(\cdot) \in W^{1,p}(-\tau,T;\mathbb{F})$. For brevity, we will often write $\mathcal{Y}^{p}_{\rho}(0,T;W^{1,p})$ to denote that space.

Let us assume that $\rho(\cdot)$ is positive, $\rho(\cdot) \in C^{1}([0,\infty); \mathbb{R})$ and the derivative $\dot{\rho}(\cdot)$ of $\rho(\cdot)$ either is also a positive weight function or identically zero. In this case, for brevity, we say that $\rho$ is a \textit{proper $C^{1}$-weight}. We endow the space $\mathcal{Y}^{p}_{\rho}(0,T;W^{1,p})$ with the norm
\begin{equation}
	\label{EQ: WindowSobolevSpaceNorm}
	\| \phi_{x,\rho} (\cdot) \|^{p}_{\mathcal{Y}^{p}_{\rho}(0,T;W^{1,p}(-\tau,0;\mathbb{F}))} := \| \phi_{x,\rho}(\cdot) \|^{p}_{\mathcal{Y}^{p}_{\rho}(0,T;L_{p})} + \| \phi_{\dot{x},\rho}(\cdot) \|^{p}_{\mathcal{Y}^{p}_{\rho}(0,T;L_{p})} + \| \phi_{x,\dot{\rho}}(\cdot) \|^{p}_{\mathcal{Y}^{p}_{\rho}(0,T;L_{p})}.
\end{equation}
Note that for $\dot{\rho} \equiv 0$ the last term in \eqref{EQ: WindowSobolevSpaceNorm} vanishes. For $T=\infty$ we define $\mathcal{Y}^{p}_{\rho}(0,\infty;W^{1,p})$ by naturally requiring the norm in \eqref{EQ: WindowSobolevSpaceNorm} to be finite.  Clearly, $\mathcal{Y}^{p}_{\rho}(0,T;W^{1,p})$ endowed with the norm from \eqref{EQ: WindowSobolevSpaceNorm} becomes a Banach space for $p \geq 1$ and a Hilbert space for $p=2$ by the same reasoning as for \eqref{EQ: NorminHTP}.

From the following theorem, we have, in particular, that $\mathcal{Y}^{p}_{\rho}(0,T;W^{1,p})$ is naturally continuously embedded into $\mathcal{E}_{p}(0,T;W^{1,p})$. This puts the theorem into the context of Theorem \ref{TH: EmbracingSobolevDifferentiabilityTheorem}.
\begin{theorem}
	\label{TH: ICgammaSobolevAdorned}
	Suppose $\rho(\cdot)$ is a proper $C^{1}$-weight and let $\gamma$ and $C^{\gamma}$ be as in \eqref{EQ: OperatorCgammaDefinitionNeutral}. Then for any $T>0$ or $T=\infty$ and $p \geq 1$ for each $x \in W^{1,p}(-\tau,T;\mathbb{F})$ we have $(\mathcal{I}_{C^{\gamma}} \phi_{x,\rho})(\cdot) \in W^{1,p}(0,T;\mathbb{M}_{\gamma})$ and
	\begin{equation}
		\label{EQ: ICgammaSobolevAdornedDerivativeIdentity}
		\frac{d}{dt}(\mathcal{I}_{C^{\gamma}} \phi_{x,\rho})(t) = (\mathcal{I}_{C^{\gamma}} \phi_{\dot{x},\rho})(t) + (\mathcal{I}_{C^{\gamma}}\phi_{x,\dot{\rho}})(t) \text{ for almost all } t \in (0,T).
	\end{equation} 
    In particular, the operator
	\begin{equation}
		\label{EQ: ICgammaSobolevAdorned}
		\mathcal{I}_{C^{\gamma}} \colon \mathcal{Y}^{p}_{\rho}(0,T;W^{1,p}(-\tau,0;\mathbb{F})) \to W^{1,p}(0,T;\mathbb{M}_{\gamma})
	\end{equation}
	is well-defined and bounded with the norm not exceeding $\operatorname{Var}_{[-\tau,0]}(\gamma)$ times a constant that depends only on $\rho$, $\dot{\rho}$ and $\tau$ (see Theorem \ref{TH: OperatorIcAdornedLp}).
\end{theorem}
\begin{proof}
	We show the statement by considering $C^{\gamma} = \delta_{-\tau_{0}}$ for some $-\tau_{0} \in [-\tau,0]$ first. Since $W^{1,p}(-\tau,T;\mathbb{F})$ is naturally continuously embedded into $C([-\tau,T];\mathbb{F})$, from \eqref{EQ: OperatorICgammaAdornedContinuousIdentity} and the Leibniz rule we have
	\begin{equation}
		\frac{d}{dt}(C^{\gamma} \phi_{x,\rho}(t)) = \frac{d}{dt}\left( \rho(t) x(t-\tau_{0}) \right) = (\mathcal{I}_{C^{\gamma}}\phi_{x,\dot{\rho}})(t) + (\mathcal{I}_{C^{\gamma}}\phi_{\dot{x},\rho})(t) \text{ for almost all } t \in (0,T)
	\end{equation}
    and, in particular, $(\mathcal{I}_{C^{\gamma}} \phi_{x,\rho})(\cdot) \in W^{1,p}(0,T;\mathbb{F})$ due to Theorem \ref{TH: OperatorIcAdornedLp}. This shows \eqref{EQ: ICgammaSobolevAdornedDerivativeIdentity}, gives the norm estimate and the embedding into $\mathcal{E}_{p}(0,T;W^{1,p})$.
    
    Now the proof can be finished by considering approximations of general $C^{\gamma}$ by operator-linear combinations $C^{\gamma_{k}}$ of $\delta$-functionals as in \eqref{EQ: ConvergenceDeltaFunc}. Namely, it is sufficient to note that, by the uniform boundedness and convergence on a dense subspace, we have for each $y \in L_{p}(-\tau,T;\mathbb{F})$ the pointwise convergence in $L_{p}(0,T;\mathbb{M}_{\gamma})$ as
    \begin{equation}
    	\mathcal{I}_{C^{\gamma_{k}}}\phi_{y,\rho} \to \mathcal{I}_{C^{\gamma}}\phi_{y,\rho} \text{ as } k \to \infty.
    \end{equation}
    Thus, \eqref{EQ: ICgammaSobolevAdornedDerivativeIdentity} is preserved for general $\gamma$ and the norm of $\mathcal{I}_{C^{\gamma}}$ from \eqref{EQ: ICgammaSobolevAdorned} can be estimated from that identity via Theorem \ref{TH: OperatorIcAdornedLp} and the Minkowski and H\"{o}lder inequalities. The proof is finished.
\end{proof}

\begin{remark}
	For $\rho \equiv 1$, Theorems \ref{TH: OperatorIcAdornedLp} and \ref{TH: ICgammaSobolevAdorned} are similar to the convolution theorems used in \cite{BurnsHerdmanStech1983} (see Lemma 2.4 therein). For general $\rho$, the results can be considered as theorems on weighted convolutions. Here we presented a simplified proof which is independent from abstract harmonic analysis (referred to in \cite{BurnsHerdmanStech1983}) and it embeds the result into a wider context which is more appropriate for further studying and generalization (see \cite{Anikushin2023Comp}).
\end{remark}

\subsection{Spaces of twisted functions}
\label{SUBSEC: SpacesTwistedFunctions}
Now we are going to introduce another class of functions with similar properties. For this, let $T_{1}(t)$, where $t \geq 0$, be the semigroup of left translates in $L_{p}(-\tau,0;\mathbb{F})$, i.e. for any $\psi \in L_{p}(-\tau,0;\mathbb{F})$ we have
\begin{equation}
	\label{EQ: DefinitionLeftTranslates}
	(T_{1}(t)\psi)(\theta) := \begin{cases}
		\psi(\theta+t), \text{ if } \theta+t \in [-\tau,0],\\
		0, \text{ otherwise}.
	\end{cases}\text{ for } \theta \in [-\tau,0] \text{ and } t \geq 0.
\end{equation}
It is well-known that $T_{1}(t)$, $t \geq 0$, is a $C_{0}$-semigroup in $L_{p}(-\tau,0;\mathbb{F})$ and its generator $A_{T_{1}}$ is the operator $\frac{d}{d\theta}$ with the domain 
\begin{equation}
	\label{EQ: OperatorAT1Domain}
	\mathcal{D}(A_{T_{1}}) = W^{1,2}_{0+}(-\tau,0;\mathbb{F}) := \{ \psi \in W^{1,p}(-\tau,0;\mathbb{F}) \ | \ \psi(0) = 0 \}.
\end{equation}

For $T>0$ and $y \in L_{p}(0,T;L_{p}(-\tau,0;\mathbb{F}))$ we consider the $L_{p}(-\tau,0;\mathbb{F})$-valued function $\psi_{y,\rho}(\cdot)$ on $[0,T]$ given by
\begin{equation}
	\label{EQ: TwistedFunctionNeutralDefinition}
	\psi_{y,\rho}(t) = \rho(t) \int_{0}^{t}T_{1}(t-s) y(s) ds \text{ for all } t \in [0,T].
\end{equation}
We call such $\psi_{y,\rho}$ a $\rho$-\textit{twisted} $L_{p}(-\tau,0;\mathbb{F})$-valued function on $[0,T]$ or simply say that $\psi_{y,\rho}$ is $\rho$-twisted. Sometimes we will say that $\psi_{y,\rho}$ is a $\rho$-\textit{twisting} of $y$. It is not hard to see that we have $\psi_{y,\rho}(\cdot) \in C([0,T];L_{p}(-\tau,0;\mathbb{F}))$.

It turns out that $\psi_{y,\rho}(\cdot)$ determines $y(\cdot)$ via \eqref{EQ: TwistedFunctionNeutralDefinition} uniquely for positive $\rho(\cdot)$. This is contained in the following lemma.
\begin{lemma}
	\label{LEM: TwistedFuncUniquenessY}
	Suppose for $T>0$, $p \geq 1$ and some $y(\cdot) \in L_{p}(0,T;L_{p}(-\tau,0;\mathbb{F}))$ we have
	\begin{equation}
		\label{EQ: TwistedFunctionUniquenessZeroIndentity}
		\int_{0}^{t}T_{1}(t-s)y(s)ds = 0 \text{ for all } t \in [0,T].
	\end{equation}
    Then $y(t) = 0$ for all $t \in [0,T]$.
\end{lemma}
\begin{proof}
	Indeed, let us define $\widetilde{y}(s,\theta) := y(s)(\theta)$ for $s \in [0,T]$ and almost all $\theta \in [-\tau,0]$. Then, according to \eqref{EQ: DefinitionLeftTranslates}, we have that \eqref{EQ: TwistedFunctionUniquenessZeroIndentity} is equivalent to
	\begin{equation}
		\label{EQ: TwistedFunctionUniquenessIntegralIdentity}
		\int\limits_{\max\{ 0, t + \theta \}}^{t} \widetilde{y}(s,\theta+t-s)ds = 0 \text{ for all } t \in [0,T] \text{ and almost all } \theta \in [-\tau,0].
	\end{equation}
    We have to show that $\widetilde{y}(\cdot,\cdot) = 0$ almost everywhere on $[0,T] \times [-\tau,0]$. Let $l(t,\theta)$ be the line segment in $[0,T] \times [-\tau,0]$ over which we integrate in \eqref{EQ: TwistedFunctionUniquenessIntegralIdentity}. Then we have
    \begin{equation}
    	\label{EQ: TwistedFunctionUniquenessUnionParallelLines}
    	[0,T] \times [-\tau,0] = \bigcup_{t \in [0,T-\tau] } l(t, -\tau) \cup \bigcup_{\theta \in [-\tau,0]} l(T,\theta),
    \end{equation}
    where the union is taken over distinct parallel segments. Moreover, since $l(t-h,\theta+h) \subset l(t,\theta)$ for all $ 0 \leq h \leq \min\{ -\theta, t \}$, from \eqref{EQ: TwistedFunctionUniquenessIntegralIdentity} we have that the integral of $\widetilde{y}$ over each segment $l \subset l(t,\theta)$ vanishes for all $t \in [0,T]$ and almost all $\theta \in [-\tau,0]$. Now it is a standard exercise\footnote{Here we mean the following. Suppose for some $a < b$ and a Lebesgue integrable real-valued function $f(\cdot)$ on $[a,b]$ we have $\int_{a}^{c}f(x)dx = 0$ for any $c \in [a,b]$. Then $f(x) = 0$ for almost all $x \in [a,b]$. This generalizes to functions taking values in a separable Hilbert space $\mathbb{F}$ by applying the scalar result to the Fourier coefficients over some orthonormal basis in $\mathbb{F}$.} from the measure theory to show that $\widetilde{y}$ must vanish on such $l(t,\theta)$. From \eqref{EQ: TwistedFunctionUniquenessUnionParallelLines} we have that $\widetilde{y}$ vanishes almost everywhere in $[0,T] \times [-\tau,0]$ and, consequently, $y(t) = 0$ for all $t \in [0,T]$. The proof is finished.
\end{proof}

Assuming that the weight function $\rho(\cdot)$ is positive, for $T>0$ we define the space $\mathcal{T}^{p}_{\rho}(0,T;L_{p}(-\tau,0;\mathbb{F}))$ of all $\rho$-twisted functions $\psi_{y,\rho}(\cdot)$ and endow it with the norm 
\begin{equation}
	\label{EQ: TwistedLpNorm}	
	\| \psi_{y,\rho}(\cdot) \|_{ \mathcal{T}^{p}_{\rho}(0,T;L_{p}(-\tau,0;\mathbb{F})) } := \left( \int_{0}^{T}\| \rho(t) y(t) \|^{p}_{L_{p}(-\tau,0;\mathbb{F})}dt\right)^{1/p},
\end{equation}
which is well-defined due to Lemma \ref{LEM: TwistedFuncUniquenessY}. Moreover, for $T=\infty$ we define the space $\mathcal{T}^{p}_{\rho}(0,\infty;L_{p}(-\tau,0;\mathbb{F}))$ as the space of all function $\psi_{y,\rho}(\cdot)$ on $[0,\infty)$ as in \eqref{EQ: TwistedFunctionNeutralDefinition} such that $\rho(\cdot)y(\cdot) \in L_{2}(0,\infty;L_{p}(-\tau,0;\mathbb{F}))$. It is clear that the space  $\mathcal{T}^{p}_{\rho}(0,T;L_{p}(-\tau,0;\mathbb{F}))$ endowed with the norm from \eqref{EQ: TwistedLpNorm} becomes a Banach space for $p \geq 1$ and a Hilbert space for $p=2$.

For brevity, we will sometimes use $\mathcal{T}^{p}_{\rho}(0,T;L_{p})$ to denote the space just introduced.

It turns out that the spaces $\mathcal{Y}^{p}_{\rho}(0,T;L_{p})$ and $\mathcal{T}^{p}_{\rho}(0,T;L_{p})$ for $p>1$ are linearly independent. This is caused by the circumstance that each $\psi_{y,\rho}(t)$ according to \eqref{EQ: TwistedFunctionNeutralDefinition} must have small $L_{p}$-norm near $0$ and the smallness is uniform in $t$. A proper development of the argument gives the following lemma.
\begin{lemma}
	\label{LEM: SumAdornedTwistedUniq}
	Suppose for $T>0$, $p>1$ and some $\phi_{x,\rho}(\cdot) \in \mathcal{Y}^{p}_{\rho}(0,T;L_{p})$ and $\psi_{y,\rho}(\cdot) \in \mathcal{T}^{p}_{\rho}(0,T;L_{p})$ we have
	\begin{equation}
		\label{EQ: AdornedTwistedNeutralLinearInd}
		\phi_{x,\rho}(t) + \psi_{y,\rho}(t) = 0 \text{ for all } t \in [0,T].
	\end{equation}
    Then $\phi_{x,\rho}(t) = \psi_{y,\rho}(t) = 0$ for all $t \in [0,T]$. 
\end{lemma}
\begin{proof}
	Since $\rho(\cdot)$ is positive, it is sufficient to consider the case $\rho(\cdot) \equiv 1$. Note that we always have $\psi_{y,\rho}(0) = 0$ and, consequently, from \eqref{EQ: AdornedTwistedNeutralLinearInd} we get that $x(s) = 0$ for almost all $s \in (-\tau,0)$. For $h \in (0,\tau)$ let us consider the set $\mathcal{D}^{T}_{h} = \bigcup_{k = 1}^{\lfloor \frac{T}{h} \rfloor} (k-1,k)h$. Clearly, we have
    \begin{equation}
    	\label{EQ: AdornedTwistedIndependenceXIntApp}
    	\int_{-\tau}^{T}|x(t)|^{p}_{\mathbb{F}}dt = \lim\limits_{h \to 0+} \int_{\mathcal{D}^{T}_{h}}|x(t)|^{p}_{\mathbb{F}}dt
    \end{equation}
    satisfied since $(0,T) \setminus \mathcal{D}^{T}_{h}$ has Lebesgue measure $<h$. 
    
    From \eqref{EQ: DefinitionLeftTranslates} and the decomposition for $t \geq h$
    \begin{equation}
    	\label{EQ: AdornedTwistedIndependenceTwistedDecomp}
    	\psi_{y,\rho}(t) = \int_{0}^{t-h} T_{1}(t-s)y(s)ds + \int_{t-h}^{t}T_{1}(t-s)y(s)ds
    \end{equation}
    it is clear that the values of $\psi_{y,\rho}(t)$ on $(-h,0)$ are determined by the second summand in \eqref{EQ: AdornedTwistedIndependenceTwistedDecomp}. Using this, \eqref{EQ: AdornedTwistedNeutralLinearInd} and the H\"{o}lder inequality, we obtain
    \begin{equation}
    	\label{EQ: AdornedTwistedIndependenceLastEstimate}
    	\int_{\mathcal{D}^{T}_{h}}|x(t)|^{p}_{\mathbb{F}}dt \leq \sum_{k=1}^{\lfloor \frac{T}{h} \rfloor} \left\|\int_{(k-1)h}^{kh}T_{1}(t-s)y(s)ds\right\|^{p}_{L_{p}} \leq h^{p-1} \sum_{k=1}^{\lfloor \frac{T}{h} \rfloor} \int_{(k-1)h}^{kh} \|y(s)\|^{p}_{L_{p}}ds \leq h^{p-1} \int_{0}^{T} \|y(t)\|^{p}_{L_{p}}dt,
    \end{equation}
    where $L_{p}$ stands for $L_{p}(-\tau,0;\mathbb{F})$. Note that we assumed $p>1$.
    
    Combining \eqref{EQ: AdornedTwistedIndependenceLastEstimate} with \eqref{EQ: AdornedTwistedIndependenceXIntApp}, we get that $x(\cdot)$ vanishes almost everywhere in $(-\tau,T)$ and, consequently, $\psi_{y,\rho}(t) = \phi_{x,\rho}(t) = 0$ for all $t \in [0,T]$. The proof is finished.
\end{proof}

Now consider the space $C_{0+}([-\tau,0];\mathbb{F})$ of continuous $\mathbb{F}$-valued functions $\psi$ on $[-\tau,0]$ such that $\psi(0) = 0$. Note that for $y(\cdot)$ taken from $C([0,T];C_{0+}([-\tau,0];\mathbb{F}))$ we have that $\psi_{y,\rho}(\cdot)$ also belongs to the same space since the integral \eqref{EQ: TwistedFunctionNeutralDefinition} can be considered as the integral of a $C_{0+}([-\tau,0];\mathbb{F})$-valued continuous function. We have the following lemma.
\begin{lemma}
	\label{LEM: TwistedICgammaEstimateOnCont}
	For all $T>0$, $p \geq 1$, any $y(\cdot) \in C([0,T];C_{0+}([-\tau,0];\mathbb{F}))$ and $\psi_{y,\rho}(\cdot)$ associated with $y(\cdot)$ via \eqref{EQ: TwistedFunctionNeutralDefinition} we have
	\begin{equation}
		\left(\int_{0}^{T}|\delta_{\theta}\psi_{y,\rho}(t)|^{p}_{\mathbb{M}_{\gamma}}dt\right)^{1/p} \leq \rho_{0} \tau^{1-1/p} \cdot \left(\int_{0}^{T}\| \rho(t) y(t) \|^{p}_{L_{p}(-\tau,0;\mathbb{F})}dt\right)^{1/p} \text{ for all } \theta \in [-\tau,0],
	\end{equation}
    where $\rho_{0}$ is given by \eqref{EQ: WeightFunctionProperty}.
\end{lemma}
\begin{proof}
	Let us take $-\tau_{0} \in [-\tau,0]$ and put $\widetilde{y}(s,\theta) := y(s)(\theta)$ for all $s \in [0,T]$ and $\theta \in [-\tau,0]$. Then from \eqref{EQ: DefinitionLeftTranslates} and \eqref{EQ: TwistedFunctionNeutralDefinition} we get
	\begin{equation}
		\label{EQ: CdeltaPsiYIntFormula}
		\delta_{-\tau_{0}}\psi_{y,\rho}(t) = \psi_{y,\rho}(t)(-\tau_{0}) = \rho(t) \int_{\max\{0,t-\tau_{0}\}}^{t}\widetilde{y}(s,-\tau_{0}+t-s)ds \text{ for all } t \in [0,T].
	\end{equation}
    From \eqref{EQ: WeightFunctionProperty} we get that $\rho(t) \leq \rho_{0} \cdot \rho(s)$ for all $t \in [0,T]$ and $s \in [ \max\{0,t-\tau_{0}\}, t]$. Using this, the H\"{o}lder inequality and monotonicity of the integral, we obtain
    \begin{equation}
    	\begin{split}
    		 \int_{0}^{T}|\delta_{-\tau_{0}}\psi_{y,\rho}(t)|^{p}_{\mathbb{M}_{\gamma}}dt = \int_{0}^{T}|\rho(t)|^{p} \left|\int_{\max\{0,t-\tau_{0}\}}^{t}\widetilde{y}(s,-\tau_{0}+t-s)ds\right|^{p}_{\mathbb{F}} dt \leq\\ \leq \rho^{p}_{0} \tau^{p-1} \int_{0}^{T}\int_{\max\{0,t-\tau_{0}\}}^{t}|\rho(s)\widetilde{y}(s,-\tau_{0}+t-s)|^{p}_{\mathbb{F}}ds dt \leq \rho^{p}_{0} \tau^{p-1} \int_{[0,T] \times [-\tau,0]} |\rho(s) \widetilde{y}(s,\theta)|^{p}_{\mathbb{F}} d\theta dt = \\ = \rho^{p}_{0} \tau^{p-1} \int_{0}^{T} \| \rho(t) y(t) \|^{p}_{L_{p}(-\tau,0;\mathbb{F})}dt,
    	\end{split}
    \end{equation}
    where in the last inequality we used the linear change of variables $(t,s) \mapsto (t,-\tau_{0}+t-s)$ which has the determinant equal to $-1$ and then we applied the monotonicity. The proof is finished.
\end{proof}

Clearly, the subspace of $\psi_{y,\rho}$ with $y(\cdot) \in C([0,T];C_{0+}([-\tau,0];\mathbb{F}))$ is dense in $\mathcal{T}^{p}_{\rho}(0,T;L_{p})$. Thus, from Lemma \ref{LEM: TwistedICgammaEstimateOnCont} we immediately obtain the following.
\begin{lemma}
	\label{LEM: EmbeddingTwistedIntoEmbracingNeutral}
	Suppose $\rho(\cdot)$ is positive. Then for $T>0$ or $T=\infty$ and $p \geq 1$ there is a natural embedding of the space $\mathcal{T}^{p}_{\rho}(0,T;L_{p}(-\tau,0;\mathbb{F}))$ into $\mathcal{E}_{p}(0,T;L_{p}(-\tau,0;\mathbb{F}))$ such that
	\begin{equation}
		\| \phi_{x,\rho}(\cdot) \|_{\mathcal{E}_{p}(0,T;L_{p}(-\tau,0;\mathbb{F}))} \leq \rho_{0} \tau^{1-1/p} \cdot \| \phi_{x,\rho}(\cdot) \|_{\mathcal{T}^{p}_{\rho}(0,T;L_{p}(-\tau,0;\mathbb{F}))},
	\end{equation}
	where $\rho_{0}$ is given by \eqref{EQ: WeightFunctionProperty}.
\end{lemma}

Using Lemma \ref{LEM: EmbeddingTwistedIntoEmbracingNeutral} and Theorem \ref{TH: EmbracingSpaceFucntionalIGammaNeutral}, we get the following.
\begin{theorem}
	\label{TH: OperatorIcTwistedLp}
	Suppose $\rho(\cdot)$ is positive and let $\gamma$ and $C^{\gamma}$ be as in \eqref{EQ: OperatorCgammaDefinitionNeutral}. Then for $T>0$ or $T = \infty$ and $p \geq 1$ there exists a bounded linear operator
	\begin{equation}
		\mathcal{I}_{C^{\gamma}} \colon \mathcal{T}^{p}_{\rho}(0,T;L_{p}(-\tau,0;\mathbb{F})) \to L_{p}(0,T;\mathbb{M}_{\gamma})
	\end{equation}
	with the norm not exceeding $\rho_{0} \cdot \tau^{1-1/p} \cdot \operatorname{Var}_{[-\tau,0]}(\gamma)$ and such that for any $\psi_{y,\rho}(\cdot) \in \mathcal{T}^{p}_{\rho}(0,T;L_{p}(-\tau,0;\mathbb{F}))$ with $y(\cdot)$ from $C([0,T];C_{0+}([-\tau,0];\mathbb{F}))$ we have
	\begin{equation}
		\label{EQ: OperatorICgammaTwistedContinuousIdentity}
		(\mathcal{I}_{C^{\gamma}}\psi_{y,\rho})(t) = C^{\gamma}\psi_{y,\rho}(t) \text{ for all } t \in [0,T].
	\end{equation}
\end{theorem}

Let us establish the following technical lemma.
\begin{lemma}
\label{LEM: TwistedCalculationDeltaGeneralFunc}
Suppose $y(\cdot) \in L_{p}(0,T;L_{p}(-\tau,0;\mathbb{F}))$ and $-\tau_{0} \in [-\tau,0]$ is fixed. Then we have
\begin{equation}
	\label{EQ: TwistedFunctionDeltaCalculationLemma}
	(\mathcal{I}_{\delta_{-\tau_{0}}}\psi_{y,\rho})(t) = \rho(t) \int_{\max\{ 0,t-\tau_{0} \}}^{t}y(s)(-\tau_{0}+t-s)ds \text{ for almost all } t \in [0,T].
\end{equation}	
\end{lemma}
\begin{proof}
	We approximate $y(\cdot)$ in $L_{p}(0,T;L_{p}(-\tau,0;\mathbb{F}))$ by a sequence of $y_{k}(\cdot)$, where $k=1,2,\ldots$, from $C([0,T];C_{0+}([-\tau,0];\mathbb{F}))$. Denote $\hat{y}_{k}(t,s) := y_{k}(s)(-\tau_{0}+t-s)$ and $\hat{y}(t,s) := y(s)(-\tau_{0}+t-s)$ for all (resp. almost all) $t \in [0,T]$ and $s \in [\max\{ 0,t-\tau_{0} \},t]$. Note that $\hat{y}(t,\cdot)$ is a well-defined element of $L_{p}$ on the corresponding line segment for almost all $t \in [0,T]$ and thus for such $t$ the integral in \eqref{EQ: TwistedFunctionDeltaCalculationLemma} is well-defined. From the convergence we get that $\hat{y}_{k}(t,\cdot)$ converges to $\hat{y}(t,\cdot)$ in that $L_{p}$ for almost all $t \in [0,T]$. Since \eqref{EQ: TwistedFunctionDeltaCalculationLemma} for $y_{k}$ is satisfied as in \eqref{EQ: CdeltaPsiYIntFormula}, we may pass to the limit as $k \to \infty$ to get the desired. The proof is finished.
\end{proof}

Let $W^{1,p}_{0+}(-\tau,0;\mathbb{F})$ be the subspace from \eqref{EQ: OperatorAT1Domain}. For $T>0$ we define the subspace $\mathcal{T}^{p}_{\rho}(0,T;W^{1,p}_{0+}(-\tau,0;\mathbb{F}))$ of $\mathcal{T}^{p}_{\rho}(0,T;L_{p})$ consisting of all $\psi_{y,\rho}(\cdot)$ such that $y(\cdot) \in L_{p}(0,T;W^{1,p}_{0+}(-\tau,0;\mathbb{F}))$. Assuming that $\rho(\cdot)$ is a proper $C^{1}$-weight (see above \eqref{EQ: WindowSobolevSpaceNorm}), we endow the space just introduced with the norm
\begin{equation}
	\label{EQ: TwistedSobolevSpaceNorm}
	\| \psi_{y,\rho} (\cdot) \|^{p}_{\mathcal{T}^{p}_{\rho}(0,T;W^{1,p}_{0+}(-\tau,0;\mathbb{F}))} := \| \psi_{y,\rho}(\cdot) \|^{p}_{\mathcal{T}^{p}_{\rho}(0,T;L_{p})} + \| \psi_{y',\rho}(\cdot) \|^{p}_{\mathcal{T}^{p}_{\rho}(0,T;L_{p})} + \| \psi_{y,\dot{\rho}}(\cdot) \|^{p}_{\mathcal{T}^{p}_{\rho}(0,T;L_{p})},
\end{equation}
where $y'(s) := \frac{d}{d\theta}y(s)$ for almost all $s \in [0,T]$ and $\frac{d}{d\theta}$ is the derivative in $W^{1,p}(-\tau,0;\mathbb{F})$. For $T=\infty$ we naturally require that all the norms in \eqref{EQ: TwistedSobolevSpaceNorm} are finite. Clearly, this makes $\mathcal{T}^{p}_{\rho}(0,T;W^{1,p}_{0+}(-\tau,0;\mathbb{F}))$ a Banach space for $p \geq 1$ and a Hilbert space for $p=2$.

We have an analog of Theorem \ref{TH: ICgammaSobolevAdorned} that also establishes the embedding of $\mathcal{T}^{p}_{\rho}(0,T;W^{1,p}_{0+})$ into the embracing space $\mathcal{E}_{p}(0,T;W^{1,p})$ and puts the result into the context of Theorem \ref{TH: EmbracingSobolevDifferentiabilityTheorem}.
\begin{theorem}
	\label{TH: ICgammaSobolevTwisted}
	Suppose $\rho(\cdot)$ is a proper $C^{1}$-weight and let $\gamma$ and $C^{\gamma}$ be as in \eqref{EQ: OperatorCgammaDefinitionNeutral}. Then for $T>0$ or $T=\infty$, $p \geq 1$ and any $y(\cdot) \in L_{p}(0,T;W^{1,p}_{0+}(-\tau,0;\mathbb{F}))$ we have $(\mathcal{I}_{C^{\gamma}}\psi_{y,\rho})(\cdot) \in W^{1,p}(0,T;\mathbb{F})$. Moreover,
	\begin{equation}
		\label{EQ: TwistedICgammaSobolevIdentity}
		\frac{d}{dt}(\mathcal{I}_{C^{\gamma}}\psi_{y,\rho})(t) = \mathcal{I}_{C^{\gamma}}\psi_{y,\dot{\rho}}(t) + \mathcal{I}_{C^{\gamma}}\psi_{y',\rho}(t) + \rho(t)C^{\gamma}y(t) \text{ for almost all } t \in (0,T),
	\end{equation}
    where $y'$ as in \eqref{EQ: TwistedSobolevSpaceNorm}. In particular,
    \begin{equation}
    	\mathcal{I}_{C^{\gamma}} \colon \mathcal{T}^{p}_{\rho}(0,T;W^{1,p}_{0+}(-\tau,0;\mathbb{F})) \to W^{1,p}(0,T;\mathbb{M}_{\gamma})
    \end{equation}
    is a bounded linear operator with the norm not exceeding $\operatorname{Var}_{[-\tau,0]}(\gamma)$ times a constant that depends only on $\rho$, $\dot{\rho}$ and $\tau$ (see Theorem \ref{TH: OperatorIcTwistedLp}).
\end{theorem}
\begin{proof}
	We firstly show the statement for $C^{\gamma} = \delta_{-\tau_{0}}$ for some $-\tau_{0} \in [-\tau,0]$. Moreover, it is sufficient to show \eqref{EQ: TwistedICgammaSobolevIdentity} for a dense subspace and finite $T>0$. For this, we consider $y(\cdot) \in C([0,T];W^{1,p}_{0+}(-\tau,0;\mathbb{F}))$. Then in terms of \eqref{EQ: CdeltaPsiYIntFormula} we have
	\begin{equation}
		\label{EQ: TwistedSobolevSpaceLemmaIdentity}
		\frac{d}{dt}\left( \delta_{-\tau_{0}} \psi_{y,\rho}(t) \right) = \delta_{-\tau_{0}}\psi_{y,\dot{\rho}}(t) + \rho(t)\frac{d}{dt}\int_{\max\{0,t-\tau_{0}\}}^{t}\widetilde{y}(s,-\tau_{0}+t-s)ds. 
	\end{equation}
    Since $y(t)(0) = 0$ for all $t \in [0,T]$, for almost all $t \in [0,T]$ from Lemma \ref{LEM: TwistedCalculationDeltaGeneralFunc} we have
    \begin{equation}
    	\begin{split}
    		\rho(t)\frac{d}{dt}\int_{\max\{0,t-\tau_{0}\}}^{t}\widetilde{y}(s,-\tau_{0}+t-s)ds &= \rho(t) y(t)(-\tau_{0}) + \\ + \rho(t)\int_{\max\{0,t-\tau_{0}\}}^{t} \left(\frac{d}{d\theta}y(s)\right)(-\tau_{0}+t-s)ds
    		&= (\mathcal{I}_{\delta_{-\tau_{0}}}\psi_{y',\rho})(t) + \rho(t) \delta_{-\tau_{0}}y(t).
    	\end{split}
    \end{equation}
    Thus we showed \eqref{EQ: TwistedICgammaSobolevIdentity} for $C^{\gamma} = \delta_{-\tau_{0}}$. Now one can use approximations of general $C^{\gamma}$ by operator-linear combinations $C^{\gamma_{k}}$ of $\delta$-functionals as in \eqref{EQ: ConvergenceDeltaFunc} and the pointwise convergence of $\mathcal{I}_{C^{\gamma_{k}}}$ to $\mathcal{I}_{C^{\gamma}}$ in $\mathcal{T}^{p}_{\rho}(0,T;L_{p})$. The proof is finished.
\end{proof}

\subsection{Spaces of agalmanated functions}
\label{SUBSEC: AgalmanatedFunctionsNeutral}

Now for $T>0$ or $T = \infty$ and $p \geq 1$ we introduce the space
\begin{equation}
	\label{EQ: SpaceAgalmanatedDefinition}
	\mathcal{A}^{p}_{\rho}(0,T;L_{p}(-\tau,0;\mathbb{F})) = \mathcal{Y}^{p}_{\rho}(0,T;L_{p}(-\tau,0;\mathbb{F})) \oplus \mathcal{T}^{p}_{\rho}(0,T;L_{p}(-\tau,0;\mathbb{F})),
\end{equation}
where $\oplus$ is the outer direct sum. Thus $\mathcal{A}^{p}_{\rho}(0,T;L_{p}(-\tau,0;\mathbb{F}))$ is the space of pairs $(\phi_{x,\rho},\psi_{y,\rho})$ with $\phi_{x,\rho}(\cdot) \in \mathcal{Y}^{p}_{\rho}(0,T;L_{p})$ (see \eqref{EQ: NorminHTP}) and $\psi_{y,\rho}(\cdot) \in \mathcal{T}^{p}_{\rho}(0,T;L_{p})$ (see \eqref{EQ: TwistedLpNorm}). For brevity, we will often write $\mathcal{A}^{p}_{\rho}(0,T;L_{p})$ to denote the space just introduced. 

In the above terms, we endow $\mathcal{A}^{p}_{\rho}(0,T;L_{p})$ with the norm given by
\begin{equation}
	\| (\phi_{x,\rho}(\cdot),\psi_{y,\rho}(\cdot)) \|^{p}_{\mathcal{A}^{p}_{\rho}(0,T;L_{p})} := \|\phi_{x,\rho}(\cdot)\|^{p}_{\mathcal{Y}^{p}_{\rho}(0,T;L_{p})} + \|\psi_{y,\rho}(\cdot)\|^{p}_{\mathcal{T}^{p}_{\rho}(0,T;L_{p})}.
\end{equation}

By combining Lemma \ref{LEM: EmbeddingAdornedIntoEmbracingNeutral}, Lemma \ref{LEM: EmbeddingTwistedIntoEmbracingNeutral} and Lemma \ref{LEM: SumAdornedTwistedUniq}, we immediately obtain the following.
\begin{theorem}
	\label{TH: AgalmanatedEmbeddingLp}
	For $p > 1$ and $T > 0 $ or $T = \infty$ the mapping
	\begin{equation}
		\mathcal{A}^{p}_{\rho}(0,T;L_{p}(-\tau,0;\mathbb{F})) \ni (\phi_{x,\rho},\psi_{y,\rho}) \mapsto \phi_{x,\rho}(\cdot) + \psi_{y,\rho}(\cdot) \in \mathcal{E}_{p}(0,T;L_{p}(-\tau,0;\mathbb{F}))
	\end{equation}
    is a continuous embedding and its norm admits an estimate depending only on $\rho$ and $\tau$. 
\end{theorem}

For $p > 1$ we will usually refer to $\mathcal{A}^{p}_{\rho}(0,T;L_{p})$ in the sense of its image in the embracing space $\mathcal{E}_{p}(0,T;L_{p}(-\tau,0;\mathbb{F}))$ as in Theorem \ref{TH: AgalmanatedEmbeddingLp} or, by Lemma \ref{LEM: EmbracingSpaceDescriptionCLp}, in $L_{p}(0,T;L_{p}(-\tau,0;\mathbb{F}))$. 

We call any $\phi \in \mathcal{A}^{p}_{\rho}(0,T;L_{p})$ a $\rho$-\textit{agalmanated}\footnote{This name comes from the Ancient Greek word $\alpha\gamma\alpha\lambda\mu\alpha$ (agalma) that means an offering to a deity that, by its worth or artistic value, gives him/her special significance. So, ``agalmanated'' semantically can be understood as ``glorified'' or ``adorned with glory''.} $L_{p}(-\tau,0;\mathbb{F})$-valued function on $[0,T]$ or simply say that $\phi(\cdot)$ is $\rho$-agalmanated.

Combining Theorem \ref{TH: AgalmanatedEmbeddingLp} with Theorem \ref{TH: EmbracingSpaceFucntionalIGammaNeutral} and Corollary \ref{EQ: ComputationICgammaEmbracingL1Loc}, we obtain the following.
\begin{theorem}
	\label{TH: AgalmanatedICgammaExistence}
	Suppose $\rho(\cdot)$ is positive and let $\gamma$ and $C^{\gamma}$ be as in \eqref{EQ: OperatorCgammaDefinitionNeutral}. Then for $T>0$ or $T=\infty$ and $p>1$ there exists a bounded linear operator
	\begin{equation}
    	\label{EQ: OperatorICgammaAgalmanatedAction}
		\mathcal{I}_{C^{\gamma}} \colon \mathcal{A}^{p}_{\rho}(0,T;L_{p}(-\tau,0;\mathbb{F})) \to L_{p}(0,T;\mathbb{M}_{\gamma})
	\end{equation}
    with the norm not exceeding $\operatorname{Var}_{[-\tau,0]}(\gamma)$ times a constant which depends only on $\rho$ and $\tau$ and such that for $\phi(\cdot) \in \mathcal{A}^{p}_{\rho}(0,T;L_{p}(-\tau,0;\mathbb{F})) \cap L_{1,loc}(0,T;C([-\tau,0];\mathbb{F}))$ we have
    \begin{equation}
    	\label{EQ: OperatorICgammaAgalmanatedComputationL1}
    	(\mathcal{I}_{C^{\gamma}}(\phi)(t) = C^{\gamma}\phi(t) \text{ for almost all } t \in (0,T),
    \end{equation}
    Moreover, for
    $\phi(\cdot) = \phi_{x,\rho}(\cdot)+\psi_{y,\rho}(\cdot)$, where $\phi_{x,\rho}(\cdot) \in \mathcal{Y}^{p}_{\rho}(0,T;L_{p})$ and $\psi_{y,\rho}(\cdot) \in \mathcal{T}^{p}_{\rho}(0,T;L_{p})$, we have
    \begin{equation}
    	\label{EQ: OperatorICgammaAgalmanatedSumDef}
    	\mathcal{I}_{C^{\gamma}}\phi = \mathcal{I}_{C^{\gamma}}\phi_{x,\rho} + \mathcal{I}_{C^{\gamma}}\psi_{y,\rho}.
    \end{equation}
    where the action of $\mathcal{I}_{C^{\gamma}}$ on $\psi_{x,\rho}$ and $\psi_{y,\rho}$ can be understood in the sense of Theorem \ref{TH: OperatorIcAdornedLp} and Theorem \ref{TH: OperatorIcTwistedLp} respectively.
\end{theorem}
% !TeX spellcheck = en_EN-EnglishUnitedKingdom
\section{Applications to delay equations of neutral type in $\mathbb{R}^{n}$}
\label{SEC: NeutralDelayEquationsApplications}
\subsection{Properties of linear problems: structural Cauchy formula and resolvent estimates}
\label{SUBSEC: NDELinearSystems}
Let us consider the following class of linear delay equations of neutral type in $\mathbb{R}^{n}$ given by
\begin{equation}
	\label{EQ: LinearDelayEqs}
	\frac{d}{dt} \left[ x(t) + D_{0}x_{t} \right] = \widetilde{A}x_{t} + \widetilde{B}\xi(t),
\end{equation}
where $\tau>0$ is fixed and $x_{t}(\cdot)=x(t+\cdot)$ denotes the history function $[-\tau,0]$; $\widetilde{A}, D_{0} \colon C([-\tau,0]; \mathbb{R}^{n}) \to \mathbb{R}^{n}$ are bounded linear operators; $\widetilde{B}$ is an $(n \times m)$-matrix and $\xi(\cdot) \in L_{2}(0,T;\Xi)$, where $\Xi$ is the control space given by $\mathbb{R}^{m}$ endowed with some (not necessarily Euclidean) inner product. We assume that $D_{0}$ is non-atomic at zero (see \eqref{EQ: OperatorD0Representation} below).

By the Riesz representation theorem, $\widetilde{A}$ is given by an $(n \times n)$-matrix-valued function of bounded variation $a(\theta)$, where $\theta \in [-\tau,0]$, as
\begin{equation}
	\label{EQ: LinearOperatorRepresentation}
	\widetilde{A}\phi = \int_{-\tau}^{0} da(\theta)\phi(\theta) \text{ for } \phi \in C([-\tau,0];\mathbb{R}^{n}).
\end{equation}
Analogously, there exists an $(n \times n)$-matrix-valued function of bounded variation $\delta_{D_{0}}(\theta)$, where $\theta \in [-\tau,0]$, such that
\begin{equation}
	\label{EQ: OperatorD0Representation}
	D_{0} \phi = \int_{-\tau}^{0}d\delta_{D_{0}}(\theta) \phi(\theta) \text{ for } \phi \in C([-\tau,0];\mathbb{R}^{n}).
\end{equation}
We assume that the total variation $\operatorname{Var}_{[-\varepsilon,0]}(\delta_{D_{0}})$ of $\delta_{D_{0}}$ on $[-\varepsilon,0]$ tends to $0$ as $\varepsilon \to 0$.

Let us consider the Hilbert space $\mathbb{H} := \mathbb{R}^{n} \times L_{2}(-\tau,0;\mathbb{R}^{n})$ and the operator $A \colon \mathcal{D}(A) \subset \mathbb{H} \to \mathbb{H}$ given by
\begin{equation}
	\label{EQ: OperatorAdelayDefinition}
	(x,\phi) \mapsto \left(\widetilde{A}\phi, \frac{d}{d \theta}\phi \right),
\end{equation}
where
\begin{equation}
	\label{EQ: DomainNeutralDelayOperator}
	(x,\phi) \in \mathcal{D}(A) := \{ (y,\psi) \in \mathbb{H} \ | \ \psi \in W^{1,2}(-\tau,0;\mathbb{R}^{n}) \text{ and } \psi(0)+D_{0}\psi=y \}.
\end{equation}
Note that $W^{1,2}(-\tau,0;\mathbb{R}^{n})$ is naturally continuously embedded into $C([-\tau,0];\mathbb{R}^{n})$, so $D_{0}\psi$ and $\widetilde{A}\psi$ are well-defined for any $\psi \in W^{1,2}(-\tau,0;\mathbb{R}^{n})$.

It is clear that $A$ is a closed operator. From Lemma 2.2 in \cite{BurnsHerdmanStech1983} it follows that the domain $\mathcal{D}(A)$ is dense in $\mathbb{H}$. Moreover, by Theorem 2.3 in \cite{BurnsHerdmanStech1983}, we have that $A$ is the generator of a $C_{0}$-semigroup in $\mathbb{H}$, which, as usual, will be denoted by $G(t)$ for $t \geq 0$. 

Define the operator $B \colon \Xi \to \mathbb{H}$ as $B\xi:=(\widetilde{B}\xi,0)$. Now we can associate with \eqref{EQ: LinearDelayEqs} a control system given by the pair $(A,B)$ and the spaces $\mathbb{H}$ and $\Xi$. In fact, we are interested in applications of the Frequency Theorem to the pair $(A+\nu I,B)$ for some $\nu \in \mathbb{R}$. 

Since we are aimed to construct Lyapunov functionals without assumptions on the controllability, we have to study the extended control system associated with the pair $(A+\nu I,B)$. In our context, in terms of \eqref{EQ: AuxiliarySpacesEmbeddings} we have $\mathbb{E}_{0} = \mathbb{W} = \mathbb{H}$ and it will be sufficient to consider the linear inhomogeneous system
\begin{equation}
	\label{EQ: LinearInhomogeneousSystemDelay}
	\dot{v}(t) = (A+\nu I)v(t) + \eta(t)
\end{equation}
with $\eta(\cdot) \in L_{2}(0,T;\mathbb{H})$.

The following theorem establishes the important property of solutions to \eqref{EQ: LinearInhomogeneousSystemDelay} related to the spaces of agalmanated functions considered in Section \ref{SUBSEC: AgalmanatedFunctionsNeutral}. Note that \eqref{EQ: StructuralCauchyFormula} gives a decomposition of the solution part into a sum of adorned and twisted functions which we call a \textit{structural Cauchy formula}. Recall that, according to Lemma \ref{LEM: SumAdornedTwistedUniq}, such a decomposition is unique. For the proof, we adapt the method of J.A.~Burns, T.L.~Herdman and H.W.~Stech described in Lemma 2.6 from their paper \cite{BurnsHerdmanStech1983}, where the case of $\eta_{2}(\cdot) \equiv 0$ is considered and, consequently, the second term in \eqref{EQ: StructuralCauchyFormula} do not appear. In further sections, the structural Cauchy formula will be used to interpret integral quadratic functionals and verify the associated assumption on the Fourier transform \nameref{DESC:FT} for the extended control system (see Lemma \ref{LEM: QuadraticFormDelayQf}).
\begin{theorem}
	\label{TH: DelayRegCond}
	Suppose $T>0$ and $\eta_{\nu}(\cdot) \in L_{2}(0,T;\mathbb{H})$. Let $v_{\nu}(\cdot)$ be a mild solution on $[0,T]$ to \eqref{EQ: LinearInhomogeneousSystemDelay} with $\eta = \eta_{\nu}$ such that $v(0)=v_{0} \in \mathbb{H}$ for some $v_{0} \in \mathbb{H}$. Moreover, let $v_{\nu}(\cdot) = (y_{\nu}(\cdot),\phi_{\nu}(\cdot))$ for $y_{\nu}(\cdot) \in C([0,T];\mathbb{R}^{n})$ and $\phi_{\nu}(\cdot) \in C([0,T];L_{2}(-\tau,0;\mathbb{R}^{n}))$ and let $\eta_{\nu}(\cdot) = (\eta^{1}_{\nu}(\cdot),\eta^{2}_{\nu}(\cdot))$ for $\eta^{1}_{\nu}(\cdot) \in L_{2}(0,T;\mathbb{R}^{n})$ and $\eta^{2}_{\nu}(\cdot) \in L_{2}(0,T;L_{2}(-\tau,0;\mathbb{R}^{n}))$. Then there exists $x(\cdot) \in L_{2}(-\tau,T;\mathbb{R}^{n})$ such that
	\begin{equation}
		\label{EQ: StructuralCauchyFormula}
		\phi_{\nu}(t) = \rho_{\nu}(t)x_{t} + \rho_{\nu}(t)\int_{0}^{t}T_{1}(t-s)e^{-\nu s}\eta^{2}_{\nu}(s)ds \text{ for all } t \in [0,T],
	\end{equation}
    where $\rho_{\nu}(t) = e^{\nu t}$ and $T_{1}(t)$, $t \geq 0$, is the semigroup of left translates in $L_{2}(-\tau,0;\mathbb{R}^{n})$ given by \eqref{EQ: DefinitionLeftTranslates}. In particular, $\phi_{\nu}(\cdot)$ belongs to the space $\mathcal{A}^{2}_{\rho}(0,T;L_{2}(-\tau,0;\mathbb{R}^{n}))$ from \eqref{EQ: SpaceAgalmanatedDefinition}. If $v_{0}=(y_{0},\phi_{0}) \in \mathcal{D}(A)$ and $\eta^{2}_{\nu}(\cdot) \in L_{2}(0,T;W^{1,2}_{0+}(-\tau,0;\mathbb{R}^{n}))$, where $W^{1,2}_{0+}(-\tau,0;\mathbb{R}^{n})$ is given by \eqref{EQ: OperatorAT1Domain}, we have
    \begin{equation}
    	\label{EQ: NeutralStructuralCauchyRegularityOfSolutions}
    	x(\cdot) \in W^{1,2}(-\tau,T;\mathbb{R}^{n}) \text{ and } \int_{0}^{\cdot}T_{1}(\cdot-s)e^{-\nu s} \eta^{2}_{\nu}(s)ds \in L_{2}(0,T;W^{1,2}_{0+}(-\tau,0;\mathbb{R}^{n})).
    \end{equation}
    
    In addition, we have
    \begin{equation}
    	y_{\nu}(t) = (\mathcal{I}_{D}\phi_{\nu})(t) \text{ for almost all } t \in \mathbb{R},
    \end{equation}
    where $D\phi = \phi(0) + D_{0}\phi$ for all $\phi \in C([-\tau,0;\mathbb{R}^{n}])$ and the operator $\mathcal{I}_{D}$ is given by Theorem \ref{TH: AgalmanatedICgammaExistence} for $p=2$, $\rho = \rho_{\nu}$ and $\mathbb{F}=\mathbb{M}=\mathbb{R}^{n}$.
	
	Moreover, for $\phi_{x,\rho_{\nu}}(t) = \rho_{\nu}(t)x_{t}$ and $\psi_{\eta^{2}_{\nu},\rho_{\nu}}(t) = \rho_{\nu}(t) \int_{0}^{t}T_{1}(t-s)e^{-\nu s}\eta^{2}_{\nu}(s)ds$ we have the estimates
	\begin{equation}
		\label{EQ: NeutralDelayStructuralFormulaEstimates}
		\begin{split}
			&\| \psi_{\eta^{2}_{\nu},\rho_{\nu}}(\cdot) \|^{2}_{\mathcal{T}^{2}_{\rho_{\nu}}(0,T;L_{2})} \leq \|\eta_{\nu}(\cdot)\|^{2}_{L_{2}(0,T;\mathbb{H})},\\
			&\| \phi_{x,\rho_{\nu}}(\cdot) \|^{2}_{\mathcal{Y}^{2}_{\rho_{\nu}}(0,T;L_{2})} \leq C(\nu,\tau) \cdot \left( \|v_{\nu}(\cdot)\|^{2}_{L_{2}(0,T;\mathbb{H})} + |v_{\nu}(0)|^{2}_{\mathbb{H}} + |v_{\nu}(T)|^{2}_{\mathbb{H}} + \|\eta_{\nu}(\cdot)\|^{2}_{L_{2}(0,T;\mathbb{H})} \right),
		\end{split}
	\end{equation} 
    where $L_{2}$ in the range of functional spaces stands for $L_{2}(-\tau,0;\mathbb{R}^{n})$; the constant $C(\nu,\tau)>0$ depends on $\operatorname{max}\{ 1, (\int_{-\tau}^{0}e^{-2\nu \theta}d\theta)^{-1} \}$, $\max\{1,e^{-\nu \tau} \}$ and $\tau$ in a monotonically increasing way and does not depend on $T$; and the norms on the left-hand side are defined by \eqref{EQ: NorminHTP} and \eqref{EQ: TwistedLpNorm}.
\end{theorem}
\begin{proof}
	Let us firstly show that the estimates from \eqref{EQ: NeutralDelayStructuralFormulaEstimates} are valid under the decomposition \eqref{EQ: StructuralCauchyFormula}. The first inequality is obvious, since by definition (see \eqref{EQ: TwistedLpNorm}) we have
	\begin{equation}
		\| \psi_{\eta^{2}_{\nu},\rho_{\nu}}(\cdot) \|^{2}_{\mathcal{T}^{2}_{\rho_{\nu}}(0,T;L_{2})} = \int_{0}^{T}\|\rho_{\nu}(t) e^{-\nu t} \eta^{2}_{\nu}(t)\|^{2}_{L_{2}}dt = \int_{0}^{T}\|\eta^{2}_{\nu}(t)\|^{2}_{L_{2}}dt.
	\end{equation}
    For the second inequality we have several steps. Firstly, from the Fubini theorem we get for $T \geq \tau$
    \begin{equation}
    	\begin{split}
    		\| \phi_{x,\rho_{\nu}}(\cdot) \|^{2}_{L_{2}(0,T;L_{2})} = \int_{0}^{T}dt \int_{-\tau}^{0}|e^{\nu t}x(t+\theta)|^{2} d\theta = \int_{-\tau}^{0}d\theta \int_{0}^{T}e^{2\nu t}|x(t+\theta)|^{2} dt =\\= \int_{-\tau}^{0} e^{-2 \nu \theta}d\theta \cdot \int_{\theta}^{T+\theta}|e^{\nu t} x(t)|^{2}dt \geq
    	\int_{-\tau}^{0}e^{-2\nu \theta}d\theta \cdot \int_{0}^{T-\tau}|e^{\nu t}x(t)|^{2}dt
     	\end{split}
    \end{equation}
    Using Lemma \ref{LEM: EmbeddingTwistedIntoEmbracingNeutral} and item 2) of Lemma \ref{LEM: EmbracingSpaceDescriptionCLp} with $p=2$ and $\rho = \rho_{\nu}$, we get that the embedding constant of $\mathcal{T}^{2}_{\rho_{\nu}}(0,T;L_{2})$ into $L_{2}(0,T;L_{2})$ can be estimated as $\tau\max\{ 1, e^{-\nu \tau} \}$. From this and \eqref{EQ: StructuralCauchyFormula} we obtain
    \begin{equation}
    	\begin{split}
    		\| \phi_{x,\rho_{\nu}}(\cdot) \|_{L_{2}(0,T;L_{2})} \leq \| \phi_{\nu}(\cdot) \|_{L_{2}(0,T;L_{2})} + \| \psi_{\eta^{2}_{\nu},\rho_{\nu}}(\cdot)  \|_{L_{2}(0,T;L_{2})} \leq \\ \leq \| v_{\nu}(\cdot) \|_{L_{2}(0,T;\mathbb{H})} + \tau\max\{ 1, e^{-\nu \tau} \} \cdot \| \psi_{\eta^{2}_{\nu},\rho_{\nu}}(\cdot) \|_{\mathcal{T}^{2}_{\rho_{\nu}}(0,T;L_{2})}.
    	\end{split}
    \end{equation}
    Now from \eqref{EQ: StructuralCauchyFormula} with $t=0$ we have
    \begin{equation}
    	\int_{-\tau}^{0}|x(\theta)|^{2}d\theta = \| \phi_{x,\rho_{\nu}}(0) \|^{2}_{L_{2}} = \| \phi_{\nu}(0) \|^{2}_{L_{2}} \leq |v_{\nu}(0)|^{2}_{\mathbb{H}}.
    \end{equation}
    Again using \eqref{EQ: StructuralCauchyFormula} with $t=T$, we obtain
    \begin{equation}
    	\begin{split}
    		\left(\int_{\max\{0, T-\tau\}}^{T}|\rho_{\nu}(t)x(t)|^{2}dt\right)^{1/2} \leq \max\{ 1, e^{-\nu \tau} \} \cdot \| \phi_{x,\rho_{\nu}}(T)\|_{L_{2}} \leq \\ \leq \max\{ 1, e^{-\nu \tau} \} \cdot |v_{\nu}(T)|_{\mathbb{H}} + \max\{ 1, e^{-\nu \tau} \} \cdot \| \psi_{\eta^{2}_{\nu},\rho_{\nu}}(T) \|_{L_{2}}
    	\end{split}
    \end{equation}
    and for the last term we have the estimate
    \begin{equation}
    	\begin{split}
    		 \| \psi_{\eta^{2}_{\nu},\rho_{\nu}}(T) \|^{2}_{L_{2}} = \left\| \rho_{\nu}(T) \int_{0}^{T}T_{1}(t-s)e^{-\nu s}\eta^{2}_{\nu}(s)ds \right\|^{2}_{L_{2}} \leq \\ \leq \tau \cdot \max\{ 1, e^{-\nu \tau} \} \cdot \int_{\max\{0, T-\tau\}}^{T}\|\eta^{2}_{\nu}(t)\|^{2}_{L_{2}}dt.
    	\end{split}
    \end{equation}
    By combining the above estimates, we obtain the second inequality from \eqref{EQ: NeutralDelayStructuralFormulaEstimates}.
	
	Now we are going to show the decomposition \eqref{EQ: StructuralCauchyFormula}. For this, we put $v(t)=(y(t),\phi(t)) := e^{-\nu t} v_{\nu}(t)$ and $\eta(t) = (\eta^{1}(t),\eta^{2}(t)) := e^{-\nu t} \eta_{\nu}(t)$. Then $v(\cdot)$ and $\eta(\cdot)$ solves \eqref{EQ: LinearInhomogeneousSystemDelay} with $\nu = 0$. Let us consider the part of the equation for $y(\cdot)$, i.e.
	\begin{equation}
		\label{EQ: StructuralCauchyFormulaBoundaryEquation}
		\dot{y}(t) = \widetilde{A}\phi(t) + \eta^{1}(t) \text{ for } t \in [0,T].
	\end{equation}
    Let $v(0) = (y_{0},\phi_{0})$ for $y_{0} \in \mathbb{R}$ and $\phi_{0} \in L_{2}(-\tau,0;\mathbb{R}^{n})$. Formally substituting the desired identities $y(t) = D\phi(t)$ and $\phi(t) = x_{t} + \int_{0}^{t}T_{1}(t-s)\eta^{2}(s)ds$ into \eqref{EQ: StructuralCauchyFormulaBoundaryEquation} and integrating it over $[0,t]$, we obtain the identity for $x(\cdot) \in L_{2}(-\tau,T;\mathbb{R}^{n})$ as
    \begin{equation}
    	\label{EQ: NeutralStructuralCauchyOperatorS}
    	x(t) = (Sx)(t) = \begin{cases}
    		y_{0} - D_{0}x_{t}  + \int_{0}^{t}\widetilde{A}x_{s}ds + P(\eta)(t), &\text{ if } t > 0,\\
    		\phi_{0}(t), &\text{ if } t \in [-\tau,0],
    	\end{cases}
    \end{equation}
    where
    \begin{equation}
    	P(\eta)(t) = -D \int_{0}^{t}T_{1}(t-s)\eta^{2}(s)ds + \int_{0}^{t}\widetilde{A}\int_{0}^{s}T_{1}(s-h)\eta^{2}(h)dhds + \int_{0}^{t} \eta^{1}(s)ds.
    \end{equation}

    We want to consider $S$ as an operator in $W^{1,2}(-\tau,\varepsilon;\mathbb{R}^{n})$ for some $\varepsilon>0$. For this, we assume that $\phi_{0} \in W^{1,2}(-\tau,0;\mathbb{R}^{n})$ and $D\phi_{0} = y_{0}$. Moreover, we assume that $\eta_{2}(\cdot) \in L_{2}(0,T;W^{1,2}_{0+}(-\tau,0;\mathbb{R}^{n}))$. Then Theorem \ref{TH: ICgammaSobolevTwisted} guarantees that $P(\eta)(\cdot) \in W^{1,2}(0,T;\mathbb{R}^{n})$ if the actions of $D$ and $\widetilde{A}$ on $\int_{0}^{t}T_{1}(t-s)\eta^{2}(s)ds$ are interpreted according to the operators $\mathcal{I}_{D}$ and $\mathcal{I}_{\widetilde{A}}$ given by the theorem. Analogously, Theorem \ref{TH: ICgammaSobolevAdorned} implies that the action of $D_{0}$ on $x_{t}$ can be interpreted in such a way that $t \mapsto D_{0}x_{t}$ will belong to $W^{1,2}(0,T;\mathbb{R}^{n})$. For the remained term $\int_{0}^{t}\widetilde{A}x_{s}ds$ there are no problems since $s \mapsto \widetilde{A}x_{s}$ is continuous and, consequently, the integral is a $C^{1}$-differentiable function of $t$.
    
    Thus, under the above imposed conditions, the operator $S$ becomes a self-map of $W^{1,2}(-\tau,\varepsilon;\mathbb{R}^{n})$. Note that for any $\varepsilon \in [0,T]$ the embedding constant of $W^{1,2}(-\tau,\varepsilon;\mathbb{R}^{n})$ into $C([-\tau,\varepsilon];\mathbb{R}^{n})$ can be estimated by $K>0$ that is independent of $\varepsilon$ (see, for example, Lemma 5.8, p. 100 \cite{Adams1976SobolevSpaces}). Then for any $x^{1}(\cdot),x^{2}(\cdot) \in W^{1,2}(-\tau,\varepsilon;\mathbb{R}^{n})$ and $t > 0$ we have
    \begin{equation}
    	\begin{split}
    		 \frac{d}{dt}\int_{0}^{t}\widetilde{A}(x^{1}_{s} - x^{2}_{s})ds = \widetilde{A}(x^{1}_{t} - x^{2}_{t}), \\
    		 \int_{0}^{\varepsilon}| \widetilde{A}(x^{1}_{t}-x^{2}_{t}) |^{2}dt \leq \varepsilon \cdot \|\widetilde{A}\|^{2} \cdot K^{2} \cdot \| x^{1}(\cdot) - x^{2}(\cdot) \|^{2}_{W^{1,2}(-\tau,\varepsilon;\mathbb{R}^{n})}.
    	\end{split}
    \end{equation}
    Let $\rho(\cdot) \equiv 1$. Then Theorem \ref{TH: ICgammaSobolevAdorned} gives a constant $K_{1}>0$ independent of $\varepsilon$ such that
    \begin{equation}
    	\label{EQ: StructureCauchyFormulaID0SobolevEstimate}
    	\| \mathcal{I}_{D_{0}}( \phi_{x^{1},\rho} - \phi_{x^{2},\rho}) \|_{W^{1,2}(0,\varepsilon;\mathbb{R}^{n})} \leq K_{1} \cdot \operatorname{Var}_{[-\tau,0]}(\delta_{D_{0}}) \cdot \| x^{1}(\cdot) - x^{2}(\cdot) \|_{W^{1,2}(-\tau,\varepsilon;\mathbb{R}^{n})}.
    \end{equation}
    Combining the above estimates, for some independent of $\varepsilon$ constant $K_{2}>0$ we obtain that
    \begin{equation}
    	\|(Sx^{1})(\cdot) - (Sx^{2})(\cdot)\|_{W^{1,2}(-\tau,\varepsilon;\mathbb{R}^{n})} \leq (K_{1} \operatorname{Var}_{[-\tau,0]}(\delta_{D_{0}}) + \varepsilon K_{2}) \cdot \| x^{1}(\cdot) - x^{2}(\cdot) \|_{W^{1,2}(-\tau,\varepsilon;\mathbb{R}^{n})}.
    \end{equation}
    This shows that $S$ is a Lipschitz mapping. Let us show that the second iterate of $S$, say $U = S^{2}$, is a contraction for all sufficiently small $\varepsilon>0$. Indeed, since $(Sx_{1})(t)-(Sx_{2})(t) = 0$ for $t \in [-\tau,0]$ and any $x^{1}(\cdot)$, $x^{2}(\cdot) \in W^{1,2}(-\tau,\varepsilon;\mathbb{R}^{n})$, it is clear that for the computation of the left-hand side from \eqref{EQ: StructureCauchyFormulaID0SobolevEstimate} where $x^{1}(\cdot)$ and $x^{2}(\cdot)$ are exchanged with $(Sx^{1})(\cdot)$ and $(Sx^{2})(\cdot)$, only the values of $\delta_{D_{0}}(\theta)$ for $\theta \in [-\varepsilon,0]$ matter. Thus, repeating the estimate \eqref{EQ: StructureCauchyFormulaID0SobolevEstimate} in such a context, we obtain
    \begin{equation}
    	\label{EQ: StructuralCauchyFormulaSecondIterateEstimate}
    	\begin{split}
       	\|(Ux^{1})(\cdot)-(Ux^{2})(\cdot)\|_{W^{1,2}(-\tau,\varepsilon;\mathbb{R}^{n})} \leq \\ \leq (K_{1} \operatorname{Var}_{[-\varepsilon,0]}(\delta_{D_{0}}) + \varepsilon K_{2}) (K_{1} \operatorname{Var}_{[-\tau,0]}(\delta_{D_{0}}) + \varepsilon K_{2}) \cdot \| x^{1}(\cdot) - x^{2}(\cdot) \|_{W^{1,2}(-\tau,\varepsilon;\mathbb{R}^{n})}.
    	\end{split}
    \end{equation}
    By definition, $\operatorname{Var}_{[-\varepsilon,0]}(\delta_{D_{0}})$ tends to $0$ as $\varepsilon \to 0+$ and, consequently, we get that $U$ is a contraction for all sufficiently small $\varepsilon>0$. Since $U=S^{2}$ has a unique fixed point, it is also a unique fixed point for $S$. Repeating the procedure, we may obtain $x(\cdot) \in W^{1,2}(-\tau,T;\mathbb{R}^{n})$ such that $x(t) = (Sx)(t)$ for $t \in [-\tau,T]$.
    
    Now define $\phi(t) := x_{t} + \int_{0}^{t}T_{1}(t-s)\eta^{2}(s)ds$ for $t \in [0,T]$. From this, we immediately get that $y(t) := D\phi(t)$ satisfies \eqref{EQ: StructuralCauchyFormulaBoundaryEquation} on $[0,T]$ as a $C^{1}$-differentiable function if $\eta^{1}(\cdot) \in C([0,T];\mathbb{R}^{n})$. Since $(y(t),\phi(t)) \in \mathcal{D}(A)$ (see \eqref{EQ: DomainNeutralDelayOperator}) for all $t \in [0,T]$, it remains to show that $\phi(t)$ satisfies
    \begin{equation}
    	\label{EQ: StructuralCauchyFormulaPhiEquation}
    	\frac{d}{dt}\phi(t) = \frac{d}{d\theta}\phi(t) + \eta_{2}(t) \text{ for } t \in [0,T],
    \end{equation}
    where $\frac{d}{d\theta}$ is the derivative in $W^{1,2}(-\tau,0;\mathbb{R}^{n})$. Since $x(\cdot) \in W^{1,2}(-\tau,T;\mathbb{R}^{n})$, by Lemma 3.4 from \cite{BatkaiPiazzera2005}, we have that $t \mapsto x_{t}$ is a $C^{1}$-differentiable $L_{2}(-\tau,0;\mathbb{R}^{n})$-valued function of $t \in [0,T]$ and it satisfies
    \begin{equation}
       	\label{EQ: StructuralCauchyFormulaXtEquation}
    	\frac{d}{dt} x_{t} = \frac{d}{d\theta} x_{t} \text{ for } t \in [0,T].
    \end{equation}
    Moreover, assuming that $\eta_{2}(\cdot) \in C^{1}([0,T];L_{2}(-\tau,0;\mathbb{R}^{n}))$, we have (see Theorem 6.5, Chapter I in \cite{Krein1971}) that $t \mapsto \int_{0}^{t}T_{1}(t-s)\eta_{2}(s)ds$ is a classical solution on $[0,T]$ to the inhomogeneous Cauchy problem associated with $A_{T_{1}}$ (see \eqref{EQ: DefinitionLeftTranslates}). In particular, it is a $C^{1}$-differentiable function of $t \in [0,T]$ that belongs to $\mathcal{D}(A_{T_{1}}) = W^{1,2}_{0+}(-\tau,0;\mathbb{R}^{n})$ and satisfies
    \begin{equation}
    	\label{EQ: StructuralCauchyFormulaTwistedEquation}
    	\frac{d}{dt}\int_{0}^{t}T_{1}(t-s)\eta^{2}(s)ds = \frac{d}{d\theta} \int_{0}^{t}T_{1}(t-s)\eta^{2}(s) + \eta_{2}(t) \text{ for all } t \in [0,T].
    \end{equation}
    By combining \eqref{EQ: StructuralCauchyFormulaXtEquation} and \eqref{EQ: StructuralCauchyFormulaTwistedEquation}, we obtain \eqref{EQ: StructuralCauchyFormulaPhiEquation}.
    
    Note that the imposed conditions on $(y_{0},\phi_{0})$ and $\eta(\cdot)$ are represented by dense subspaces in $\mathbb{H}$ and $L_{2}(0,T;\mathbb{H})$ respectively. Since we know\footnote{In fact, the consideration of the operator $S$ from \eqref{EQ: NeutralStructuralCauchyOperatorS} with $\eta(\cdot) = 0$ in $L_{2}(-\tau,T;\mathbb{R}^{n})$ and applying similar arguments to investigate its fixed point along with the fact that $S^{2}$ is a uniform (in initial conditions) contraction show that $A$ generates a $C_{0}$-semigroup.} that $A$ is the generator of a $C_{0}$-semigroup, we may obtain the decomposition \eqref{EQ: StructuralCauchyFormula} by continuity from the already obtained estimates \eqref{EQ: NeutralDelayStructuralFormulaEstimates}. The proof is finished.
\end{proof}

In the following, there is a slight abuse of notation since we have been considering real vector spaces from the beginning of this section and, consequently, the resolvent is defined on the complexification $\mathbb{H}^{\mathbb{C}} = \mathbb{C}^{n} \times L_{2}(-\tau,0;\mathbb{C}^{n})$ of $\mathbb{H}$. In order not to get confused, the reader should think that all the spaces are complex.

Let us consider the matrix-valued functions of $p \in \mathbb{C}$ given by
\begin{equation}
	\label{EQ: AandBoperatorsDelayRepresentation}
	\alpha(p) = \int_{-\tau}^{0}e^{p\theta} da(\theta) \text{ and } \delta(p) = \int_{-\tau}^{0} e^{p \theta} d\delta_{D_{0}}(\theta). 
\end{equation}
These matrices are convenient for computing the resolvent and the spectrum of $A$. Namely, let $I_{n}$ denote the identity $(n\times n)$-matrix and $I$ denote the identity operator in $\mathbb{H}$. Put $p \in \mathbb{C}$ and suppose $(y,\psi) \in \mathbb{H}$ is given. We have to find $(x,\phi) \in \mathcal{D}(A)$ such that $(A-pI)(x,\phi) = (y,\psi)$. According to \eqref{EQ: OperatorAdelayDefinition} and \eqref{EQ: DomainNeutralDelayOperator}, this means that
\begin{equation}
	\begin{split}
		&\frac{d}{d\theta}\phi(\theta) - p\phi(\theta) = \psi(\theta) \text{ in } L_{2}(-\tau,0;\mathbb{C}^{n}),\\
		&\phi(0) + D_{0}\phi = x,\\
		&\widetilde{A}\phi - p x = y.
	\end{split}
\end{equation}
Substituting the second identity to the third, expanding the latter according to \eqref{EQ: LinearOperatorRepresentation} and \eqref{EQ: OperatorD0Representation} and substituting $\phi(\theta)$ obtained from the first identity via the Cauchy formula, gives 
\begin{equation}
	\label{EQ: DelayResolventPropetyLemma}
	\begin{split}
		\phi(\theta) &= e^{p \theta} \phi(0) + \int_{0}^{\theta} e^{p(\theta-s)} \psi(s) ds \text{ for } \theta \in [-\tau,0],\\
		\left( \alpha(p) - p I_{n} - p \delta(p) \right) \phi(0) &= y - \int_{-\tau}^{0}da(\theta) \int_{0}^{\theta} e^{p(\theta-s)}\psi(s)ds + p\int_{-\tau}^{0}d\delta_{D_{0}}(\theta) \int_{0}^{\theta} e^{p(\theta-s)}\psi(s)ds.
	\end{split}
\end{equation}
From \eqref{EQ: DelayResolventPropetyLemma} it is clear that the resolvent $(A-pI)^{-1}$ exists if and only if the matrix $\alpha(p) - p I_{n} - p \delta(p)$ is invertible. Thus, the spectrum of $A$ consists only of $p \in \mathbb{C}$ satisfying
\begin{equation}
	\label{EQ: SpectrumADelay}
	\operatorname{det} \left[\alpha(p) - p I_{n} - p\delta(p)\right] = 0,
\end{equation}
Clearly, $A$ has compact resolvent and thus its spectrum consists of eigenvalues with finite algebraic multiplicity. More detailed spectral properties will be established in Theorem \ref{TH: SpectralDecompositionsDelayNeutral} below.

Let us put $\mathbb{E}:=C([-\tau,0];\mathbb{R}^{n})$ and consider the continuous and dense embedding $\mathbb{E} \subset \mathbb{H}$ given by the map $\phi \mapsto (\phi(0)+D_{0}\phi,\phi)$. As before, we identify $\mathbb{E}$ with its image in $\mathbb{H}$ under the embedding. Clearly, the inclusion embedding $\mathcal{D}(A) \subset \mathbb{E}$ is continuous. 

We have the following lemma.
\begin{lemma}
	\label{LEM: DelayResCond}
	Consider the operator $A$ given by \eqref{EQ: OperatorAdelayDefinition} and let $p \in \mathbb{C}$ be its regular (i.e. non-spectral) point. Then
	\begin{equation}
		\label{EQ: DelayNeutralResolventEstimateEfromH}
		\| (A-pI)^{-1} \|_{\mathcal{L}(\mathbb{H};\mathbb{E}) } \leq C_{1}(p) \cdot \| (A-pI)^{-1} \|_{\mathcal{L}(\mathbb{H})} + C_{2}(p),
	\end{equation}
    where the constants $C_{1}(p)$ and $C_{2}(p)$ depend on $\operatorname{max}\{ e^{-\tau\operatorname{Re}p }, 1 \}$ and $\tau$ in a monotonically increasing way.
\end{lemma}
\begin{proof}
	Let $(A-pI) (x,\phi) = (y,\psi)$ for given $(y,\psi) \in \mathbb{H}$ and $(x,\phi) \in \mathcal{D}(A)$. Using the first identity from \eqref{EQ: DelayResolventPropetyLemma} and that $(x,\phi) = (A-pI)^{-1}(y,\psi)$, we obtain
	\begin{equation}
		\label{EQ: DelayNeutralResolventEstimateZeroValue}
		\begin{split}
			|\phi(0)| \leq \max\{ e^{\tau \operatorname{Re}p}, 1 \} \cdot |(x,\phi)|_{\mathbb{H}} + \sqrt{\tau} \max\{ e^{\tau \operatorname{Re}p}, 1 \} \cdot |(y,\psi)|_{\mathbb{H}} \leq\\
			\leq \left( \max\{ e^{\tau \operatorname{Re}p}, 1 \} \cdot \| (A-pI)^{-1} \|_{\mathcal{L}(\mathbb{H})} + \sqrt{\tau} \max\{ e^{\tau \operatorname{Re}p}, 1 \} \right) \cdot |(y,\psi)|_{\mathbb{H}}.
		\end{split}
	\end{equation}
    Now again using the first identity in \eqref{EQ: DelayResolventPropetyLemma} and \eqref{EQ: DelayNeutralResolventEstimateZeroValue} to estimate $\sup_{\theta \in [-\tau,0]} |\phi(\theta)| = \| (x,\phi) \|_{\mathbb{E}}$, we obtain \eqref{EQ: DelayNeutralResolventEstimateEfromH}. The proof is finished.
\end{proof}

Note that the semigroup $G(t)$ generated by $A$ is not eventually compact in the general case (unlike when there is no neutral part, i.e. $D_{0} = 0$). This places limitations concerned with spectral decompositions and resolvent bounds as in \eqref{EQ: DelayNeutralResolventEstimateEfromH} uniformly on a line $\operatorname{Re} p = \nu$. Following J.K.~Hale \cite{Hale1977}, we consider the following value
\begin{equation}
	\label{EQ: SqueezingDExponent}
	a_{D} := \sup \{ \operatorname{Re}p \ | \ \operatorname{det}\left[I_{n} + \delta(p) \right] = 0 \}.
\end{equation}
In the case the set over which the supremum is taken \eqref{EQ: SqueezingDExponent} is empty, we put $a_{D} := -\infty$. 

To discuss the role of $a_{D}$ in the studying spectrum of $A$, as in \cite{Hale1977}, we assume that $D_{0}$ has no singular part, i.e. there exist $\tau_{k} \in (0,\tau]$ and $m_{k} \in \mathbb{R}$, where $k=1,2,\ldots$, and an $(n\times n)$-matrix-valued function $M(\cdot)$ on $[-\tau,0]$ such that
\begin{equation}
	\label{EQ: D0HasnoSingularPart}
	\begin{split}
		&D_{0} \phi = \sum_{k=1}^{\infty} m_{k}\phi(-\tau_{k}) + \int_{-\tau}^{0}M(\theta)\phi(\theta)d\theta \text{ for any } \phi \in C([-\tau,0];\mathbb{R}^{n}),\\
		&\sum_{k=1}^{\infty}|m_{k}| + \int_{-\tau}^{0}|M(\theta)|d\theta < \infty.
	\end{split}
\end{equation}

The following theorem is an analog of the main result from Section 12.3 in \cite{Hale1977}, where $A$ is considered as an operator in $\mathbb{E} = C([-\tau,0];\mathbb{R}^{n})$. Since the proof in our context is similar, we only indicate some main nuances concerned with the consideration of $A$ in $\mathbb{H}$ and explain why the equation 
\begin{equation}
	\label{EQ: DelayFrictionSemihroupSpectrum}
	\operatorname{det}\left[I_{n} + \delta(p) \right] = 0
\end{equation} 
arises in \eqref{EQ: SqueezingDExponent}. In fact, $a_{D}$ determines the growth bound of a semigroup such that the difference of its time-$t$ mapping with $G(t)$ is a compact operator (see \eqref{EQ: DecompositionGFrictionCompactLinearNeutralDelay}) and, consequently, the essential spectral radii of these semigroups coincide.
\begin{theorem}
	\label{TH: SpectralDecompositionsDelayNeutral}
	Suppose for some $\nu \in \mathbb{R}$ we have $-\nu > a_{D}$, where $a_{D}$ is given by \eqref{EQ: SqueezingDExponent} and $D_{0}$ has no singular part (see \eqref{EQ: D0HasnoSingularPart}). Then we have
	\begin{enumerate}
		\item[1)] There is a finite number of eigenvalues of $A$, i.e. solutions $p=\lambda \in \mathbb{C}$ to \eqref{EQ: SpectrumADelay}, satisfying $\operatorname{Re}\lambda \geq -\nu$;
		\item[2)] Let $\mathbb{H}^{u}(\nu)$ be the finite-dimensional spectral subspace corresponding to the roots with $\operatorname{Re}\lambda \geq -\nu$ and let $\mathbb{H}^{s}(\nu)$ be its complementary spectral subspace. Then for any $\varepsilon>0$ there exists $M_{\varepsilon}$ such that
		\begin{equation}
			\label{EQ: NeutralDelaySpectralDecompositionsEstimateComplementary}
			| G(t)v_{0} |_{\mathbb{H}} \leq M_{\varepsilon} e^{(a_{D}(\nu) + \varepsilon) t} |v_{0}|_{\mathbb{H}} \text{ for all } t \geq 0 \text{ and } v_{0} \in \mathbb{H}^{s}(\nu),
		\end{equation}
	    where $a_{D}(\nu) := \sup\{ \operatorname{Re}\lambda \ | \ \lambda \in \operatorname{Spec}A, \ \operatorname{Re}\lambda < -\nu \}$ and $\operatorname{Spec}A$ denotes the spectrum of $A$.
	    \item[3)] If there is no spectrum of $A$ on the line $-\nu + i \mathbb{R}$, then
	    \begin{equation}
	    	\label{EQ: DelayNeutralResolventBoundedness}
	    	\sup_{\omega \in \mathbb{R}}\| (A - (-\nu + i \omega) I)^{-1} \|_{\mathcal{L}(\mathbb{H})} < \infty.
	    \end{equation}
        In particular, due to Lemma \ref{LEM: DelayResCond} we have \nameref{DESC:RES2} and \nameref{DESC:RES3} satisfied for the pair $(A  + \nu I, B)$.
	\end{enumerate} 
\end{theorem}
\begin{proof}
Let $G_{0}(t)$, where $t \geq 0$, denote the $C_{0}$-semigroup in $\mathbb{H}$ generated by the operator $A_{0}$ given by \eqref{EQ: OperatorAdelayDefinition} and \eqref{EQ: DomainNeutralDelayOperator} with $\widetilde{A} = 0$ and the given $D_{0}$. Then for each $v_{0} \in \mathbb{H}$ we consider $G(t)v_{0}$ as a solution to the inhomogeneous Cauchy problem corresponding to $A_{0}$. Let $\Pi_{1}$ and $\Pi_{2}$ be the orthogonal projectors in $\mathbb{H} = \mathbb{C}^{n} \times L_{2}(-\tau,0;\mathbb{C}^{n})$ onto $\mathbb{C}^{n} \times \{ 0 \}$ and $\{0 \} \times L_{2}(-\tau,0;\mathbb{C}^{n})$ respectively. Then, by the Cauchy formula, we have
\begin{equation}
	G(t)v_{0} = G_{0}(t)v_{0} + \int_{0}^{t}G_{0}(t-s) \begin{bmatrix}
		\widetilde{A} \Pi_{2}G(s)v_{0}\\
		0
	\end{bmatrix}ds \text{ for all } t \geq 0,
\end{equation}
where $\widetilde{A} \Pi_{2}G(s)v_{0}$ is understood as $\mathcal{I}_{\widetilde{A}}(\Pi_{2}G(\cdot)v_{0})(s)$ according to Theorem \ref{TH: OperatorIcAdornedLp} that is justified by Theorem \ref{TH: DelayRegCond} (note that we have $\eta^{2}_{\nu} \equiv 0$ in the context of the latter).

Note that $\mathbb{C}^{n} \times \{ 0 \}$ and $\{0 \} \times L_{2}(-\tau,0;\mathbb{C}^{n})$ are invariant subspaces of $\mathbb{H}$ w.r.t. $G_{0}(t)$. From this we obtain that the operator
\begin{equation}
	v_{0} \in \mathbb{H} \mapsto G(t)v_{0}-G_{0}(t)v_{0} = \int_{0}^{t}G_{0}(t-s) \begin{bmatrix}
		\widetilde{A} \Pi_{2}G(s)v_{0}\\
		0
	\end{bmatrix}ds \in \mathbb{C}^{n} \times \{ 0 \}
\end{equation}
is compact for each $t \geq 0$ since its image lies in the finite-dimensional subspace $\mathbb{C}^{n} \times \{ 0 \}$.

So, we have that $G_{0}(t) = G_{0}(t)\Pi_{2} + G_{0}(t)\Pi_{1}$, where $G_{0}(t)\Pi_{1}$ is a $C_{0}$-semigroup in $\mathbb{C}^{n} \times \{ 0 \}$. Thus, for all $t \geq 0$ we get that
\begin{equation}
	\label{EQ: DecompositionGFrictionCompactLinearNeutralDelay}
	G(t) = G_{0}(t)\Pi_{2} + U(t), \text{ where } U(t) \text{ is a compact operator}.
\end{equation}
Note that the spectrum of the generator of  $G_{0}(t)\Pi_{2}$ is determined by \eqref{EQ: DelayFrictionSemihroupSpectrum}. Thus, according to \eqref{EQ: DecompositionGFrictionCompactLinearNeutralDelay}, Corollary 2.11 on p.~258 and Theorem 3.1 on p.~329  from \cite{EngelNagel2000}, for items 1) and 2) of the theorem it is sufficient to show that the growth bound of $G_{0}(t)\Pi_{2}$ is determined by $a_{D}$. But this can be done following the same arguments as in Lemma 3.3 and Theorem 3.4 from Section 12.3 in \cite{Hale1977}, where this is shown in the context of the restriction of $G_{0}(t)\Pi_{2}$ to $\mathbb{E} \cap \left(\{0\}\times L_{2}(-\tau,0;\mathbb{C}^{n})\right)$ that is the subspace of $\phi \in C([-\tau,0];\mathbb{C}^{n})$ such that $\phi(0) + D_{0}\phi = 0$.

For item 3) we use spectral decompositions. Namely, let $\Pi^{s}$ and $\Pi^{u}$ be the complementary spectral projectors onto $\mathbb{H}^{s}(\nu)$ and $\mathbb{H}^{u}(\nu)$ respectively. Then for $p = -\nu + i \omega$ and all $w_{0} \in \mathbb{H}$ we have
\begin{equation}
	\label{EQ: DelayNeutralResolventDecompositionBoundedness}
	(A-pI)^{-1}w_{0} = (A-pI)^{-1}\Pi^{s}w_{0} + (A-pI)^{-1}\Pi^{u}w_{0} = \int_{0}^{\infty}e^{-p t}G(t)\Pi^{s}w_{0}dt + (A-pI)^{-1}\Pi^{u}w_{0},
\end{equation}
where we used the well-known representation of the resolvent through the Laplace transform of the semigroup (see Theorem 1.10 on p. 55 in \cite{EngelNagel2000}). From \eqref{EQ: DelayNeutralResolventDecompositionBoundedness} we obtain \eqref{EQ: DelayNeutralResolventBoundedness} since $\Pi^{u}$ has finite rank and the estimate \eqref{EQ: NeutralDelaySpectralDecompositionsEstimateComplementary} holds. The proof is finished.
\end{proof}

\begin{remark}
	\label{REM: ResolventEstimateNeutralDelay}
	Suppose we have the inequality $\| D_{0} \| \cdot \operatorname{max}\{e^{\nu \tau},1\} < 1$ satisfied, which is encountered in the case of equations without the neutral term (i.e. with $D_{0} \equiv 0$) and certain equations with small delays (see Section \ref{SUBSEC: EquationsWithSmallDelays}). Then from \eqref{EQ: SqueezingDExponent} it is clear that $a_{D} < -\nu$ since $|\delta(p)| < 1$ for all $p$ with $\operatorname{Re}p \geq -\nu$. Moreover, in this case we also have that the norms of $(A-pI)^{-1}B$ in $\mathcal{L}(\Xi;\mathbb{E})$ for $p=-\nu + i\omega$ tend to zero as $\omega \to \infty$. This follows from \eqref{EQ: DelayResolventPropetyLemma} where we deal with the case $\psi \equiv 0$ due to the boundary action of $B$. This property guarantees that the frequency inequalities as \eqref{EQ: Frequency-domainDelay} below may be checked only for $\omega$ varying in a finite interval. The latter is important for numerical verification of such inequalities. However, it seems that this property fails to hold for general delay equations of neutral type.
\end{remark}
\subsection{Nonlinear problems and linear inhomogeneous systems with quadratic constraints}
\label{SUBSEC: NDENonlinear}
Let us consider a nonlinear version of \eqref{EQ: LinearDelayEqs} given by
\begin{equation}
\label{EQ: NonlinearDelayEqs}
\frac{d}{dt} \left[ x(t) + D_{0}x_{t} \right] = \widetilde{A}x_{t} + \widetilde{B}F(t,Cx_{t})+\widetilde{W}(t),
\end{equation}
where $D_{0}$, $\widetilde{A}$ and $\widetilde{B}$ are as in \eqref{EQ: LinearDelayEqs}, $F \colon \mathbb{R} \times \mathbb{R}^{r} \to \mathbb{R}^{m}$ is a nonlinear continuous function, $\widetilde{W} \colon \mathbb{R} \to \mathbb{R}^{n}$ is a bounded continuous function, $C \colon C([-\tau,0];\mathbb{R}^{n}) \to \mathbb{R}^{r}$ is a bounded linear operator. For discussions on the well-posedness of \eqref{EQ: NonlinearDelayEqs} we refer below the text.

Recall that we do not distinguish between the elements of $\mathbb{E} = C([-\tau,0];\mathbb{R}^{r})$ and their images under the embedding $\phi \mapsto (\phi(0) + D_{0}\phi, \phi)$ into $\mathbb{H} = \mathbb{R}^{n} \times L_{2}(-\tau,0;\mathbb{R}^{n})$. Moreover, we keep the same notation for the induced by the embedding operators. Thus, we put $Cv := C\phi$ for $v=(\phi(0)+D_{0}\phi, \phi) \in \mathbb{H}$ and $\phi \in C([-\tau,0];\mathbb{R}^{r})$.

To study \eqref{EQ: NonlinearDelayEqs} by means of Theorem \ref{TH: FrequencyTheoremRealOperator} we have to consider a quadratic form $\mathcal{F}(v,\xi)$, where $v \in \mathbb{E}$ and $\xi \in \Xi = \mathbb{R}^{m}$, which takes into account some properties of the nonlinearity $F$. Namely, we will use the following property.
\begin{equation}
		\label{EQ: MonotoneSectorCondtionGeneral}
		\mathcal{F}(v_{1}-v_{2},F(t,Cv_{1}) - F(t,Cv_{2})) \geq 0 \text{ for any } t \in \mathbb{R} \text{ and } v_{1}, v_{2} \in \mathbb{E}.
\end{equation}
We call \eqref{EQ: MonotoneSectorCondtionGeneral} the \textit{monotone sector condition}. We will also require that $\mathcal{F}(v,0) \geq 0$ for all $v \in \mathbb{E}$. To understand such a name, consider, for example, the case of $F \colon \mathbb{R} \times \mathbb{R} \to \mathbb{R}$ (i.e. $m=1, r=1$) such that for some numbers $\varkappa_{1}, \varkappa_{2} \in \mathbb{R}$, $\varkappa_{1} \leq 0 < \varkappa_{2}$ we have
\begin{equation}
	\label{EQ: SectorCondMonotone}
	\varkappa_{1} \leq \frac{F(t,y_{1}) - F(t,y_{2})}{y_{1}-y_{2}} \leq \varkappa_{2} \text{ for all } t,y_{1},y_{2} \in \mathbb{R} \text{ with } y_{1}\not=y_{2}.
\end{equation}
For this case, an appropriate choice of $\mathcal{F}$ will be $\mathcal{F}(v,\xi) = (\xi - \varkappa_{1}Cv) (\varkappa_{2}Cv - \xi)$. It is easy to check that $\mathcal{F}$ satisfies \eqref{EQ: MonotoneSectorCondtionGeneral}. 

For more general situations it is convenient to consider the following class of quadratic forms $\mathcal{F}$, which can be used for Lipschitz nonlinearities $F \colon \mathbb{R} \times \mathbb{R}^{r} \to \mathbb{R}^{m}$. Suppose that the spaces $\Xi := \mathbb{R}^{m}$ and $\mathbb{M}:= \mathbb{R}^{r}$ are endowed with some (not necessarily Euclidean) inner products. Assume that there exists a constant $\Lambda > 0$ such that
\begin{equation}
		\label{EQ: LIPFcondition}
		|F(t,y_{1}) - F(t,y_{2})|_{\Xi} \leq \Lambda |y_{1}-y_{2}|_{\mathbb{M}} \text{ for any } t \in \mathbb{R} \text{ and } y_{1},y_{2} \in \mathbb{M}.
\end{equation}
Then it is clear that under \eqref{EQ: LIPFcondition} the form $\mathcal{F}(v,\xi):= \Lambda^{2} |Cv|^{2}_{\mathbb{M}} - |\xi|^{2}_{\Xi}$ satisfies \eqref{EQ: MonotoneSectorCondtionGeneral}. Note that the consideration of different inner products in $\mathbb{R}^{m}$ and $\mathbb{R}^{r}$ adds some flexibility (at least theoretically) to the corresponding frequency inequality \eqref{EQ: DelaySmithCondition}.

Now we see that the considered quadratic forms naturally include the operator $C$. Let us generalize the above construction. Let $\mathcal{G}(y, \xi)$ be a quadratic form of $y \in \mathbb{M}$ and $\xi \in \Xi$, where $\mathbb{M}=\mathbb{R}^{r}$ and $\Xi = \mathbb{R}^{m}$ are endowed with some (not necessarily Euclidean) inner products. Then we put $\mathcal{F}(v,\xi) := \mathcal{G}(Cv,\xi)$. To describe the Hermitian extension $\mathcal{F}^{\mathbb{C}}$ of $\mathcal{F}$, suppose that $\mathcal{G}$ is given by
\begin{equation}
	\label{EQ: NeutralDelayQuadraticFormG}
	\mathcal{G}(y, \xi) = (y, \mathcal{G}_{y}y)_{\mathbb{M}} + (\xi,\mathcal{G}_{y\xi}y)_{\Xi} + (\xi,\mathcal{G}_{\xi}\xi)_{\Xi},
\end{equation}
where $\mathcal{G}_{y} \in \mathcal{L}(\mathbb{M})$ and $\mathcal{G}_{\xi} \in \mathcal{L}(\Xi)$ are self-adjoint and $\mathcal{G}_{y\xi} \in \mathcal{L}(\mathbb{M};\Xi)$. Then for each $v= v_{1}+i v_{2}$ with $v_{1},v_{2} \in \mathbb{E}$ and $\xi = \xi_{1} + i\xi_{2}$ with $\xi_{1},\xi_{2} \in \Xi$ from the definition $\mathcal{F}^{\mathbb{C}}(v,\xi) = \mathcal{F}(v_{1},\xi_{1}) + \mathcal{F}(v_{2},\xi_{2})$ we have
\begin{equation}
	\label{EQ: HermitianExtensionNeutralDelayGeneral}
	\mathcal{F}^{\mathbb{C}}(v,\xi) = \mathcal{G}^{\mathbb{C}}(Cv,\xi) =  (Cv,\mathcal{G}_{y}Cv)_{\mathbb{M}^{\mathbb{C}}} + \operatorname{Re}( \xi,\mathcal{G}_{y\xi} Cv )_{\Xi^{\mathbb{C}}} + (\xi,\mathcal{G}_{y}\xi)_{\Xi^{\mathbb{C}}},
\end{equation}
where we omitted mentioning complexifications of the operators.

Now let us consider the extended control pair $(A,\widehat{B})$, where $\widehat{B} \in \mathcal{L}(\widehat{\Xi};\mathbb{H})$ and $\widehat{B}(\xi,\eta) = B\xi + \eta$ for $(\xi,\eta) \in \Xi \times \mathbb{H} =: \widehat{\Xi}$. For a given $\nu \in \mathbb{R}$, let $\mathfrak{M}^{T}_{v_{0}}(\nu)$ be the space of processes on $[0,T]$ through $v_{0}$ corresponding to $(A+\nu I,\widehat{B})$, i.e. the solution pairs $(v(\cdot), (\xi(\cdot),\eta(\cdot)))$ such that $\xi(\cdot) \in L_{2}(0,T;\Xi)$, $\eta(\cdot) \in L_{2}(0,T;\mathbb{H})$ and $v(\cdot)$ as a mild solution satisfies for $t \in [0,T]$
\begin{equation}
	\label{EQ: NeutralDelayTranslatedExtendedControlSys}
	\dot{v}(t) = (A+\nu I) v(t) + B \xi(t) + \eta(t)
\end{equation}
with $v(0) = v_{0}$. For $T=\infty$ we write simply $\mathfrak{M}_{v_{0}}(\nu)$ and additionally require that $v(\cdot) \in L_{2}(0,\infty;\mathbb{H})$ and $(\xi(\cdot),\eta(\cdot)) \in L_{2}(0,\infty;\widehat{\Xi})$.

Let $\Pi_{2} \colon \mathbb{H} \to L_{2}(-\tau,0;\mathbb{R}^{n})$ be the projector $(y,\phi) \mapsto \phi$, where $(y,\phi) \in \mathbb{R}^{n} \times L_{2}(-\tau,0;\mathbb{R}^{n}) = \mathbb{H}$. For any $(v(\cdot),(\xi(\cdot),\eta(\cdot))) \in \mathfrak{M}^{T}_{v_{0}}(\nu)$ we define $\mathcal{J}^{T}_{\mathcal{F}}$ as
\begin{equation}
	\label{EQ: NeutralInterpretationQuadraticFunctional}
	\mathcal{J}^{T}_{\mathcal{F}}(v(\cdot),\xi(\cdot),\eta(\cdot)) := \int_{0}^{T}\mathcal{G}(\mathcal{I}_{C}(\Pi_{2}v(t)), \xi(t)) dt,
\end{equation}
where the operator $\mathcal{I}_{C}$ is given by Theorem \ref{TH: AgalmanatedICgammaExistence}. From Theorem \ref{TH: DelayRegCond}, we have that \eqref{EQ: NeutralInterpretationQuadraticFunctional} is a well-defined quadratic functional and there exists a constant $C_{\mathcal{F}}>0$ independent of $T$ and such that
\begin{equation}
	\label{EQ: NeutralInterpretationEstimateFiniteT}
	\int_{0}^{T}|\mathcal{G}(\mathcal{I}_{C}(\Pi_{2}v(t)), \xi(t))|dt \leq C_{\mathcal{F}} \left( |v(0)|^{2}_{\mathbb{H}} + |v(T)|^{2}_{\mathbb{H}} + |v(\cdot)|^{2}_{L_{2}(0,T;\mathbb{H})} + | (\xi(\cdot),\eta(\cdot)) |^{2}_{L_{2}(0,T;\widehat{\Xi})}  \right).
\end{equation}
Due to \eqref{EQ: OperatorICgammaAgalmanatedComputationL1} and \eqref{EQ: NeutralStructuralCauchyRegularityOfSolutions}, $\mathcal{J}^{T}_{\mathcal{F}}(v(\cdot),\xi(\cdot),\eta(\cdot))$ agrees with $\int_{0}^{T}\mathcal{F}(Cv(t),\xi(t))dt$ for $v_{0} \in \mathcal{D}(A)$ and $(\xi(\cdot),\eta(\cdot)) \in C^{1}([0,T];\widehat{\Xi})$. Thus, the pair $(A+\nu I,\widehat{B})$ satisfies \nameref{DESC:REG2} w.r.t. $\widehat{\mathcal{F}}$, where $\widehat{\mathcal{F}}(v,\xi,\eta) := \mathcal{F}(v,\xi)$. In virtue of $\mathbb{E}_{0} = \mathbb{W} = \mathbb{H}$, we have \nameref{DESC:REG1} satisfied. Thus, the quadratic functional $\mathcal{J}_{\mathcal{F}}$ given by the limit of $\mathcal{J}^{T}_{\mathcal{F}}$ as $T \to \infty$ is well-defined, i.e. \nameref{DESC:QF} is satisfied due to Remark \ref{REM: RegConditionsImplyExtension}. However, for further investigations, it is more convenient to describe this functional by applying Theorem \ref{TH: AgalmanatedICgammaExistence} with $T=\infty$. This gives
\begin{equation}
	\label{EQ: NeutralInterpretationQuadraticFunctionalInfinity}
	\mathcal{J}_{\mathcal{F}}(v(\cdot),\xi(\cdot),\eta(\cdot)) = \int_{0}^{\infty}\mathcal{G}(\mathcal{I}_{C}(\Pi_{2}v(t)), \xi(t)) dt,
\end{equation}
where $\mathcal{I}_{C}$ is also given by Theorem \ref{TH: AgalmanatedICgammaExistence} with $T = \infty$. Since such an operator agrees with the corresponding operators from \eqref{EQ: NeutralInterpretationQuadraticFunctional} on finite-time intervals due to its definition and the commutative diagram from \eqref{EQ: CommDiagramICgammaAdorned}, we get that $\mathcal{J}_{\mathcal{F}}$ from \eqref{EQ: NeutralInterpretationQuadraticFunctionalInfinity} is indeed the limit of $\mathcal{J}^{T}_{\mathcal{F}}$ as $T \to \infty$.

Of course, the above properties lift to the complexificated problem and, in particular, the functional $\mathcal{J}_{\mathcal{F}^{\mathbb{C}}}$ corresponding to the Hermitian extension $\mathcal{F}^{\mathbb{C}}$ of $\mathcal{F}$ as in \eqref{EQ: HermitianExtensionNeutralDelayGeneral} can be defined. We are going to show that $\mathcal{J}_{\mathcal{F}^{\mathbb{C}}}$ satisfies \nameref{DESC:FT} w.r.t. the complexificated pair $(A+\nu I,\widehat{B})$. This is contained in the following lemma.
\begin{lemma}
	\label{LEM: QuadraticFormDelayQf}
	In the context of the complexificated problem associated with the extended control pair $(A+\nu I,\widehat{B})$, for any $(v(\cdot),(\xi(\cdot),\eta(\cdot))) \in \mathfrak{M}_{0}(\nu)$ we have
	\begin{equation}
		\label{EQ: NeutralFTIdentityLemma}
		\mathcal{J}_{\mathcal{F}^{\mathbb{C}}}(v(\cdot),\xi(\cdot),\eta(\cdot)) = \int_{-\infty}^{+\infty}\mathcal{F}^{\mathbb{C}}(\widehat{v}(\omega),\widehat{\xi}(\omega))d\omega,
	\end{equation}
    where $\widehat{v}(\cdot)$ and $\widehat{\xi}(\cdot)$ are the Fourier transforms of $v(\cdot)$ and $\xi(\cdot)$ respectively after extending them by zero to the negative semi-axis.
\end{lemma}
\begin{proof}
	 According to \eqref{EQ: HermitianExtensionNeutralDelayGeneral}, for $(v(\cdot),(\xi(\cdot),\eta(\cdot))) \in \mathfrak{M}_{0}(\nu)$ we have\footnote{Here, for convenience, we put $(\mathcal{I}_{C}v)(\cdot) := (\mathcal{I}_{C}\Pi_{2}v)(\cdot) = (\mathcal{I}_{C}\phi)(\cdot)$ for $v(\cdot)=(y(\cdot),\phi(\cdot))$. Moreover, as always, we omit emphasizing complexifications of the operators.}
    \begin{equation}
    	\label{EQ: DelayQuadraticFormLemmaFT}
    	\begin{split}
    	\mathcal{J}_{\mathcal{F}^{\mathbb{C}}} (v(\cdot),\xi(\cdot),\eta(\cdot)) = \int_{0}^{+\infty} \left( (\mathcal{I}_{C}v)(t), \mathcal{Q}_{y} (\mathcal{I}_{C}v)(t) \right)_{\mathbb{M}^{\mathbb{C}}}dt + 2\operatorname{Re}\int_{0}^{+\infty} \left(\xi(t), \mathcal{Q}_{y\xi} (\mathcal{I}_{C}v)(t) \right)_{{\Xi}^\mathbb{C}}dt + \\ + \int_{0}^{+\infty} ( \xi(t), \mathcal{Q}_{\xi}\xi(t) )_{{\Xi}^\mathbb{C}}dt.
    	\end{split}
    \end{equation}
    Applying the Parseval identity to each integral from the right-hand side in \eqref{EQ: DelayQuadraticFormLemmaFT}, one can see that in order to show \eqref{EQ: NeutralFTIdentityLemma} it is sufficient to prove that $\widehat{\left(\mathcal{I}_{C}v \right)}(\omega) = C\widehat{v}(\omega)$ for almost all $\omega$. Since $\widehat{v}(\omega)$ satisfies
    \begin{equation}
    	i\omega \widehat{v}(\omega) = (A +\nu I)\widehat{v}(\omega) + B\widehat{\xi}(\omega) + \widehat{\eta}(\omega) \text{ for almost all } \omega \in \mathbb{R},
    \end{equation}
    we get that $\widehat{v}(\cdot) \in L_{2,loc}(\mathbb{R};\mathcal{D}(A))$. In particular, $\Pi_{2} \widehat{v}(\cdot) \in L_{2,loc}(\mathbb{R};C([-\tau,0];\mathbb{C}^{n}))$. Note that $\Pi_{2}v(\cdot) \in \mathcal{A}^{2}_{\rho_{\nu}}(0,\infty;L_{2}(-\tau,0;\mathbb{C}^{n}))$ by Theorem \ref{TH: DelayRegCond} and the latter space is embedded into the embracing space $\mathcal{E}_{2}(0,\infty;L_{2}(-\tau,0;\mathbb{C}^{n}))$ due to Theorem \ref{TH: AgalmanatedEmbeddingLp}. Then Theorem \ref{TH: EmbracingFourierCommutesIGamma} guarantees that $\Pi_{2}\widehat{v}(\cdot) \in \mathcal{E}_{2}(\mathbb{R};L_{2}(-\tau,0;\mathbb{C}^{n}))$. Consequently, \eqref{EQ: OperatorCGammaComputationEmbracingL1loc} gives that $\widehat{\left(\mathcal{I}_{C}v \right)}(\omega) = C\widehat{v}(\omega)$ for almost all $\omega$. The proof is finished.
\end{proof}

Now let us assume that $F$ satisfies \eqref{EQ: LIPFcondition}. In this case solutions of \eqref{EQ: NonlinearDelayEqs} exist globally and continuously depend on initial conditions (see Chapter 12 in \cite{Hale1977}). Thus, for every $\phi_{0} \in C([-\tau,0];\mathbb{R}^{n})$ and $t_{0} \in \mathbb{R}$ there exists a unique continuous function $x(\cdot)=x(\cdot;t_{0},\phi_{0}) \colon [-\tau+t_{0},+\infty) \to \mathbb{R}^{n}$ such that $x(\theta) = \phi_{0}(\theta)$ for $\theta \in [-\tau+t_{0},t_{0}]$, $x(\cdot)+D_{0}x_{\cdot} \in C^{1}([t_{0},+\infty);\mathbb{R}^{n})$) and $x(t)$ satisfies \eqref{EQ: NonlinearDelayEqs} for $t \geq t_{0}$. We define a cocycle $\psi^{t}(q, \cdot) \colon \mathbb{E} \to \mathbb{E}$ over the translation flow\footnote{That is $\vartheta$ is the family of mappings $\vartheta^{t} \colon \mathbb{R} \to \mathbb{R}$ such that $\vartheta^{t}(s):=t+s$ for any $t, s \in \mathbb{R}$} $\vartheta$ on $\mathcal{Q}:=\mathbb{R}$ given by these solutions as
\begin{equation}
	\label{EQ: CocycleDelayRn}
	\psi^{t}(q,v_{0}):=v(t+q;q,v_{0}) \text{ for } v_{0} \in \mathbb{E}, t \geq 0 \text{ and } q \in \mathcal{Q},
\end{equation}
where $v(t;q,v_{0})=(x(t)+D_{0}x_{t},x_{t})$, $v_{0}=( \phi_{0}(0) + D_{0}\phi_{0} , \phi_{0})$ and $x(\cdot)=x(\cdot;q,\phi_{0})$. Clearly, $\psi^{0}(q,v)=v$ and the cocycle property $\psi^{t+s}(q,v)=\psi^{t}(\vartheta^{s}(q),\psi^{s}(q,v))$ is satisfied for all $v \in \mathbb{E}$, $q \in \mathcal{Q}$ and $t,s \geq 0$. For brevity, we denote the cocycle just by $\psi$. 

In the case both $F(t,y)$ and $\widetilde{W}(t)$ from \eqref{EQ: NonlinearDelayEqs} are $\sigma$-periodic in $t$ for some $\sigma>0$, it is convenient to take as $\mathcal{Q}$ the circle of length $\sigma$, i.e. $\mathbb{R}/\sigma \mathbb{Z}$. Then \eqref{EQ: CocycleDelayRn} also defines a cocycle over the translation semiflow on the circle. More generally, one may consider any flow (or a semiflow) $\vartheta$ on a complete metric space $\mathcal{Q}$ and generate a cocycle over $(\mathcal{Q},\vartheta)$ from the solutions of \eqref{EQ: NonlinearDelayEqs} with $F(t,Cx_{t})$ changed to $F(\vartheta^{t}(q),Cx_{t})$ and $W(t)$ changed to $W(\vartheta^{t}(q))$, where the new $F \colon \mathcal{Q} \times \mathbb{R}^{r} \to \mathbb{R}^{m}$ and $W \colon \mathcal{Q} \to \mathbb{R}^{n}$ are continuous. This is appropriate for studying almost periodic equations or for the linearization of semiflows over invariant sets. Here we omit such generalizations to avoid technicalities. See \cite{Anikushin2020Semigroups,Anikushin2021AADyn} for details.

The following lemma shows that $\psi$ is uniformly Lipschitz on compact time intervals.
\begin{lemma}
	\label{LEM: DelayLipschitzE}
	Let \eqref{EQ: LIPFcondition} be satisfied. Then for every $T>0$ there exists a constant $L_{T}>0$ such that
	\begin{equation}
		\| \psi^{t}(q,v_{1})-\psi^{t}(q,v_{2}) \|_{\mathbb{E}} \leq L_{T} \cdot \|v_{1}-v_{2}\|_{\mathbb{E}} \text{ for all } v_{1},v_{2} \in \mathbb{E}, t \geq 0 \text{ and } q \in \mathcal{Q}.
	\end{equation}
\end{lemma}
\begin{proof}
	Let $q \in \mathbb{R}$ and $\phi_{1},\phi_{2} \in C([-\tau,0];\mathbb{R}^{n})$ be fixed. For $i =1, 2$ consider the corresponding solutions $x_{i}(t)=x(t;q,\phi_{i})$ of \eqref{EQ: NonlinearDelayEqs}, where $t \geq q$, and put $v_{i}(t):=(x_{i}(t)+D_{0}x_{i,t},x_{i,t})$. We have to estimate $\|v_{1}(t)-v_{2}(t)\|_{\mathbb{E}}$, which is, by definition, the supremum norm $\| x_{1,t} - x_{2,t} \|_{\infty}$. Let $D \colon C([-\tau,0];\mathbb{R}^{n}) \to \mathbb{R}^{n}$ be the operator given by $D\phi = \phi(0) + D_{0}\phi$. Integrating \eqref{EQ: NonlinearDelayEqs} on $[q,t]$, we get that $\Delta_{t} = x_{1,t}-x_{2,t}$ satisfies
	\begin{equation}
		D\Delta_{t} = h(t), \text{ where } h(t) = D(\phi_{1}-\phi_{2}) + \int_{q}^{t}(F(s,Cx_{1,s})-F(s,Cx_{2,s}))ds.
	\end{equation}
    By Lemma 3.1, Chapter 12 in \cite{Hale1977}, there are constants $M>0$ and $\varkappa \in \mathbb{R}$ (depending only on $D$) such that
    \begin{equation}
    	\|\Delta_{t}\|_{\infty} \leq Me^{\varkappa (t-q)}\left[ \|\Delta_{q}\|_{\infty} + \sup_{q \leq s \leq t} |h(s)| \right].
    \end{equation}
    Moreover, we have the obvious estimate
    \begin{equation}
    	|h(t)| \leq \| D \| \cdot \|\Delta_{q}\|_{\infty} + \Lambda \cdot \|C\| \cdot \int_{q}^{t}\| \Delta_{s}  \|_{\infty} ds.
    \end{equation}
    Combining these inequalities with the Gronwall inequality, we obtain the conclusion of the lemma. The proof is finished.
\end{proof}

The following lemma establishes the link between classical solutions to the nonlinear problem \eqref{EQ: NonlinearDelayEqs} and the linear inhomogeneous problem corresponding to the pair $(A,B)$. We state it for differences between such solutions that is convenient for further studying. Its statement is a direct consequence of Theorem \ref{TH: DelayRegCond}.
\begin{lemma}
	\label{LEM: DelayVarOfConstants}
	For $i = 1, 2$ consider $\phi_{i} \in C([-\tau,0];\mathbb{R}^{n})$ and let $x_{i}(\cdot)=x_{i}(\cdot;t_{0},\phi_{i})$ be the corresponding solutions  to \eqref{EQ: NonlinearDelayEqs}. Put $v_{i}(t):=(x_{i}(t)+D_{0}x_{i,t},x_{i,t})$ and $w(t)=v_{1}(t)-v_{2}(t)$. Then we have for all $t \geq t_{0}$
	\begin{equation}
		\label{EQ: DelayVariationOfConstantFormula}
		w(t) = G(t-t_{0})w(t_{0}) + \int_{t_{0}}^{t}G(t-s)B\xi(s)ds,
	\end{equation}
    where $\xi(t)=F(t,Cv_{1}(t))-F(t,Cv_{2}(t))$.
\end{lemma}

We also need to study the extension of the cocycle $\psi$ onto $\mathbb{H}$. Note that our approach from \cite{Anikushin2020Semigroups}, which is based on the variation of constant formula \eqref{EQ: DelayVariationOfConstantFormula} and the estimate from Lemma \ref{LEM: DelayMeasuEstimate}, works only for a restricted class of the operators $D_{0}$ (including $D_{0} = 0$). To overcome this difficulty, we use the equivalent to \eqref{EQ: NonlinearDelayEqs} integral equation and follow arguments from Lemma 2.6 in \cite{BurnsHerdmanStech1983}, which were already used in the proof in Theorem \ref{TH: DelayRegCond}.
\begin{lemma}
	\label{LEM: ExtensionOfCocycleFromEtoHDelayNeutral}
	Under \eqref{EQ: LIPFcondition} the cocycle $\psi$ from \eqref{EQ: CocycleDelayRn} can be uniquely extended to a cocycle in $\mathbb{H}$. Moreover, for any $T>0$ there exists a constant $L_{T}>0$ such that for any $v_{1},v_{2} \in \mathbb{H}$, $q \in \mathcal{Q}$ and $t \in [0,T]$ we have
	\begin{equation}
		\label{EQ: NDEUniformLipschitzH}
			| \psi^{t}(q,v_{1}) - \psi^{t}(q,v_{2}) |_{\mathbb{H}} \leq L_{T} \cdot | v_{1} - v_{2} |_{\mathbb{H}} \text{ for } t \geq 0.
	\end{equation}
\end{lemma}
\begin{proof}
	We give a sketch of the proof following ideas from Lemma 2.6 in \cite{BurnsHerdmanStech1983}. Let us put $q = 0$ (for simplicity), fix $\varepsilon>0$, $v_{0}:=(y,\phi) \in \mathbb{H}$ and consider the operator $L_{2}(-\tau,\varepsilon;\mathbb{R}^{n}) \to L_{2}(-\tau,\varepsilon;\mathbb{R}^{n})$ given (after integrating \eqref{EQ: NonlinearDelayEqs}) as
	\begin{equation}
		(S_{v_{0}}x)(t) := \begin{cases}
			-D_{0}x_{t} + y + \int_{0}^{t}(\widetilde{A}x_{s} + \widetilde{B}F(s,Cx_{s}) + \widetilde{W}(s) )ds \text{ for } 0 < t \leq \varepsilon, \\
			\phi(t) \text{ for } -\tau \leq t \leq 0.
		\end{cases}
	\end{equation}
    Note that it is well defined due to \eqref{EQ: LIPFcondition} and Theorem \ref{TH: OperatorIcAdornedLp} that allows to interpret the actions of $D_{0}$, $\widetilde{A}$ and $C$ on $x_{t}$ for almost all $t \in [0,\varepsilon]$ via the corresponding operators. Moreover, from Theorem \ref{TH: OperatorIcAdornedLp} we have the Lipschitz estimate
    \begin{equation}
    	\label{EQ: LipschitzEstimateFirstIterationNeutralDelay}
    	\| S_{v_{0}}x_{1}(\cdot) - S_{v_{0}} x_{2}(\cdot) \|_{L_{2}(-\tau,\varepsilon;\mathbb{R}^{n})} \leq ( \operatorname{Var}_{[-\tau,0]}(\delta_{D_{0}}) + \varepsilon M ) \cdot \| x_{1}(\cdot) - x_{2}(\cdot) \|_{L_{2}(-\tau,\varepsilon;\mathbb{R}^{n})},
    \end{equation}
    where $M>0$ is an independent of $\varepsilon$ constant. Similarly to \eqref{EQ: StructuralCauchyFormulaSecondIterateEstimate} for the second iterate $U_{v_{0}}:=S^{2}_{v_{0}}$ we get
    \begin{equation}
    	\begin{split}
    		  \| U_{v_{0}}x_{1}(\cdot) - U_{v_{0}} x_{2}(\cdot) \|_{L_{2}(-\tau,\varepsilon;\mathbb{R}^{n})} \leq \\ \leq ( \operatorname{Var}_{[-\varepsilon,0]}\delta_{D_{0}} + \varepsilon M ) \cdot ( \operatorname{Var}_{[-\tau,0]}(\delta_{D_{0}}) + \varepsilon M ) \cdot \| x_{1}(\cdot) - x_{2}(\cdot) \|_{L_{2}(-\tau,\varepsilon;\mathbb{R}^{n})}.
    	\end{split}
    \end{equation}
    Since $\operatorname{Var}_{[-\varepsilon,0]}\delta_{D_{0}}$ tends to zero as $\varepsilon \to 0+$, we get that $U_{v_{0}}$ is a uniform in $v_{0} = (y,\phi)$ contraction if $\varepsilon$ is chosen sufficiently small.
    
    Since the fixed point of $U_{v_{0}}=S^{2}_{v_{0}}$ is unique, it must be also the unique fixed point of $S_{v_{0}}$. Moreover,
    \begin{equation}
    	\| S_{v_{1}}x(\cdot) - S_{v_{2}}x(\cdot) \|_{L_{2}(-\tau,\varepsilon;\mathbb{R}^{n})} \leq | v_{1} - v_{2} |_{\mathbb{H}} \text{ for any } x(\cdot) \in L_{2}(-\tau,\varepsilon;\mathbb{R}^{n}).
    \end{equation}
    From this we can get that $U_{v_{0}}$ is a family of uniform in $v_{0}$ contractions with a Lipschitz dependence on $v_{0} \in \mathbb{H}$. Then the standard triangle inequality argument shows that its fixed point also depends on $v_{0}$ in a Lipschitz manner. Since $\varepsilon$ can be chosen independently of $v_{0}$, we can iterate the process to obtain for any $T>0$ a constant $C_{T}>0$ such that for $i=1,2$, any $v_{i} \in \mathbb{H}$ and $x_{i}(\cdot) \in L_{2}(-\tau,T;\mathbb{R}^{n})$ satisfying $S_{v_{i}}x_{i}(\cdot) = x_{i}(\cdot)$ we have the estimate
    \begin{equation}
    	\| x_{1}(\cdot) - x_{2}(\cdot) \|_{L_{2}(-\tau,T;\mathbb{R}^{n})} \leq C_{T} \cdot | v_{1} - v_{2} |_{\mathbb{H}}.
    \end{equation}
    From this, by applying Theorem \ref{TH: OperatorIcAdornedLp} to estimate $y(t) = Dx_{t}$ (here $D = \delta_{0} + D_{0}$ as in the proof of Lemma \ref{LEM: DelayLipschitzE}), we obtain the conclusion. The proof is finished.
\end{proof}

\begin{remark}
\label{REM: NDEGeneralizedDifferenceSolutions}
Clearly, for the solutions in $\mathbb{H}$ we also have the variation of constants formula as in \eqref{EQ: DelayVariationOfConstantFormula} satisfied if we interpret $Cv_{1}(t)$ and $Cv_{2}(t)$ according to Theorem \ref{TH: OperatorIcAdornedLp} that is justified by Theorem \ref{TH: DelayRegCond}. Thus, under \eqref{EQ: LIPFcondition} for any pair of trajectories $v_{1}(t)=\psi^{t}(q,v_{1,0})$, $v_{2}(t)=\psi^{t}(q,v_{2,0})$, where $v_{1,0}$,$v_{2,0} \in \mathbb{H}$ and $q \in \mathbb{R}$, the difference $w(t)=v_{1}(t)-v_{2}(t)$ satisfies (in the sense of mild solutions) the linear inhomogeneous equation
\begin{equation}
	\dot{w}(t)=Aw(t) + B\xi(t)
\end{equation}
for a proper $\xi(t)$ such that $\mathcal{F}(v(t),\xi(t)) \geq 0$ for almost all $t \geq 0$ and $\mathcal{F}(v,\xi) = \Lambda^{2}|Cv|^{2}_{\mathbb{M}} - |\xi|^{2}_{\Xi}$. Such linear inhomogeneous systems with quadratic constraints given by the forms like $\mathcal{F}$ will be considered in Theorem \ref{TH: QuadraticLyapunovFuncDelay}.
\end{remark}

Suppose that for some $\nu_{0} \in \mathbb{R}$ such that $-\nu_{0} > a_{D}$, where $a_{D}$ is given by \eqref{EQ: SqueezingDExponent}, there are exactly $j$ roots of \eqref{EQ: SpectrumADelay} with $\operatorname{Re} p > -\nu_{0}$ and no roots lie on the line $-\nu_{0} + i \mathbb{R}$. Using the spectral decompositions, we obtain $\mathbb{H} = \mathbb{H}^{s}(\nu_{0}) \oplus \mathbb{H}^{u}(\nu_{0})$, where $\mathbb{H}^{u}(\nu_{0})$ is the $j$-dimensional generalized eigenspace corresponding to the roots with $\operatorname{Re} p > -\nu_{0}$ and $\mathbb{H}^{s}(\nu_{0})$ is the complementary spectral subspace. In the case $D_{0}$ has no singular part, Theorem \ref{TH: SpectralDecompositionsDelayNeutral} guarantees that for any $v_{0} \in \mathbb{H}^{s}(\nu_{0})$ we have $e^{\nu_{0} t}G(t)v_{0} \to 0$ in $\mathbb{H}$ as $t \to +\infty$ and for any $v_{0} \in \mathbb{H}^{u}(\nu_{0})$ the past $G(t)v_{0} \in \mathbb{H}^{u}(\nu_{0})$ for $t \leq 0$ is uniquely defined (since $\mathbb{H}^{u}(\nu_{0})$ is finite-dimensional) and $e^{\nu_{0} t}G(t)v_{0} \to 0$ in $\mathbb{H}$ as $t \to -\infty$. This decomposition of the semigroup $G(t)$ is essential to determine the sign properties of $P$ from Theorem \ref{TH: QuadraticLyapunovFuncDelay} below.

Now let us consider the transfer operator $W(p):=C(A-pI)^{-1}B$ of the triple $(A,B,C)$ defined for $p$ from the resolvent set of $A$ (to be more precise, one should consider the complexifications of the operators $A,B,C$). As in \eqref{EQ: AandBoperatorsDelayRepresentation} there exists an $(r \times n)$-matrix-valued function of bounded variation $c(\theta)$, $\theta \in [-\tau,0]$, representing the operator $C$ as
\begin{equation}
	C\phi = \int_{-\tau}^{0} d c(\theta) \phi(\theta) \text{ for } \phi \in C([-\tau,0];\mathbb{R}). 
\end{equation}
Putting  $\gamma(p) := \int_{-\tau}^{0} e^{p \theta} d c(\theta)$ and using \eqref{EQ: DelayResolventPropetyLemma}, one may identify the transfer operator $W(p) \colon \Xi^{\mathbb{C}} \to \mathbb{M}^{\mathbb{C}}$ with the matrix
\begin{equation}
	\label{EQ: TransferOperatorDelayMatrixRepresentation}
	W(p)=\gamma(p)(\alpha(p)-pI_{n}-p\delta(p))^{-1}\widetilde{B}.
\end{equation}

Now we can formulate a theorem on the existence of quadratic Lyapunov functionals for the cocycle $\psi$ generated by \eqref{EQ: NonlinearDelayEqs} in $\mathbb{H}$. Its proof is similar to the proof of other theorems of this type (see \cite{Anikushin2020FreqParab,Gelig1978,KuzReit2020}) and we present it here for the sake of completeness.
\begin{theorem}
	\label{TH: QuadraticLyapunovFuncDelay}
	Suppose $D_{0}$ has no singular part (see \eqref{EQ: D0HasnoSingularPart}) and let $F$ satisfy \eqref{EQ: LIPFcondition}. Suppose we are given with a quadratic form $\mathcal{F}(v,\xi):= \mathcal{G}(Cv,\xi)$ (where $\mathcal{G}$ as in \eqref{EQ: NeutralDelayQuadraticFormG}), which satisfies $\mathcal{F}(v,0) \geq 0$ for all $v \in \mathbb{E}$ and \eqref{EQ: MonotoneSectorCondtionGeneral}. Suppose also that for some  $-\nu_{0} > a_{D}$, where $a_{D}$ is given by \eqref{EQ: SqueezingDExponent}, equation \eqref{EQ: SpectrumADelay} does not have roots with $\operatorname{Re}p = -\nu_{0}$ and have exactly $j$ roots with $\operatorname{Re}p > -\nu_{0}$. Let the frequency inequality 
	\begin{equation}
		\label{EQ: Frequency-domainDelay}
		\sup_{\omega \in \mathbb{R}}\sup\limits_{\xi \in \Xi^{\mathbb{C}}} \frac{\mathcal{G}^{\mathbb{C}}(-W(- \nu_{0}+i\omega)\xi,\xi)}{|\xi|^{2}_{\Xi^{\mathbb{C}}}} < 0
	\end{equation}
	be satisfied. Then there exist $P^{*}=P \in \mathcal{L}(\mathbb{H})$ and $\delta>0$ such that for $V(v):=(v, Pv )_{\mathbb{H}}$ and the cocycle $\psi$ in $\mathbb{H}$ given by Lemma \ref{LEM: ExtensionOfCocycleFromEtoHDelayNeutral} we have
	\begin{equation}
		\label{EQ: QuadraticInequality}
		\begin{split}
			e^{2\nu_{0} r}V(\psi^{r}(q,v_{1})-\psi^{r}(q,v_{2}))-e^{2\nu_{0} l}V(\psi^{l}(q,v_{1})-\psi^{l}(q,v_{2})) \leq\\\leq -\delta \int_{l}^{r}e^{2\nu_{0} s}|\psi^{s}(q,v_{1})-\psi^{s}(q,v_{2})|^{2}_{\mathbb{H}}ds
		\end{split}
	\end{equation}
	satisfied for all $q \in \mathbb{R}$, $0 \leq l \leq r$ and all $v_{1},v_{2} \in \mathbb{H}$. Moreover, $P$ is positive on $\mathbb{H}^{s}(\nu_{0})$ and negative on the $j$-dimensional subspace $\mathbb{H}^{u}(\nu_{0})$.
\end{theorem}
\begin{proof}
	Due to the cocycle property, it is sufficient to show \eqref{EQ: QuadraticInequality} for $l=0$. Let us justify that Theorem \ref{TH: FrequencyTheoremRealOperator} can be applied to the pair $(A+\nu_{0} I,B)$ corresponding to \eqref{EQ: NonlinearDelayEqs} and the form $\mathcal{F}$ given in the statement. Note that in our case $\mathbb{E}_{0}=\mathbb{W}=\mathbb{H}$. By item 3) from Theorem \ref{TH: SpectralDecompositionsDelayNeutral}, we have \nameref{DESC:RES} satisfied. From the arguments preceding \eqref{EQ: NeutralInterpretationQuadraticFunctionalInfinity} and Lemma \ref{LEM: QuadraticFormDelayQf} we get \nameref{DESC:QF} and \nameref{DESC:FT} satisfied for the extended control system $(A+\nu_{0} I,\widehat{B})$ and the extended quadratic form. Note that inequality \eqref{EQ: Frequency-domainDelay} in our context is exactly the same as \eqref{EQ: FrequencyConditionNegative}. Thus, Theorem \ref{TH: FrequencyTheoremRealOperator} gives a self-adjoint operator $P \in \mathcal{L}(\mathbb{H})$ and a number $\delta>0$ such that for the quadratic form $V(v):= (v, Pv)_{\mathbb{H}}$ we have
	\begin{equation}
		\label{EQ: MonotoneInequality1}
		V(v(T))-V(v(0)) + \int_{0}^{T}\mathcal{F}(v(s),\xi(s))ds \leq -\delta \int_{0}^{T}(|v(s)|^{2}_{\mathbb{H}} + |\xi(s)|^{2}_{\Xi})ds
	\end{equation}
	for all $\xi(\cdot) \in L_{2}(0,T;\Xi)$ and $v(\cdot)$ being a mild solution on $[0,T]$ to 
	\begin{equation}
		\dot{v}(t) = (A + \nu_{0} I)v(t) + B\xi(t)
	\end{equation}
	with $v(0)=v_{0} \in \mathbb{H}$.
	
	Putting $\xi(\cdot) \equiv 0$ in \eqref{EQ: MonotoneInequality1} and using the property $\mathcal{F}(v,0) \geq 0$, we get
	\begin{equation}
		\label{EQ: MonotoneInequality2}
		V(v(T)) - V(v(0)) \leq -\delta \int_{0}^{T}|v(s)|^{2}_{\mathbb{H}}ds
	\end{equation}
    that is the Lyapunov inequality for the $C_{0}$-semigroup $e^{\nu_{0} t}G(t)$ generated by the operator $A+\nu_{0} I$. Note that this semigroup admits an exponential dichotomy in $\mathbb{H}$ with $\mathbb{H}^{s}(\nu_{0})$ and $\mathbb{H}^{u}(\nu_{0})$ being the stable and unstable subspaces respectively according to Theorem \ref{TH: SpectralDecompositionsDelayNeutral}. Thus, Theorem \ref{TH: HighRankConesTheorem} gives us the desired sign properties of $P$.
	
	In terms of \eqref{EQ: MonotoneInequality1} we have $v(t)=e^{\nu_{0} t}\widetilde{v}(t)$ and $\xi(t) = e^{\nu_{0} t} \widetilde{\xi}(t)$, where $\widetilde{v}(t)=\widetilde{v}(t;v_{0},\widetilde{\xi})$ is a solution to \eqref{EQ: ControlSystem} corresponding to the pair $(A,B)$. Thus, \eqref{EQ: MonotoneInequality1} can be described as	
	\begin{equation}
		\label{EQ: MonotoneIneq1New}
		e^{2\nu_{0} T}V(v(T))-V(v(0)) + \int_{0}^{T}e^{2\nu_{0} s}\mathcal{F}(v(s),\xi(s))ds \leq -\delta \int_{0}^{T}e^{2\nu_{0} s}(|v(s)|^{2}_{\mathbb{H}} + |\xi(s)|^{2}_{\Xi})ds
	\end{equation}
	for all solution pairs $\xi(\cdot) \in L_{2}(0,T;\Xi)$ and $v(t)=v(t;v_{0},\xi)$ corresponding to $(A,B)$.
	
	Let us consider \eqref{EQ: MonotoneIneq1New} with $v(t)=\psi^{t}(q,v_{1})-\psi^{t}(q,v_{2})$ and $\xi(t) = F(t+q,C\psi^{t}(q,v_{1}))-F(t+q,C\psi^{t}(q,v_{2}))$ that is justified by Lemma \ref{LEM: DelayVarOfConstants} and Remark \ref{REM: NDEGeneralizedDifferenceSolutions}. Note that we have $\mathcal{F}(v(s),\xi(s)) \geq 0$ for almost all $s \geq 0$ and, consequently, the integral from the left-hand side in \eqref{EQ: MonotoneIneq1New} can be eliminated. This clearly gives \eqref{EQ: QuadraticInequality} and finishes the proof. 
\end{proof}

\begin{remark}
	\label{REM: NDEoptimalityFreqIneq}
	It can be shown that frequency inequalities as in \eqref{EQ: Frequency-domainDelay} are optimal in the following sense. They provide necessary and sufficient conditions for the existence of a common quadratic functional like $V(\cdot)$ such that an analog of \eqref{EQ: QuadraticInequality} is satisfied for all solutions corresponding to the pair $(A,B)$ and satisfying the quadratic constraint given by $\mathcal{F}$. The ``sufficiency'' follows from the arguments used in Theorem \ref{TH: QuadraticLyapunovFuncDelay}. The ``necessity'' is a purely algebraic fact known as a theorem on the losslessness of the S-procedure, which was discovered by V.A.~Yakubovich \cite{Gelig1978}. For particular classes of the forms like $\mathcal{F}$, the class of systems can be further refined. For example, for the forms obtained through the Lipschitz condition \eqref{EQ: LIPFcondition}, i.e. $\mathcal{F}(v,\xi) = \Lambda^{2}|Cv|^{2}_{\mathbb{M}} - |\xi|^{2}_{\Xi}$, it is sufficient to consider only the class of linear problems. Note that if the frequency inequality is violated, the dynamics may significantly depend on the choice of the nonlinearity $F$ from the considered class. We refer to our work \cite{Anikushin2020FreqParab} for details and discussions (see Section 5.2 therein).
\end{remark}

In our work \cite{Anikushin2020Geom}, besides \eqref{EQ: QuadraticInequality} and the uniform Lipschitz property \eqref{EQ: NDEUniformLipschitzH}, for the existence of inertial manifolds it is important to estimate the exponent of decay for the Kuratowski measure of noncompactness. For this we use the exponent $a_{D}$ from \eqref{EQ: SqueezingDExponent}. It can be shown analogously to \eqref{EQ: DecompositionGFrictionCompactLinearNeutralDelay} that the difference of $\psi^{t}(q,\cdot) - G_{0}(t)\Pi_{2}$ is a finite-rank operator with the range contained in $\mathbb{R}^{n} \times \{0\}$ and it takes bounded subsets in $\mathbb{H}$ to uniformly in ($q \in \mathcal{Q}$) bounded subsets. In particular, since $a_{D}$ determines the growth bound of $G_{0}(t)\Pi_{2}$, this implies that for every $\varepsilon > 0$ there exists $M_{\varepsilon}>0$ such that the Kuratowski measure $\alpha_{K}$ in $\mathbb{H}$ admits the exponential estimate
\begin{equation}
	\label{EQ: AlphaMeasureSqueezingNeutral}
	\alpha_{K}(\psi^{t}(\mathcal{Q}, \mathcal{B})) \leq M_{\varepsilon} \cdot e^{(a_{D}+\varepsilon)t} \cdot \operatorname{diam}\mathcal{B} \text{ for all } t \geq 0,
\end{equation}
for any bounded subset $\mathcal{B} \subset \mathbb{H}$. In the case $a_{D} < 0$ from \eqref{EQ: AlphaMeasureSqueezingNeutral} we have the exponential decay for a sufficiently small $\varepsilon>0$. Clearly, the decay exponent $-(a_{D}+\varepsilon)$ can always be chosen greater that $\nu_{0}$ from Theorem \ref{TH: QuadraticLyapunovFuncDelay}. This is the condition we require in \cite{Anikushin2020Geom} to construct inertial manifolds.
\begin{remark}
	\label{REM: AfterDelayQFTh}
	Under the hypotheses of Theorem \ref{TH: QuadraticLyapunovFuncDelay} with $\nu_{0} > 0$ we have\footnote{If there exists at least one bounded complete trajectory.} a family of $j$-dimensional submanifolds in $\mathbb{H}$, which forms an inertial manifold for the cocycle $\psi$ (see \cite{Anikushin2020Geom}). These manifolds are Lipschitz and attract exponentially fast (with the exponent $\nu_{0}$) all trajectories of $\psi$ by trajectories lying on the manifold (this is the so-called exponential tracking property). Moreover, they are also $C^{1}$-differentiable and uniformly normally hyperbolic\footnote{For this it is required to satisfy \eqref{EQ: Frequency-domainDelay} for two distinct values of $\nu_{0}$ separating the same $j$ roots. Clearly, if \eqref{EQ: Frequency-domainDelay} is satisfied for some $\nu_{0}$, it is also satisfied for all sufficiently close to $\nu_{0}$ values due to the first resolvent identity combined with \eqref{EQ: DelayNeutralResolventEstimateEfromH} and \eqref{EQ: DelayNeutralResolventBoundedness}.} provided that $F(t,y)$ is $C^{1}$-differentiable in $y$ (see \cite{Anikushin2020Geom,Anikushin2020Semigroups}).
\end{remark}
\begin{remark}
	In \cite{Anikushin2020Semigroups} it is shown that in the case $D_{0} \equiv 0$ we have $\psi^{\tau}(q,\mathbb{H}) \subset \mathbb{E}$ and for all $T \geq \tau$ there exists a constant $L_{T} > 0$ such that the smoothing estimate
	\begin{equation}
		\label{EQ: DelaySmoothingEstimate}
		\| \psi^{t}(q,v_{1})-\psi^{t}(q,v_{2}) \|_{\mathbb{E}} \leq L_{T} \cdot |v_{1}-v_{2}|_{\mathbb{H}}
	\end{equation}
    is valid for all $t \in [\tau,T]$, $q \in \mathcal{Q}$ and $v_{1},v_{2} \in \mathbb{H}$. In the case of general neutral equations, there are no such estimates. However, one can show that \eqref{EQ: DelaySmoothingEstimate} is satisfied provided that $D_{0}$ can be extended to a bounded operator from $L_{2}(-\tau,0;\mathbb{R}^{n})$ to $\mathbb{R}^{n}$.
\end{remark}
\begin{remark}
We may use \eqref{EQ: DelaySmoothingEstimate} to construct inertial manifolds in $\mathbb{E}$. However, in the general situation we do not know how to show that such manifolds belong to $\mathbb{E}$. Note also that S.~Chen and J.~Shen \cite{ChenShen2021} constructed inertial manifolds in $\mathbb{E}$ (see Subsection \ref{SUBSEC: EquationsWithSmallDelays} below).
\end{remark}

Now let us discuss several forms of the frequency inequality from \eqref{EQ: Frequency-domainDelay}.

Let us assume that $F$ satisfies \eqref{EQ: LIPFcondition} and consider the form $\mathcal{F}(v,\xi):= \Lambda^{2} |Cv|^{2}_{\mathbb{M}} - |\xi|^{2}_{\Xi}$. In this case the frequency inequality \eqref{EQ: Frequency-domainDelay} is equivalent to the inequality
\begin{equation}
	\label{EQ: DelaySmithCondition}
	\sup_{\omega \in \mathbb{R}}\|W(-\nu_{0} + i \omega )\|_{\Xi^{\mathbb{C}} \to \mathbb{M}^{\mathbb{C}}} < \Lambda^{-1},
\end{equation}
which was used by R.A.~Smith \cite{Smith1992,Smith1987OrbStab} (see also M.~Miklav\v{c}i\v{c} \cite{Miclavcic1991}). In particular, our thery allows to unify his works on convergence theorems for periodic equations with delay and developments of the Poincar\'{e}-Bendixson theory (see \cite{Anikushin2020Geom}). We will show in Subsection \ref{SUBSEC: EquationsWithSmallDelays} that certain results concerned with the existence of $n$-dimensional inertial manifolds for (neutral) delay equations with small delays can be obtained through the Smith frequency condition \eqref{EQ: DelaySmithCondition} and our theory from \cite{Anikushin2020Geom}.

For the case described in \eqref{EQ: SectorCondMonotone} consider the quadratic form $\mathcal{F}(v,\xi):= (\xi - \varkappa_{1} Cv)(\varkappa_{2} Cv - \xi)$. Then the frequency inequality from \eqref{EQ: Frequency-domainDelay} is equivalent to
\begin{equation}
	\label{EQ: CircleCriteriaFreqCond}
	\sup_{\omega \in \mathbb{R}}\operatorname{Re}\left[ (1+\varkappa_{1} W(-\nu_{0}+i\omega))^{*}(1+\varkappa_{2} W(-\nu_{0}+i\omega) ) \right] > 0.
\end{equation}
Let us suppose that $\varkappa_{1} \leq 0 < \varkappa_{2}$ to make $\mathcal{F}$ satisfy $\mathcal{F}(v,0) \geq 0$. In control theory, \eqref{EQ: CircleCriteriaFreqCond} is known as the Circle Criterion \cite{Gelig1978}. Its geometric meaning is in that the curve $\omega \mapsto W(-\nu_{0} + i\omega)$ on the complex plane must lie inside the circle $\operatorname{Re}[(1+\varkappa_{1}z)^{*}(1+\varkappa_{2}z)]=0$. In \cite{Anikushin2020Semigroups,Anikushin2021SS} we used this criterion (with $\varkappa_{1} = 0$) to obtain the existence of two-dimensional inertial manifolds in the Suarez-Schopf model. 

We can generalize \eqref{EQ: CircleCriteriaFreqCond} as follows. Assuming that $F(t,y)$ is $C^{1}$-differentiable in $y$, \eqref{EQ: CircleCriteriaFreqCond} reflects that $\Xi = \mathbb{M}$ and $F'_{y}(t,y) := \frac{d}{dy}F(t,y)$ is a self-adjoint operator in $\mathcal{L}(\mathbb{M})$ such that 
\begin{equation}
	\label{EQ: NeutralDelayFreqIneqSelfAdjointCondition}
	\varkappa_{1} (y,y)_{\mathbb{M}} \leq (y, F'_{y}(t,y_{0}) y)_{\mathbb{M}} \leq \varkappa_{2} (y,y)_{\mathbb{M}} \text{ for all } y,y_{0} \in \mathbb{M} \text{ and }  t \geq 0. 
\end{equation}
Then the choice $\mathcal{G}(y,\xi) := -\varkappa_{1}\varkappa_{2} (y,y)_{\mathbb{M}} + (\varkappa_{1} + \varkappa_{2})(y,\xi)_{\mathbb{M}} - (\xi,\xi)_{\mathbb{M}}$ leads to that $\mathcal{F}(v,\xi) := \mathcal{G}(Cv,\xi)$ will satisfy \eqref{EQ: MonotoneSectorCondtionGeneral}. For this, it is sufficient to show that $\mathcal{G}(y,\xi) \geq 0$ for any $y \in \mathbb{M}$ and $\xi = S y$, where $S$ is any self-adjoint operator in $\mathcal{L}(\mathbb{M})$ satisfying \eqref{EQ: NeutralDelayFreqIneqSelfAdjointCondition} with $F'_{y}(t,y_{0})$ exchanged with $S$. This can be done by using the orthonormal basis of eigenvectors of $S$. In more general situations, the Spectral Theorem for bounded self-adjoint operators can be used (see Section 5.5 in \cite{Anikushin2023Comp}). Then $\mathcal{G}(Cv_{1} - Cv_{2},F(t,Cv_{1})-F(t,Cv_{2})) \geq 0$ is equivalent to $\mathcal{G}(y,Sy) \geq 0$ with $y = Cv_{1} - Cv_{2}$ and $S = \int_{0}^{1}F'_{y}(t,(1-s)Cv_{1} + s Cv_{2})ds$. Thus, the corresponding quadratic constraint takes into account the self-adjointness of $F'_{y}$. For the corresponding frequency inequality \eqref{EQ: Frequency-domainDelay} we have to satisfy for some $\delta>0$ and any $\xi \in \Xi^{\mathbb{C}} = \mathbb{M}^{\mathbb{C}}$ the inequality
\begin{equation}
	\varkappa_{1} \varkappa_{2} | W(-\nu_{0} + i \omega) \xi |^{2}_{\Xi^{\mathbb{C}}} + (\varkappa_{1} + \varkappa_{2})\operatorname{Re}( W(-\nu_{0} + i \omega) \xi, \xi)_{\Xi^{\mathbb{C}}} + |\xi|^{2}_{\Xi^{\mathbb{C}}} \geq \delta |\xi|^{2}_{\Xi^{\mathbb{C}}} \text{ for all } \omega \in \mathbb{R}.
\end{equation}
\subsection{Equations with small delays}
\label{SUBSEC: EquationsWithSmallDelays}
Let us consider \eqref{EQ: NonlinearDelayEqs} with $\widetilde{A} \equiv 0$ and $\widetilde{B} = I_{n}$ being the identity $(n \times n)$-matrix, i.e. the delay equation
\begin{equation}
	\label{EQ: ConstantDelayExample}
	\frac{d}{dt}\left( x(t) + D_{0}x_{t} \right)=F(t,Cx_{t}) + \widetilde{W}(t),
\end{equation}
where $F$ satisfies \eqref{EQ: LIPFcondition} for some $\Lambda>0$. Let us assume that $\Xi = \mathbb{R}^{n}$ and $\mathbb{M}=\mathbb{R}^{r}$ are endowed with the Euclidean inner products. We assume that $C \colon C([-\tau,0]; \mathbb{R}^{n}) \to \mathbb{R}^{r}$ has the form
\begin{equation}
	\label{EQ: SmallDelaysOpC}
	C\phi = \left( \int_{-\tau}^{0}\phi_{j_{1}}(\theta)d\mu_{1}(\theta),\ldots,  \int_{-\tau}^{0}\phi_{j_{r}}(\theta)d\mu_{r}(\theta) \right),
\end{equation}
where $\mu_{k}$ are Borel signed measures on $[-\tau,0]$ having total variation $\leq 1$ and $1 \leq j_{1} \leq j_{2} \leq \ldots \leq j_{r} \leq n$ are integers. In particular, taking $\mu_{k} = \delta_{-\tau_{k}}$ with some $\tau_{k} \in [0,\tau]$, we may get discrete delays since $\int_{-\tau}^{0}\phi_{j_{k}}(\theta)d\mu_{k}(\theta) = \phi_{j_{k}}(-\tau_{k})$.

Consideration of \eqref{EQ: ConstantDelayExample} in the context of Theorem \ref{TH: QuadraticLyapunovFuncDelay} gives the following.
\begin{theorem}
	\label{TH: SmallDelayNdimTheorem}
	In the context of \eqref{EQ: ConstantDelayExample} suppose that $D_{0}$ has no singular part (see \eqref{EQ: D0HasnoSingularPart}) and \eqref{EQ: SmallDelaysOpC} is satisfied. Let the Lipschitz constant $\Lambda$ from \eqref{EQ: LIPFcondition} for $F$ be calculated w.r.t. the Euclidean norms. Suppose that for some $\nu_{0} > 0$ we have
	\begin{equation}
		\label{EQ: SmallDelaysSqueezingKappa}
		\varkappa(\nu_{0}) := e^{ \tau \nu_{0}} \cdot \|D_{0}\| < 1.
	\end{equation}
	Let the inequality
	\begin{equation}
		\label{EQ: SmallDelaysThMainIneq}
		\sqrt{r} \cdot \frac{e^{\tau \nu_{0}}}{\nu_{0}} \cdot \frac{1}{1 - \varkappa(\nu_{0})} < \Lambda^{-1}.
	\end{equation}
	be satisfied. Then the conclusion of Theorem \ref{TH: QuadraticLyapunovFuncDelay} holds for \eqref{EQ: ConstantDelayExample} with the given $\nu_{0}$ and $j=n$.
\end{theorem}
\begin{proof}
	For brevity, by $|\cdot|$ we denote the Euclidean norm and the associated matrix norm. Let us verify the frequency inequality from \eqref{EQ: DelaySmithCondition} for $p= -\nu_{0} + i\omega$. Since $\widetilde{A} \equiv 0$ and $\widetilde{B} = I_{n}$, in terms of \eqref{EQ: TransferOperatorDelayMatrixRepresentation} we have $W(p) = \gamma(p) (-pI - p\delta(p))^{-1} = -p^{-1} \cdot \gamma(p) \cdot (I+\delta(p))^{-1}$. Note that $|\delta(p)| \leq e^{\tau \nu_{0}} \|D_{0}\| = \varkappa(\nu_{0}) < 1$ and, consequently, we obtain
	\begin{equation}
		| I + \delta(p) |^{-1} \leq (1-\varkappa(\nu_{0}))^{-1}.
	\end{equation}
    Moreover, in virtue of \eqref{EQ: SmallDelaysOpC}, for any $k=1,\ldots, r$ there is a unique nonzero element in the $k$-th row of the $(r \times n)$-matrix $\gamma(p)$ and it is given by $\int_{-\tau}^{0}e^{p\theta}d\mu_{k}(\theta)$. Thus,
    \begin{equation}
    	|\gamma(p)| \leq \sqrt{r} \cdot \max_{1 \leq k \leq r} \left|\int_{-\tau}^{0}e^{p\theta}d\mu_{k}(\theta) \right| \leq \sqrt{r} e^{\tau \nu_{0}}.
    \end{equation}
    This implies that for any $\omega \in \mathbb{R}$ we have
    \begin{equation}
    	|W(p)| \leq \sqrt{r} \cdot \frac{e^{\tau \nu_{0}}}{\nu_{0}} \cdot \frac{1}{1 - \varkappa(\nu_{0})}
    \end{equation}
    Consequently, the frequency inequality is satisfied under \eqref{EQ: SmallDelaysThMainIneq}.
    
    In Remark \ref{REM: ResolventEstimateNeutralDelay} we justified that $-\nu_{0} > a_{D}$. Since in our case (i.e. $\widetilde{A} = 0$) the eigenvalues of the linear part are given by $p \in \mathbb{C}$ such that $p^{n} \cdot \operatorname{det}(I_{n}+\delta(p)) = 0$, there is the only eigenvalue with $\operatorname{Re p} \geq -\nu_{0}$ given by the zero with multiplicity $n$. Thus, the proof is finished.
\end{proof}

Note that \eqref{EQ: SmallDelaysThMainIneq} is satisfied if the delay value $\tau$ is taken sufficiently small and $\nu_{0}$ is taken sufficiently large such that \eqref{EQ: SmallDelaysSqueezingKappa} is preserved\footnote{Strictly speaking, the operator $D_{0}$ also depends on $\tau$ since it acts in the space $C([-\tau,0];\mathbb{R}^{n})$. In applications, we usually have $M:=\sup_{\tau}\|D_{0}(\tau)\| < 1$. For example, $D_{0}\phi = \alpha \phi(-\tau)$ with $|\alpha| < 1$. This allows to vary $\tau \to 0$ and $\nu_{0} \to +\infty$ such that $\tau \nu_{0} < -\ln M$.}. We are going to discuss several papers on inertial manifolds for delay equations with small delays. Speaking in terms of delay equations in $\mathbb{R}^{n}$, their results guarantee the existence of $n$-dimensional inertial manifolds. Such results can be understood in the following sense. Suppose we have some ODE in $\mathbb{R}^{n}$ with a globally Lipschitz vector field. If we let some variables to act with a delay (that is natural for models of electrical circuits), then the resulting dynamical system becomes infinite-dimensional. However, if the delay value is small enough (and does not significantly change the Lipschitz constant), the limiting dynamics turns out to be still $n$-dimensional, as for the original ODE, although there may appear new phenomena caused by the presence of delays.

In particular, when $D_{0} \equiv 0$ and $\nu_{0} = \tau^{-1}$ the inequality in \eqref{EQ: SmallDelaysThMainIneq} takes the form
\begin{equation}
	\label{EQ: SmallDelaysDzeroIneq}
	\sqrt{r} \cdot e \cdot \tau < \Lambda^{-1}.
\end{equation} 

In \cite{Chicone2003} C.~Chicone extended the works of Yu.A.~Ryabov \cite{Ryabov1967} and R.D.~Driver \cite{Driver1968SmallDelays} on the existence of $n$-dimensional inertial manifolds for delay equations in $\mathbb{R}^{n}$ with small delays. Let us compare their results with Theorem \ref{TH: SmallDelayNdimTheorem}. First of all, the class of equations in \cite{Chicone2003} is greater than \eqref{EQ: ConstantDelayExample}, but our class covers many equations arising in practice. Secondly, in \cite{Chicone2003} it is used a Lipschitz constant $K$ of the nonlinearity as a mapping from $C([-\tau,0];\mathbb{R}^{n}) \to \mathbb{R}^{n}$. Thus, in our case $K = \sqrt{r} \Lambda$ since for $\phi_{1}$,$\phi_{2} \in C([-\tau,0];\mathbb{R}^{n})$ we have
\begin{equation}
	|F(t,C\phi_{1})-F(t,C\phi_{2})|_{\mathbb{R}^{n}} \leq \Lambda \cdot \|C\| \cdot \| \phi_{1}-\phi_{2}\|_{\infty} \leq \Lambda \cdot \sqrt{r} \cdot \| \phi_{1}-\phi_{2}\|_{\infty}
\end{equation}
and one can see that the used estimates are sharp for general $C$.

A theorem of Ryabov and Driver (see item 1) of Theorem 2.1 from \cite{Chicone2003}; see also \cite{Ryabov1967,Driver1968SmallDelays}) guarantees the existence of inertial manifolds if $K\tau e < 1$ or, equivalently, $\sqrt{r} e \tau < \Lambda^{-1}$. This is exactly our condition \eqref{EQ: SmallDelaysDzeroIneq}. A slight strengthening of the inequality $\sqrt{r} e \tau < \Lambda^{-1}$, which as $\tau \to 0$ remains asymptotically the same, guarantees the exponential tracking (see item 2) of Theorem 2.1 from \cite{Chicone2003}), but do not provide any estimate for the exponent. Moreover, Theorem 2.2 from \cite{Chicone2003} guarantees the $C^{1}$-differentiability of the inertial manifold if $F(t,y)$ is $C^{1}$-differentiable in $y$ and the rougher condition $2 \sqrt{r} \sqrt{e} \tau < \Lambda^{-1}$ is satisfied. Our geometric approach from \cite{Anikushin2020Geom} shows that, in fact, these properties (at least in the cases covered by \eqref{EQ: ConstantDelayExample}) are satisfied without further strengthening of \eqref{EQ: SmallDelaysDzeroIneq}. This solves the question posed by R.D. Driver in \cite{Driver1968SmallDelays} (see the remark on p.~338 therein) in the considered class of equations. Moreover, our approach gives the exact description for the exponent of attraction as $\nu_{0} = \tau^{-1}$ and allows to construct stable foliations along the manifold.

Analogous results for neutral delay equations in $\mathbb{R}^{n}$ were obtained by S.~Chen and J.~Shen in \cite{ChenShen2021}. They gave noneffective (and, consequently, nonoptimal) conditions for the existence of $n$-dimensional inertial manifolds and showed their $C^{k}$-differentiability provided that $F(t,y)$ is $C^{k}$-differentiable in $y$. In \cite{ChenShen2021} the class of nonlinearities $F$ is also more general and $D_{0}$ may be time-dependent with $\sup_{t \in \mathbb{R}}\|D_{0}(t)\| < 1$. However, in our case we get the effective estimate \eqref{EQ: SmallDelaysThMainIneq} derived from the optimal (in the sense discussed in Remark \ref{REM: NDEoptimalityFreqIneq}) frequency inequality.

In order to obtain $C^{k}$-differentiability for $k \geq 2$ we need to satisfy \eqref{EQ: Frequency-domainDelay} for two positive numbers $\nu_{0} = \nu_{1}$ and $\nu_{0} = \nu_{2}$ with the same $j$ and such that $\nu_{2} / \nu_{1} > k$ (see \cite{Anikushin2020Semigroups} for a discussion). This is not the case in general and there indeed may exist inertial manifolds, which are only $C^{1}$-differentiable. For equations with small delays such large spectral gaps exist if $\tau$ is sufficiently small.

Note also that the results of \cite{Chicone2004,Chicone2003} and \cite{ChenShen2021} allow to obtain a decomposition of the vector field on the inertial manifold considering $\tau$ as a small parameter. This decomposition can be used to obtain effective dimension estimates for such equations (see \cite{Anikushin2020Semigroups}).
\section{Discussion}
\label{SEC: Discussion}
From the point of view given by our geometric theory \cite{Anikushin2020Geom}, where quadratic Lyapunov functionals are used, the Frequency Theorem presents an analytical tool for the construction of such functionals in particular problems. This approach allows us to separate the geometric and analytical parts of the construction of inertial manifolds. As a result we obtain optimal conditions for the existence of inertial manifolds and unify many scattered results in the field within a common context.

Results of Subsection \ref{SUBSEC: NDENonlinear} can be generalized to allow the presence of infinite number of delays. For such systems, the measurement space $\mathbb{M}$ may be a weighted $l_{2}$-space. As we have mentioned in Introduction, extensions of these results for partial differential equations with delay are also possible.

A far going generalization of the results from Section \ref{SEC: LemmaOnEstimatesForLF} is presented in our work \cite{Anikushin2023Comp} concerned with compound cocycles generated by delay equations. It allows to apply Theorem \ref{TH: FrequencyTheoremRealOperator} to obtain frequency conditions for the uniform exponential stability or existence of gaps in the Sacker-Sell spectrum for such cocycles. As a consequence, this provides a general method to resolve the problem of obtaining effective dimension estimates for global attractors in nonlinear problems (see a discussion in \cite{Anikushin2020Semigroups}) and shows the high computational complexity of the problem.

In applications, there are several ways to write down a given system in the form \eqref{EQ: NonlinearDelayEqs}. This gives some flexibility and may significantly improve the final results (see the work of R.A.~Smith \cite{Smith1992} for a nice example). Moreover, in applied models there are many situations where inertial manifolds are expected to exist, but direct applications of the Frequency Theorem do not give the desired result. This is usually caused by the roughness of the method which does not take into account some specificity of concrete problems. In this regard, it is worth noting that the most effective way to construct inertial manifolds is provided by combining the rough analytical approach and a proper change of variables. We refer to the papers of A.~Kostianko and S.~Zelik \cite{KostiankoZelikKwak,KostiankoZelik20181dII,KostiankoZelik20171dI} for such applications and discussions in the context of semilinear parabolic equations. However, for delay equations, using the change of variables seems even more complicated.

A perspective for further research is given by the question posed in \cite{Anikushin2020FreqParab} that asks to establish a connection between the Spatial Averaging Principle \cite{KostiankoZelikSA2020} and the Frequency Theorem for nonstationary problems. We plan to develop this topic in forthcoming papers.
\section*{Acknowledgements}
The author is grateful to the anonymous referee for careful reading and valuable suggestions, which particularly revealed technical flaws and led to significant improvements in the overall exposition.

\section*{Funding}
The reported study was funded by the Russian Science Foundation (Project 22-11-00172).

\section*{Conflicts of interest}
The author has no conflicts of interest to declare that are relevant to the content of this article.

\section*{Data availability statement}
Data sharing not applicable to this article as no datasets were generated or analyzed during the current study.
%\input{Funding}

%\begin{acknowledgements}

%\end{acknowledgements}

% Authors must disclose all relationships or interests that 
% could have direct or potential influence or impart bias on 
% the work: 
%
% \section*{Conflict of interest}
%
% The authors declare that they have no conflict of interest.

% BibTeX users please use one of
%\bibliographystyle{spbasic}      % basic style, author-year citations
%\bibliographystyle{spmpsci}      % mathematics and physical sciences
%\bibliographystyle{spphys}       % APS-like style for physics
%\bibliography{}   % name your BibTeX data base

\begin{thebibliography}{10}

\bibitem{Adams1976SobolevSpaces}
Adams R.A. Sobolev Spaces. Academic Press (1975)

\bibitem{Anikushin2023Comp}
Anikushin~M.M. Spectral comparison of compound cocycles generated by delay equations in Hilbert spaces. \textit{arXiv preprint}, arXiv:2302.02537 (2023)

\bibitem{Anikushin2021SS}
Anikushin~M.M., Romanov~A.O. Hidden and unstable periodic orbits as a result of homoclinic bifurcations in the {S}uarez-{S}chopf delayed oscillator and the irregularity of {ENSO}. \textit{Phys. D: Nonlinear Phenom.}, \textbf{445}, 133653 (2023)

\bibitem{Anikushin2020Semigroups}
Anikushin~M.M. Nonlinear semigroups for delay equations in Hilbert spaces, inertial manifolds and dimension estimates, \textit{Differ. Uravn. Protsessy Upravl.}, 4, (2022)

\bibitem{Anikushin2020Geom}
Anikushin~M.M. Inertial manifolds and foliations for asymptotically compact cocycles in Banach spaces. \textit{arXiv preprint}, arXiv:2012.03821v3 (2022)

\bibitem{Anikushin2019+OnCom} 
Anikushin~M.M. On the compactness of solutions to certain operator inequalities arising from the Likhtarnikov-Yakubovich frequency theorem. \textit{Vestnik St. Petersb. Univ. Math.}, \textbf{54}(4), 301--310 (2021)

\bibitem{Anikushin2020FreqParab}
Anikushin~M.M. Frequency theorem for parabolic equations and its relation to inertial manifolds theory. \textit{J. Math. Anal. Appl.}, \textbf{505}(1), 125454 (2021)

\bibitem{Anikushin2021AADyn}
Anikushin~M.M. Almost automorphic dynamics in almost periodic cocycles with one-dimensional inertial manifold. \textit{Differ. Uravn. Protsessy Upr.}, 2, 13--48 (2021), in Russian

\bibitem{Anikushin2020Red}
Anikushin~M.M. A non-local reduction principle for cocycles in Hilbert spaces. \textit{J. Differ. Equations}, \textbf{269}(9), 6699--6731 (2020)

\bibitem{Anikushin2019Vestnik}
Anikushin~M.M. On the Liouville phenomenon in estimates of fractal dimensions of forced quasi-periodic oscillations. \textit{Vestnik St. Petersb. Univ. Math.}. \textbf{52}(3), 234--243 (2019)

\bibitem{AnikushinRR2019}
Anikushin~M.M., Reitmann~V., Romanov~A.O. Analytical and numerical estimates of the fractal dimension
of forced quasiperiodic oscillations in control systems. \textit{Differ. Uravn. Protsessy Upravl.}, 2 (2019), in Russian

\bibitem{Anikushin2019SmithRed}
Anikushin~M.M. On the Smith reduction theorem for almost periodic ODEs satisfying the squeezing property. \textit{Rus. J. Nonlin. Dyn.}, \textbf{15}(1), 97--108 (2019)

\bibitem{ArovStaffans2006}
Arov~D.Z., Staffans~O.J. The infinite-dimensional continuous time Kalman-Yakubovich-Popov inequality. Operator Theory: Advances and Applications, Birkh\"{a}user Verlag Basel, Switzerland, 37--72 (2006)

\bibitem{BatkaiPiazzera2005}
B\'{a}tkai A., and Piazzera S. \textit{Semigroups for Delay Equations}. A K Peters, Wellesley (2005).

\bibitem{BoutetChueshovRezounenko1998}
Boutet~de~Monvel~L., Chueshov~I.D., Rezounenko~A.V. Inertial manifolds for retarded semilinear parabolic equations, \textit{Nonlinear Anal. Theory Methods Appl.}, \textbf{34}(6), 907--925 (1998)

\bibitem{BurnsHerdmanStech1983}
Burns~J.A., Herdman~T.L., Stech~H.W. Linear functional differential equations as semigroups on product spaces. \textit{SIAM J. Math. Anal.}, \textbf{14}(1), 98--116 (1983)

\bibitem{ChenShen2021}
Chen~S., Shen~J. Smooth inertial manifolds for neutral differential equations with small delays. \textit{J. Dyn. Diff. Equat.} (2021)

\bibitem{Chicone2004}
Chicone~C. Inertial flows, slow flows, and combinatorial identities for delay equations. \textit{J. Dyn. Diff. Equat.}, \textbf{16}(3), 805--831 (2004)

\bibitem{Chicone2003}
Chicone~C. Inertial and slow manifolds for delay equations with small delays. \textit{J. Differ. Equations}, \textbf{190}(2), 364--406 (2003)

\bibitem{Chueshov2015}
Chueshov~I. \textit{Dynamics of Quasi-stable Dissipative Systems}. Berlin: Springer (2015)

\bibitem{Driver1968SmallDelays}
Driver~R.D. On Ryabov's asymptotic characterization of the solutions of quasi-linear differential equations with small delays. \textit{SIAM Rev.}, \textbf{10}(3), 329--341 (1968)

\bibitem{EngelNagel2000}
Engel~K.-J., Nagel~R. \textit{One-Parameter Semigroups for Linear Evolution Equations}. Springer-Verlag (2000)

\bibitem{Gelig1978}
Gelig~A.Kh., Leonov~G.A., Yakubovich~V.A. \textit{Stability of Nonlinear Systems with Non-Unique Equilibrium State}. Nauka, Moscow (1978)

\bibitem{Hale1977}
Hale~J.K. \textit{Theory of Functional Differential Equations}. Springer-Verlag, New York (1977)

\bibitem{KokschSiegmund2002}
Koksch~N., Siegmund~S. Pullback attracting inertial manifolds for nonautonomous dynamical systems. \textit{J. Dyn. Differ. Equ.}, \textbf{14}(4), 889--941 (2002)

\bibitem{KostiankoZelikKwak}
Kostianko~A., Zelik~S. Kwak transform and inertial manifolds revisited. \textit{J. Dyn. Differ. Equ.}, 1--21 (2021)

\bibitem{KostiankoZelikSA2020}
Kostianko~A., Li~X., Sun~C., Zelik~S. Inertial manifolds via spatial averaging revisited. \textit{SIAM J. Math. Anal.}, \textbf{54}(1) (2022)

\bibitem{KostiankoZelik20181dII}
Kostianko~A. Zelik~S. Inertial manifolds for 1D reaction-diffusion-advection systems. Part II: periodic boundary conditions, \textit{Commun. Pure Appl. Anal.}, \textbf{17}(1), 285--317 (2018)

\bibitem{KostiankoZelik20171dI}
Kostianko~A. Zelik~S. Inertial manifolds for 1D reaction-diffusion-advection systems. Part I: Dirichlet and Neumann boundary conditions, \textit{Commun. Pure Appl. Anal.}, \textbf{16}(6), 2357--2376 (2017)

\bibitem{Krein1971}
Krein~S.G. \textit{Linear Differential Equations in Banach Space}, AMS, (1971)

\bibitem{KuzReit2020}
Kuznetsov~N.V., Reitmann~V. \textit{Attractor Dimension Estimates for Dynamical Systems: Theory and Computation}. Switzerland: Springer International Publishing AG (2020)

\bibitem{Likhtarnikov1977}
Likhtarnikov~A.L., Yakubovich~V.A. The frequency theorem for continuous one-parameter semigroups. \textit{Math. USSR-Izv.} \textbf{41}(4), 895--911 (1977), in Russian

\bibitem{Likhtarnikov1976}
Likhtarnikov~A.L., Yakubovich~V.A. The frequency theorem for equations of evolutionary type. \textit{Sib. Math. J.}, \textbf{17}(5), 790--803 (1976)

\bibitem{Lions1971OptCont}
Lions~J.-L. Optimal control of systems governed by partial differential equations. Vol. 170. Springer Verlag, 1971

\bibitem{LouisWexler1991}
Louis~J.-Cl., Wexler~D. The Hilbert space regulator problem and operator Riccati equation under stabilizability. \textit{Annales de la Soci\'{e}t\'{e} Scientifique de Bruxelles}, \textbf{105}(4), 137--165 (1991)

\bibitem{Miclavcic1991}
Miklav\v{c}i\v{c}~M. A sharp condition for existence of an inertial manifold. \textit{J. Dyn. Differ. Equ.}, \textbf{3}(3), 437--456 (1991)

\bibitem{Pandolfi1995}
Pandolfi~L. The standard regulator problem for systems with input delays. An approach through singular control theory. Applied Mathematics and Optimization, \textbf{31}(2), 119--136 (1995)

\bibitem{Pritchard1985}
Pritchard~A., Salamon~D. The linear-quadratic control problem for retarded systems with delays in control and observation. \textit{IMA J. Math. Control.}, \textbf{2}(4), 335--362 (1985)

\bibitem{Proskurnikov2015}
Proskurnikov~A.V. A new extension of the infinite-dimensional KYP lemma in the coercive case. \textit{IFAC-PapersOnLine}, \textbf{48}(1), 246--251 (2015)

\bibitem{Romanov1994}
Romanov~A.V. Sharp estimates of the dimension of inertial manifolds for nonlinear parabolic equations. \textit{Izvestiya: Mathematics}, \textbf{43}(1), 31--47 (1994)

\bibitem{Ryabov1967}
Ryabov~Yu.A. Asymptotic properties of solutions of weakly nonlinear systems with small delay. \textit{Trudy Sem. Teor. Differential. Uravnenii s Otklon. Argumentom Univ. Druzhby Narodov Patrisa Lumumby}, 5, 213--222 (1967), in Russian

\bibitem{Smith1992}
Smith~R.A. Poincar\'{e}-Bendixson theory for certain retarded functional-differential equations. \textit{Differ. Integral Equ.}, \textbf{5}(1), 213--240 (1992)

\bibitem{Smith1987OrbStab}
Smith~R.A. Orbital stability for ordinary differential equations. \textit{J. Differ. Equations}. \textbf{69}(2), 265--287 (1987)

\bibitem{Smith1980}
Smith~R.A. Existence of periodic orbits of autonomous retarded functional differential equations. \textit{Math. Proc. Cambridge Philos. Soc.}, \textbf{88}(1) (1980)

\bibitem{Suarez1988}
Suarez~M.J., Schopf~P.S. A delayed action oscillator for ENSO. \textit{J. Atmos. Sci.}, \textbf{45}(21), 3283--3287 (1988)

\bibitem{Zelik2014}
Zelik~S. Inertial manifolds and finite-dimensional reduction for dissipative PDEs. \textit{P. Roy. Soc. Edinb. A}, \textbf{144}(6), 1245--1327 (2014)

\bibitem{ZelikAttractors2022}
Zelik~S. Attractors. Then and now. \textit{arXiv preprint}, arXiv:2208.12101v1 (2022)
\end{thebibliography}

% Non-BibTeX users please use
\bibliographystyle{plain}

\end{document}